\documentclass[a4paper,11pt]{article}

\usepackage[style=alphabetic, maxnames = 6, maxalphanames = 6]{biblatex}

\RequirePackage{amsmath,amsthm,amssymb}
\usepackage{footnote}
\usepackage{cancel}
\usepackage{xcolor}
\usepackage{tikz}
\usepackage{pdflscape}
\usepackage{hyperref}

\newtheorem{thm}{Theorem}
\newtheorem{lem}[thm]{Lemma}
\newtheorem{prop}[thm]{Proposition}
\newtheorem{corol}[thm]{Corollary}

\theoremstyle{definition}
\newtheorem{conj}[thm]{Conjecture}

\theoremstyle{remark}
\newtheorem{rem}[thm]{Remark}

\newcommand{\I}{\mathcal{I}}
\newcommand{\W}{\mathcal{W}}
\newcommand{\Ustirling}[2]{\genfrac{[}{]}{0pt}{}{#1}{#2}}
\renewcommand{\min}{\mathbf{min}}
\renewcommand{\max}{\mathbf{max}}
\newcommand{\forb}{\mathbf{Forb}}
\newcommand{\vals}{\mathbf{Vals}}
\newcommand{\dist}{\mathbf{dist}}
\newcommand{\firstmax}{\mathbf{firstmax}}
\newcommand{\lastmax}{\mathbf{lastmax}}
\newcommand{\occ}{\mathbf{occ}}
\newcommand{\sites}{\mathbf{sites}}
\newcommand{\Sites}{\mathbf{Sites}}
\newcommand{\rep}{\mathbf{rep}}
\renewcommand{\top}{\mathbf{top}}
\newcommand{\inc}{\mathbf{inc}}
\newcommand{\des}{\mathbf{Des}}
\newcommand{\bounce}{\mathbf{bounce}}
\renewcommand{\sec}{\mathbf{sec}}
\newcommand{\dec}{\mathbf{dec}}

\renewcommand{\arraystretch}{1.3}

\title{Completing the enumeration of inversion sequences avoiding one or two patterns of length 3}

\author{Benjamin Testart}

\usepackage{a4wide}
\begin{document}
\renewcommand{\thefootnote}{\arabic{footnote}}
\begin{center}
{\LARGE Completing the enumeration of inversion sequences \\ avoiding one or two patterns of length 3} \bigbreak
{\Large Benjamin Testart} \bigbreak
{Universit\'e de Lorraine, CNRS, Inria, LORIA, F-54000 Nancy, France}
\end{center}

\abstract{
We present four constructions of inversion sequences, and use them to compute the enumeration sequences of 24 classes of pattern-avoiding inversion sequences. This completes the enumeration of inversion sequences avoiding one or two patterns of length 3. Some of our constructions are based on generating trees. Others involve pattern-avoiding words, which we also count using generating trees. To solve some of these cases, we introduce a generalization of inversion sequences, which we call \emph{shifted} inversion sequences. Lastly, we briefly discuss the asymptotics of pattern-avoiding inversion sequences, focusing on their exponential or super-exponential behavior.
} \medbreak


\section{Introduction}

\subsection{Basic definitions}
Let $\mathbb N$ be the set of natural numbers, including 0. Given a natural number $n \in \mathbb N$, we call \emph{integer sequences} of \emph{size} $n$ the elements of $\mathbb N^n$. We write the terms of an integer sequence $\sigma = (\sigma_1, \dots, \sigma_n) \in \mathbb N^n$. We denote by $\I_n$ the set of \emph{inversion sequences} \cite{Mansour_Shattuck_2015, Corteel_Martinez_Savage_Weselcouch_2016} of size $n$, that is the set of sequences $\sigma \in \mathbb N^n$ such that $\sigma_i < i$ for all $i \in \{1, \dots, n\}$.

There is a simple bijection between $\I_n$ and the set of permutations of $n$ elements, called the \emph{Lehmer code}, which explains the name ``inversion sequence". If $\pi$ is a permutation of the set $\{1, \dots, n\}$, the inversion sequence $\sigma \in \I_n$ associated with $\pi$ by the Lehmer code is defined by $\sigma_i = |\{j \; : \; \pi(j) > \pi(i) \text{ and } j < i \}|$ for all $i \in \{1, \dots, n\}$, i.e. $\sigma_i$ counts the number of \emph{inversions} of $\pi$ whose second entry is at position $i$.

Given two integer sequences $\sigma = (\sigma_1, \dots, \sigma_n) \in \mathbb N^n$ and $\rho = (\rho_1, \dots, \rho_k) \in \mathbb N^k$, we say that $\sigma$ \emph{contains} the \emph{pattern} $\rho$ if there exists a subsequence of $\sigma$ which is order-isomorphic to $\rho$. Such a subsequence is called an \emph{occurrence} of $\rho$. In this work, we exclusively study patterns of length 3, which we denote $\rho_1\rho_2\rho_3$ instead of $(\rho_1,\rho_2,\rho_3)$ for simplicity. For instance the sequence $(4,3,2,5,4)$ contains the pattern 021, since the subsequences $(3,5,4)$ and $(2,5,4)$ are both occurrences of 021. A sequence \emph{avoids} a pattern $\rho$ if it does not contain $\rho$, e.g. the inversion sequence $(0,0,2,3,2,0,1,5)$ avoids the pattern 101. If $P$ is a set of patterns, we denote by $\I_n(P)$ the set of inversion sequences of size $n$ avoiding all patterns in $P$.

\subsection{Context and summary of results}
The study of pattern-avoiding inversion sequences (and many more types of sequences avoiding patterns) branched from pattern-avoiding permutations, a well-established field of research in enumerative combinatorics, see e.g. \cite{Kitaev_2011} or \cite{Vatter_2015}.
Pattern-avoiding inversion sequences were first introduced in \cite{Mansour_Shattuck_2015} and \cite{Corteel_Martinez_Savage_Weselcouch_2016}, independently. 
Their study was continued in many articles, such as \cite{Martinez_Savage_2018, Beaton_Bouvel_Guerrini_Rinaldi_2019, Auli_Elizalde_2019_1, Auli_Elizalde_2019_2, Yan_Lin_2020, Lin_Yan_2020, Auli_Elizalde_2021, Mansour_Yıldırım_021, Pantone201+210, Kotsireas_Mansour_Yildirim_2024}, among others.

\begin{table}
\begin{center}
    \begin{tabular}{|c|c|c|c|}
    \hline
    Pattern $\rho$ & $|\I_n(\rho)|$ for $n = 1, \dots, 7$ & Solved in & \cite{OEIS} \\
    \hline
    000 & 1, 2, 5, 16, 61, 272, 1385 & \cite{Corteel_Martinez_Savage_Weselcouch_2016} & A000111\\
    001 & 1, 2, 4, 8, 16, 32, 64 & \cite{Corteel_Martinez_Savage_Weselcouch_2016} & A000079\\
    \textbf{010} & \textbf{1, 2, 5, 15, 53, 215, 979} & \textbf{Theorem \ref{thm010}} & \textbf{A263779}\\
    011 & 1, 2, 5, 15, 52, 203, 877 & \cite{Corteel_Martinez_Savage_Weselcouch_2016} &  A000110\\
    012 & 1, 2, 5, 13, 34, 89, 233 & \cite{Corteel_Martinez_Savage_Weselcouch_2016} and \cite{Mansour_Shattuck_2015} & A001519\\
    021 & 1, 2, 6, 22, 90, 394, 1806 & \cite{Corteel_Martinez_Savage_Weselcouch_2016} and \cite{Mansour_Shattuck_2015} & A006318\\
    100 & 1, 2, 6, 23, 106, 565, 3399 & \cite{Kotsireas_Mansour_Yildirim_2024} & A263780\\
    101 or 110 & 1, 2, 6, 23, 105, 549, 3207 & \cite{Corteel_Martinez_Savage_Weselcouch_2016} & A113227\\
    102 & 1, 2, 6, 22, 89, 381, 1694 & \cite{Mansour_Shattuck_2015} & A200753\\
    120 & 1, 2, 6, 23, 103, 515, 2803 & \cite{Mansour_Shattuck_2015} & A263778\\
    201 or 210 & 1, 2, 6, 24, 118, 674, 4306 & \cite{Mansour_Shattuck_2015} & A263777\\
    \hline
    \end{tabular}
    \caption{Enumeration sequences of inversion sequences avoiding a single pattern of length 3.}
    \label{table1}
\end{center}  \vspace{-10pt}
\end{table}

The enumeration of inversion sequences avoiding a single pattern of length 3 was already solved for all patterns except 010, see Table \ref{table1}.
A systematic study of inversion sequences avoiding pairs of patterns of length 3 was conducted by Yan and Lin in \cite{Yan_Lin_2020}, which left open the enumeration of inversion sequences avoiding 32 of the 78 total pairs. Since then, eight cases were solved in \cite{Kotsireas_Mansour_Yildirim_2024}, and one additional case was solved in \cite{Chen_Lin} and \cite{Pantone201+210} independently. Most of these cases were solved using bijections with other known combinatorial objects, or through generating trees.

In our work, we solve\footnote{We consider the enumeration of $P$-avoiding inversion sequences ``solved" once an algorithm is known to compute each number $|\I_n(P)|$ in polynomial time in $n$. Such an algorithm is sometimes called a ``Wilfian formula" after Herbert Wilf's paper \cite{Wilf_1982}. In our work, Wilfian formulas always take the form of a recurrence relation or an explicit expression.} the enumeration of inversion sequences avoiding 010, and the remaining 23 cases for pairs of patterns through four different constructions of pattern-avoiding inversion sequences, see Table \ref{table2}.

\begin{table}[ht]
\small
\begin{center}
    \setlength{\tabcolsep}{5pt}
    \begin{tabular}{|c|c|c|c|c|}
    \hline
    Pattern pair $P$ & $|\I_n(P)|$ for $n = 1, \dots, 7$ & Solved in & Performance\footnotemark & \cite{OEIS} \\
    \hline
    \{000, 100\} & 1, 2, 5, 16, 60, 260, 1267 & Theorem \ref{thm000+100} & 990 & A279564 \\
    \{102, 201\} & 1, 2, 6, 22, 87, 354, 1465 & Theorems \ref{thm102+201},\ref{thmGF102+201} & 6100 & A279566 \\
    \hline
    \{000, 102\} & 1, 2, 5, 14, 40, 121, 373 & Theorem \ref{thm000+102} & 5800 & A374541 \\
    \{102, 210\} & 1, 2, 6, 22, 87, 351, 1416 & Theorems \ref{thm102+210},\ref{thmGF102+210}& 5600 & A374542 \\
    \{000, 201\}\,or\,\{000, 210\} & 1, 2, 5, 16, 60, 257, 1218 & Theorem \ref{thm000+201} & 735 & A374543 \\
    \{100, 110\} & 1, 2, 6, 22, 93, 437, 2233 & Theorem \ref{thm100+110} & 820 & A374544 \\
    \{100, 101\} & 1, 2, 6, 22, 93, 439, 2267 & Theorem \ref{thm100+101} & 815 & A374545 \\
    \{110, 201\} & 1, 2, 6, 23, 103, 512, 2739 & Theorem \ref{thm110+201} & 825 & A374546 \\
    \{101, 210\} & 1, 2, 6, 23, 103, 513, 2763 & Theorem \ref{thm101+210} & 810 & A374547\\
    \hline
    \{011, 120\} & 1, 2, 5, 14, 42, 132, 431 & Theorem \ref{thm011+120} & 430 & A374548 \\
    \{100, 120\} & 1, 2, 6, 22, 92, 421, 2062 & Theorem \ref{thm100+120} & 350 & A374549 \\
    \{120, 201\} & 1, 2, 6, 23, 102, 498, 2607 & Theorem \ref{thm120+201} & 340 & A374550 \\
    \{110, 120\} & 1, 2, 6, 22, 92, 423, 2091 & Theorem \ref{thm110+120} & 330 &  A279570 \\
    \{010, 120\} & 1, 2, 5, 15, 52, 201, 845 & Theorem \ref{thm010+120} & 330 & A279559 \\
    \{101, 120\} & 1, 2, 6, 22, 90, 397, 1859 & Theorem \ref{thm101+120} & 240 & A374551 \\
    \{000, 120\} & 1, 2, 5, 15, 50, 185, 737 & Theorem \ref{thm000+120} & 355 & A374552 \\
    \{000, 010\} & 1, 2, 4, 10, 29, 95, 345 & Theorem \ref{thm000+010} & 235 & A279552 \\
    \{010, 201\}\,or\,\{010, 210\} & 1, 2, 5, 15, 53, 214, 958 & Theorem \ref{thm010+210} & 185 & A360052 \\
    \{010, 110\} & 1, 2, 5, 15, 52, 201, 847 & Theorem \ref{thm010+110} & 145 & A359191 \\
    \hline
    \{010, 102\} & 1, 2, 5, 15, 51, 186, 707 & Theorem \ref{thm010+102} & 265 & A374553 \\
    \{100, 102\} & 1, 2, 6, 21, 80, 318, 1305 & Theorem \ref{thm100+102} & 380 & A374554 \\
    \hline
    \end{tabular} 
\caption{Enumeration sequences of inversion sequences avoiding pairs of patterns studied in this article. Each section of the table corresponds to a different construction.}
\label{table2}
\end{center}
\end{table}

Each construction is presented in a different section. The construction of Section \ref{sectionGT} simply consists in inserting each entry of a sequence from left to right. In Section \ref{sectionMax}, we construct sequences by inserting their entries in increasing order of value (the order of insertion of entries having the same value varies according to the patterns considered). Section \ref{sectionSplitMax} is an improved and more complete version of a previous (unpublished) work \cite{Testart2022}. It relies on a decomposition of inversion sequences around their first maximum; this means sequences are obtained by concatenating two smaller sequences. Section \ref{sectionShift} introduces shifted inversion sequences and uses a decomposition around their first minimum, similar to that of Section \ref{sectionSplitMax}. It is sometimes easier to construct pattern-avoiding inversion sequences by seeing them as a particular case of shifted inversion sequences.

In Section \ref{sectionGF}, we present conjectures about the algebraicity of the generating functions of several classes of inversion sequences avoiding pairs of patterns, and prove two of those conjectures. We conclude with a brief discussion on the asymptotic behavior of the number of pattern-avoiding inversion sequences in Section \ref{sectionAsymptotics}. In particular, we give sufficient conditions on a set of patterns $P$ to show that the growth of the enumeration sequence of $P$-avoiding inversion sequences is bounded above by an exponential function, or to show that it is super-exponential.

\subsection{Notation and preliminaries}

\subsubsection*{Notation}
For any integers $a,b \in \mathbb Z$, we denote the integer interval $[a,b] = \{k \in \mathbb Z \; : \; a \leqslant k \leqslant b\}$.
We denote by $\delta_{a,b} = \begin{cases}
1 & \text{if} \quad a = b \\
0 & \text{if} \quad a \neq b
\end{cases} \;$ the Kronecker delta function.
For all $n \in \mathbb N$, we denote by $C_n = \frac{1}{n+1}\binom{2n}{n}$ the Catalan numbers.
We denote by $\varepsilon$ the \emph{empty sequence}, that is the only sequence of size $0$.

Let $n,m \in \mathbb N$, and let $\sigma \in \mathbb N^n, \tau \in \mathbb N^m$ be two integer sequences.
We define the \emph{concatenation} of $\sigma$ and $\tau$ as $\sigma \cdot \tau = (\sigma_1, \dots, \sigma_n, \tau_1, \dots, \tau_m) \in \mathbb N^{n+m}$. For all $k \in \mathbb N$, we denote by $\sigma^k$ the concatenation of $k$ copies of $\sigma$ (in particular, $\sigma^0 = \varepsilon$).

For any $k \in \mathbb N$, we denote by $\sigma + k = (\sigma_i + k)_{i \in [1,n]}$ the sequence obtained by adding $k$ to each term of $\sigma$.

\footnotetext{Approximate number of terms of each enumeration sequence we were able to compute in 1 minute, with C++ programs running on a personal computer. Naive methods can compute around 20 terms at most.}

\subsubsection*{Terminology}
A \emph{combinatorial class} is a set of objects $\mathcal C$ together with a size function $|\cdot| : \mathcal C \to \mathbb N$ such that there is a finite number of objects of size $n$ for each $n \in \mathbb N$. For any set of patterns $P$, the set of all $P$-avoiding inversion sequences $\I(P) = \coprod_{n \geqslant 0} \I_n(P)$ is a combinatorial class.

If $\mathcal C$ is a combinatorial class, and $\mathcal C_n$ its subset of objects of size $n$ for each $n \in \mathbb N$, we call $(|\mathcal C_n|)_{n \in \mathbb N}$ the \emph{enumeration sequence} of $\mathcal C$ (here, the vertical bars are used to denote set cardinality).
The (ordinary) \emph{generating function} of the combinatorial class $\mathcal C = \coprod_{n \in \mathbb N} \mathcal C_n$ is the formal power series $\sum_{n \in \mathbb N} |\mathcal C_n| x^n$ in the indeterminate $x$.

Given two integer sequences $\sigma \in \mathbb N^n$ and $\tau \in \mathbb N^m$, we say that $\tau$ is a \emph{factor} of $\sigma$ if $\tau$ is a subsequence of consecutive terms of $\sigma$, i.e. if there exists two integers $a \leqslant b \in [1,n]$ such that $\tau = (\sigma_i)_{i \in [a,b]}$, or $\tau = \varepsilon$.

\subsubsection*{Statistics}
For all $n \in \mathbb N$ and $\sigma \in \mathbb N^n$, let
\begin{itemize}
    \item $\vals(\sigma) = \{\sigma_i \; : \; i \in [1,n]\}$ be the \emph{set of values} of $\sigma$,
    \item $\min(\sigma) = \min(\vals(\sigma))$ be the \emph{minimum} of $\sigma$, with the convention $\min(\varepsilon) = +\infty$,
    \item $\max(\sigma) = \max(\vals(\sigma))$ be the \emph{maximum} of $\sigma$, with the convention $\max(\varepsilon) = -1$,
    \item $\dist(\sigma) = |\vals(\sigma)|$ be the number of \emph{distinct values} of $\sigma$,
    \item $|\sigma| = n$ be the \emph{size} or \emph{length} of $\sigma$,
    \item $\firstmax(\sigma) = \min(i \in [1,n] \; : \; \sigma_i = \max(\sigma))$ be the position of the \emph{first maximum} of $\sigma$, with the convention $\firstmax(\varepsilon) = 0$,
    \item $\lastmax(\sigma) = \max(i \in [1,n] \; : \; \sigma_i = \max(\sigma))$ be the position of the \emph{last maximum} of $\sigma$, with the convention $\lastmax(\varepsilon) = 0$.
\end{itemize}

\subsubsection*{Words}
For all $n,k \in \mathbb N$, we denote by $\W_{n,k} = [0,k-1]^n$ the set of words of length $n$ over the alphabet $[0,k-1]$, and by $\overline \W_{n,k} = \{\omega \in \W_{n,k} \; : \; \dist(\omega) = k\}$ the subset of words in which each letter of the alphabet $[0,k-1]$ appears at least once\footnote{Equivalently, $\W_{n,k}$ is the set of maps $[1,n] \to [0,k-1]$, and $\overline \W_{n,k}$ is the subset of surjective maps. We remark that the set of surjective maps $[1,n] \to [1,k]$ is also known as the set of \emph{Cayley permutations} \cite{Mor_Frankel_1984} of length $n$ and maximum $k$.}. For any set of patterns $P$, we denote by $\W_{n,k}(P)$ and $\overline \W_{n,k}(P)$ the subsets of $P$-avoiding words of $\W_{n,k}$ and $\overline \W_{n,k}$. The following proposition shows that it is essentially equivalent to solve the enumeration of $\W_{n,k}(P)$ or that of $\overline \W_{n,k}(P)$.

\begin{prop} \label{propbinomwords}
    For any set of patterns $P$, for all $n,k \in \mathbb N$,
    $$|\W_{n,k}(P)| = \sum_{d=0}^{\min(n,k)} \binom{k}{d} |\overline \W_{n,d}(P)|.$$
\end{prop}
\begin{proof}
Let $n,k \geqslant 0$, $d \in [0, \min(n,k)]$, and let $\mathfrak W_{n,k,d} = \{ \omega \in \W_{n,k} \; : \; d = \dist(\omega)\}$ be the subset of $\W_{n,k}$ of words containing exactly $d$ distinct letters.
Let $F_{d,k}$ be the set of strictly increasing functions from $[0,d-1]$ to $[0,k-1]$. It is easy to see that the function $F_{d,k} \times \overline \W_{n,d} \to \mathfrak W_{n,k,d}$, $(\varphi, \omega) \mapsto (\varphi(\omega_i))_{i \in [1,n]}$ is a bijection. For any set of patterns $P$, restricting this function to the domain $F_{d,k} \times \overline \W_{n,d}(P)$ yields a bijection with the set of $P$-avoiding words in $\mathfrak W_{n,k,d}$ (the image of $(\varphi, \omega)$ is order-isomorphic to $\omega$, so it contains the same patterns).
This concludes the proof, since $\W_{n,k} = \coprod_{d=0}^{\min(n,k)} \mathfrak W_{n,k,d}$, and $|F_{d,k}| = \binom{k}{d}$.
\end{proof}

\section{Generating trees growing on the right} \label{sectionGT}

\subsection{Method} \label{trees}
A \emph{generating tree} is a rooted, labelled tree such that the label of each node determines its number of children, and their labels. Generating trees were first introduced in the context of pattern-avoiding permutations in \cite{West1, West2}. We call \emph{combinatorial generating tree} a generating tree labelled by the objects of a combinatorial class and such that each object of size $n$ labels exactly one node, at level $n$ (with the convention that the root is at level 0).
A combinatorial generating tree can be defined using a map, called \emph{ECO operator} \cite{ECO}, which constructs each object of size $n+1$ from some object of size $n$. More generally, a generating tree can always be defined by a \emph{succession rule}
$$\Omega = \begin{cases} a \\
\ell \leadsto \ell_1 \dots \ell_{c(\ell)}
\end{cases}$$
composed of an \emph{axiom} $a$, which labels the root of the tree, and a \emph{production} which associates to each label $\ell$ the labels of the children of any node labelled $\ell$ in the tree. In the example above, we denoted $c(\ell)$ the number of children of nodes labelled $\ell$, and $\ell_1, \dots, \ell_{c(\ell)}$ the labels\footnote{The same label may appear on several children of $\ell$. Formally, the production maps each label to a multiset of labels.} of those children. A simple combinatorial generating tree for the class of inversion sequences is defined by the succession rule
$$\Omega_{\text{inv}} = \begin{cases} \varepsilon \\
\sigma \leadsto \sigma \cdot i \quad \text{for} \quad i \in [0,|\sigma|].
\end{cases}$$

For any set of patterns $P$, a combinatorial generating tree for the class of $P$-avoiding inversion sequences $\I(P)$ can be obtained by restricting the above to only accept values of $i$ such that $\sigma \cdot i$ avoids the patterns in $P$. This does define a generating tree: indeed, if $\sigma \cdot i$ avoids $P$, then $\sigma$ must also avoid $P$. We call this tree the generating tree \emph{growing on the right} for inversion sequences avoiding the patterns $P$. The present section is dedicated to such generating trees, which are one of the simplest and most common construction for pattern-avoiding inversion sequences.
In all of Section \ref{sectionGT}, we call $i$ a \emph{forbidden} value for $(\sigma, P)$ (or simply for $\sigma$, when $P$ is implicit) if $\sigma \cdot i$ contains a pattern in $P$. \medbreak

Most generating trees used in the literature to solve enumeration problems are not combinatorial generating trees, but instead use labels that are much simpler (typically, integers or tuples of integers). We define a generating tree to be \emph{concise} if it does not contain two isomorphic subtrees rooted in nodes having different labels. Informally, a generating tree is concise if it involves the minimal number of labels required to describe the ``shape" of the tree.

We say that two nodes of a tree $\mathcal T$ are \emph{$\mathcal T$-equivalent} if they are roots of isomorphic subtrees of $\mathcal T$. In particular, two nodes of $\mathcal T$ which have the same label are always $\mathcal T$-equivalent: the label of a node determines its number of children and their labels, hence a label also determines the entire subtree rooted in its node, by induction. Note that $\mathcal T$ is concise if and only if two $\mathcal T$-equivalent nodes always have the same label. We now explain how to go (in favorable cases) from a combinatorial generating tree to a concise generating tree.

Assume we have a combinatorial generating tree $\mathcal T$ for a combinatorial class $\mathcal C$ and a statistic $s : \mathcal C \to L$ for some set $L$, such that for each $\sigma \in \mathcal C$, the value $s(\sigma)$ determines the number of children of the node labelled $\sigma$ in $\mathcal T$, and the value of $s$ when applied to each child. Formally, this means there exists a function $f : L \to \mathbb N^L$ (where $\mathbb N^L$ is the set of multisets of elements in $L$) such that for all $\sigma \in \mathcal C$, $f(s(\sigma))$ is the finite multiset $s(\tau_1), \dots, s(\tau_q)$, where $\tau_1, \dots, \tau_q$ are the labels of the children of the node labelled $\sigma$ in $\mathcal T$.
Replacing each label $\sigma$ by the value $s(\sigma)$ yields a generating tree $\mathcal T'$ which is isomorphic to $\mathcal T$. Notice that $\mathcal T'$ is defined by a succession rule which only involves the values of $s$ (the axiom of this rule is the image under $s$ of the label of the root of $\mathcal T$, and its production is the function we denoted $f$).
On our previous example, we can use the statistic ``size" to turn $\Omega_{\text{inv}}$ into a simpler succession rule
$$\Omega_{\text{factorial}} = \begin{cases} (0) \\
(n) \leadsto (n+1)^{n+1}
\end{cases}$$
where the production means that each node labelled $(n)$ has $n+1$ children, each labelled $(n+1)$. From this succession rule, we can easily see that there are $n!$ nodes at level $n$.

By keeping only the value $s(\sigma)$ rather than the ``complete" object $\sigma \in \mathcal C$ as a label, we can retain a lower amount of information about objects, which is still sufficient to describe the tree (up to isomorphism). We say that two objects $\sigma, \tau \in \mathcal C$ such that $s(\sigma) = s(\tau)$ are \emph{$s$-equivalent}.
Each label $\ell \in L$ of $\mathcal T'$ corresponds to the \emph{$s$-equivalence class} $s^{-1}(\ell) \subseteq \mathcal C$. \medbreak

Since $s$-equivalent objects of $\mathcal C$ are always labels of $\mathcal T$-equivalent nodes, the coarsest $s$-equivalence relation is obtained by defining $s$ to be the statistic which maps each object of $\mathcal C$ to the $\mathcal T$-equivalence class of the corresponding node. 
In particular, for any generating tree $\mathcal T$, there exists a concise generating tree $\mathcal T'$ isomorphic to $\mathcal T$, obtained by replacing the label of each node of $\mathcal T$ by its $\mathcal T$-equivalence class.

In practice, finding this tree $\mathcal T'$ requires some way of knowing whether two nodes are $\mathcal T$-equivalent, which is not always obvious.
For slightly different definitions\footnote{In \cite{Kotsireas_Mansour_Yildirim_2024}, generating trees are defined as plane trees, and two nodes are equivalent if they are roots of isomorphic plane subtrees. This relation relies on an order on the children of each node, and it is finer than the $\mathcal T$-equivalence we defined.} of generating trees and equivalence, \cite{Kotsireas_Mansour_Yildirim_2024} presents an algorithm which can test whether two nodes are equivalent in finite time, for any combinatorial generating tree growing on the right for inversion sequences avoiding a finite set of patterns. This algorithm then labels each node of the tree by an inversion sequence (the minimal sequence in lexicographic order) which corresponds to a node in the same equivalence class. \medbreak

Classically, generating trees are used to solve enumeration problems because they induce recurrence relations on the corresponding enumeration sequences (which may also be turned into equations satisfied by their generating function). Here we formalize how a succession rule can be turned into such a recurrence relation. Assume we have a combinatorial generating tree $\mathcal T$ for a combinatorial class $\mathcal C$, a generating tree $\mathcal T'$ labelled by elements in some set $L$, and a function $s : \mathcal C \to L$ such that replacing each label $\sigma$ of $\mathcal T$ by its image $s(\sigma)$ yields $\mathcal T'$. For all $(n,\ell) \in \mathbb N \times L$, let $\mathfrak c_{n,\ell} = |\{\sigma \in \mathcal C \; : \; |\sigma| = n, \, s(\sigma) = \ell\}|$, so that the number of objects of size $n$ in $\mathcal C$ is $\sum_\ell \mathfrak c_{n,\ell}$. Let $f : L \to \mathbb N^L$ be the production of $\mathcal T'$. For all $n \geqslant 1$ and $\ell \in L$, we have
$$\mathfrak c_{n, \ell} = \sum_{k \in L} f(k)(\ell) \cdot \mathfrak c_{n-1,k},$$
where $f(k)(\ell)$ counts the multiplicity of the label $\ell$ among the children of a node labelled $k$ in $\mathcal T'$. \medbreak

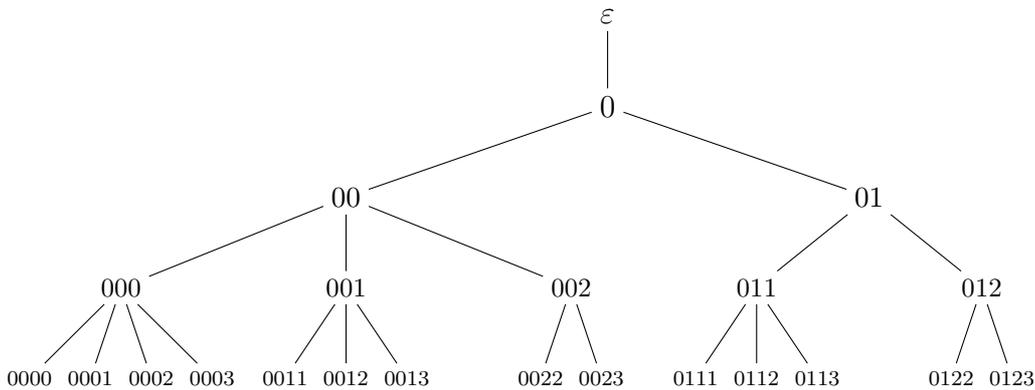
\begin{figure}[ht]
\centering
\begin{scriptsize}
\begin{tikzpicture}[level distance=15mm]
\begin{scope}[scale=0.8]
\tikzstyle{level 2}=[sibling distance=86mm]
\tikzstyle{level 3}=[sibling distance=37mm]
\tikzstyle{level 4}=[sibling distance=10mm]
\node{\large $\varepsilon$}
 child {node{\large $0$}
        child {node {\normalsize $00$}
         child {node {\small $000$}
            child {node {$0000$}
            }
            child {node {$0001$}
            }
            child {node {$0002$}
            }
            child {node {$0003$}
            }
         }
         child {node {\small $001$}
            child {node {$0011$}
            }
            child {node {$0012$}
            }
            child {node {$0013$}
            }
         }
         child {node {\small $002$}
            child {node {$0022$}
            }
            child {node {$0023$}
            }
         }
       }
       child {node {\normalsize $01$}
         child {node {\small $011$}
            child {node {$0111$}
            }
            child {node {$0112$}
            }
            child {node {$0113$}
            }
       }
         child {node {\small $012$}
            child {node {$0122$}
            }
            child {node {$0123$}
            }
       }
    }
};
\end{scope}
\end{tikzpicture}
\end{scriptsize}
\caption{First five levels of the combinatorial generating tree growing on the right for $\I(10)$.}  
\label{CombiGT10}
\end{figure}

We end this introduction with a simple example. The following succession rule describes the generating tree growing on the right for inversion sequences avoiding the pattern 10 (i.e. nondecreasing inversion sequences), represented in Figure \ref{CombiGT10}.
$$\Omega_{10} = \begin{cases} \varepsilon \\
\sigma \leadsto \sigma \cdot i \quad \text{for} \quad i \in [\max(\sigma),|\sigma|].
\end{cases}$$

For each $\sigma \in \I(10)$, let $s(\sigma)$ be the number of children of the node labelled $\sigma$ in this tree. It can be shown that $s(\varepsilon) = 1$, and $s(\sigma) = 1 + |\sigma| - \max(\sigma)$ if $\sigma$ is nonempty. By replacing each label $\sigma$ by the value $s(\sigma)$, we obtain an isomorphic generating tree, represented in Figure \ref{ConciseGT10}, and described by the succession rule
$$\Omega_{\text{Cat}} = \begin{cases} (1) \\
(k) \leadsto (i) \quad \text{for} \quad i \in [2,k+1].
\end{cases}$$
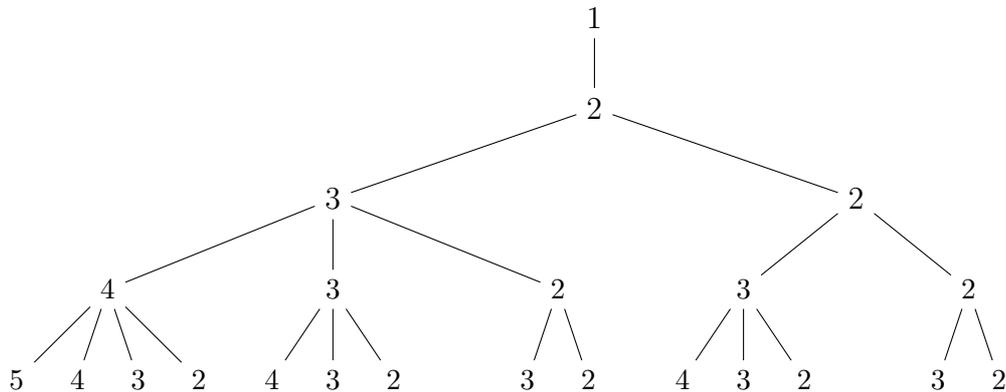
\begin{figure}[ht]
\centering
\begin{tikzpicture}[level distance=15mm]
\begin{scope}[scale=0.8]
\tikzstyle{level 2}=[sibling distance=86mm]
\tikzstyle{level 3}=[sibling distance=37mm]
\tikzstyle{level 4}=[sibling distance=10mm]
\node{\large $1$}
 child {node{\large $2$}
        child {node {\large $3$}
         child {node {$4$}
            child {node {$5$}
            }
            child {node {$4$}
            }
            child {node {$3$}
            }
            child {node {$2$}
            }
         }
         child {node {$3$}
            child {node {$4$}
            }
            child {node {$3$}
            }
            child {node {$2$}
            }
         }
         child {node {$2$}
            child {node {$3$}
            }
            child {node {$2$}
            }
         }
       }
       child {node {\large $2$}
         child {node {$3$}
            child {node {$4$}
            }
            child {node {$3$}
            }
            child {node {$2$}
            }
       }
         child {node {$2$}
            child {node {$3$}
            }
            child {node {$2$}
            }
       }
    }
};
\end{scope}
\end{tikzpicture}
\caption{First five levels of the generating tree described by the succession rule $\Omega_{\text{Cat}}$.}
\label{ConciseGT10}
\end{figure}

It is known \cite{West1} that the succession rule $\Omega_{\text{Cat}}$ describes a tree in which the number of nodes at each level $n$ is the Catalan number $C_n = \frac{1}{n+1}\binom{2n}{n}$.
It is easy to see that this succession rule describes a concise generating tree: the label of each node counts its number of children, so if two nodes are roots of isomorphic subtrees, they must have the same label.

From now on, we allow ourselves to represent a combinatorial generating tree and an isomorphic concise generating tree on the same figure, by placing two labels on each node.

\subsection{The pair \{000, 100\}}
\begin{thm} \label{thm000+100}
The enumeration of inversion sequences avoiding the patterns 000 and 100 is given by the succession rule
$$\Omega_{\{000,100\}} = \begin{cases} (1,0) \\
(a,b) \leadsto \begin{aligned}[t] & (a+1, b-1)^b \\
& (a+1-j, b+j) \quad \text{for} \quad j \in [1,a]. \end{aligned}
\end{cases}$$
\end{thm}
\begin{proof}
    For all $n \in \mathbb N$ and $\sigma \in \I_n(000,100)$, let
    \begin{itemize}
        \item $A(\sigma) = \{i > \max(\sigma) \; : \; \sigma \cdot i \in \I(000,100)\}$ be the set of values which can be inserted at the end of $\sigma$ ``strictly above" its maximum,
        \item $B(\sigma) = \{i \leqslant \max(\sigma) \; : \; \sigma \cdot i \in \I(000,100)\}$ be the set of values which can be inserted at the end of $\sigma$ ``weakly below" its maximum.
    \end{itemize}
    
    Inserting any value greater than $\max(\sigma)$ at the end of $\sigma$ cannot create an occurrence of the patterns 000 or 100. As a result, $A(\sigma) = [\max(\sigma)+1, n]$. Let $a(\sigma) = |A(\sigma)|$, and $b(\sigma) = |B(\sigma)|$. In particular, $(a(\varepsilon), b(\varepsilon)) = (1,0)$.
    
    We show that the combinatorial generating tree growing on the right for $\I(000,100)$ is isomorphic to the tree described by $\Omega_{\{000,100\}}$. This isomorphism relabels each node $\sigma$ by $(a(\sigma), b(\sigma))$.
    
    Let $n \in \mathbb N$, $\sigma \in \I_n(000,100)$, $i \in A(\sigma) \sqcup B(\sigma)$, and $\sigma' = \sigma \cdot i$.
    \begin{itemize}
        \item If $i \in B(\sigma)$, then $A(\sigma') = A(\sigma) \sqcup \{n+1\}$ and $B(\sigma') = B(\sigma) \backslash \{i\}$, since if $i < \max(\sigma)$, then $i$ becomes forbidden to avoid 100, and if $i = \max(\sigma)$, then $i$ becomes forbidden to avoid 000. In this case, $\sigma'$ satisfies $(a(\sigma'), b(\sigma')) = (a(\sigma)+1, b(\sigma)-1)$.
        \item If $i \in A(\sigma)$, then no additional values are forbidden, and the values $[\max(\sigma)+1, i]$ are now less than or equal to $\max(\sigma')$. Let $j = i - \max(\sigma)$. In this case, $\sigma'$ satisfies $(a(\sigma'), b(\sigma')) = (a(\sigma)+1-j, b(\sigma)+j)$. As $i$ ranges over $A(\sigma)$, $j$ varies from $1$ to $a(\sigma)$. \qedhere
    \end{itemize}
\end{proof}

\subsection{The pair \{102, 201\}} \label{102+201}
In this section, we present a generating tree construction to enumerate $\I(102,201)$. We later use this construction to find the generating function of $\I(102,201)$ in Theorem \ref{thmGF102+201}.

A core idea in our construction of inversion sequences avoiding 102 and 201 is that we may restrict the ``future" of combinatorial objects in a generating tree construction; we discuss this further in Remark \ref{remPhantom}. This idea comes from Pantone, who calls this a ``commitment", see \cite{Pantone201+210}.

We begin by establishing some general properties of integer sequences avoiding 102 and 201. For any $\sigma \in \mathbb N^n$ such that $\dist(\sigma) \geqslant 2$, let $\sec(\sigma) = \max(\vals(\sigma) \backslash \{\max(\sigma)\})$ be the second largest value of $\sigma$.

\begin{prop} \label{prop102+201}
    Let $\sigma$ be an integer sequence of size $n$ such that $\dist(\sigma) \geqslant 2$, let ${m = \max(\sigma)}$ and $s = \sec(\sigma)$. Then $\sigma$ avoids 102 and 201 if and only if all three following conditions are satisfied:
    \begin{enumerate}
        \item $(\sigma_i)_{i \in [1, \firstmax(\sigma)]}$ is nondecreasing,
        \item $(\sigma_i)_{i \in [\lastmax(\sigma),n]}$ is nonincreasing,
        \item $\forall i \in [\firstmax(\sigma),\lastmax(\sigma)], \, \sigma_i \in \{m,s\}$.
    \end{enumerate}
\end{prop}
\begin{proof}
    First, we show that each condition is necessary:
    \begin{enumerate}
        \item If $i < j \in [1,\firstmax(\sigma)]$ satisfy $\sigma_i > \sigma_j$, then $\sigma_j < m$. In particular, we have $j \neq \firstmax(\sigma)$ hence $(\sigma_i, \sigma_j, m)$ is an occurrence of 102.
        \item If $i < j \in [\lastmax(\sigma), n]$ satisfy $\sigma_i < \sigma_j$, then $\sigma_i < m$. In particular, we have $i \neq \lastmax(\sigma)$ hence $(m, \sigma_i, \sigma_j)$ is an occurrence of 201.
        \item If $i \in [\firstmax(\sigma),\lastmax(\sigma)]$ satisfies $\sigma_i < s$, let $j$ be an integer such that $\sigma_j = s$. If $j < i$, then $(\sigma_j, \sigma_i, m)$ is an occurrence of 102. If $i < j$, then $(m, \sigma_i, \sigma_j)$ is an occurrence of 201.
    \end{enumerate}
    
    Assume, for the sake of contradiction, that $\sigma$ satisfies all three conditions, and that $(\sigma_i, \sigma_j, \sigma_k)$ is an occurrence of 102 for some $i < j < k \in [1,n]$. Since $\sigma_i > \sigma_j$, condition 1 implies that $j > \firstmax(\sigma)$. Since $\sigma_j < \sigma_k$, condition 2 implies that $j < \lastmax(\sigma)$. Hence by condition 3, $\sigma_j \geqslant s$, which implies that $(\sigma_i, \sigma_j) = (m,s)$ because $\sigma_i > \sigma_j$. Finally $\sigma_k > \sigma_i = m$, a contradiction. The same reasoning holds for the pattern 201, by the reverse symmetry.
\end{proof} \noindent
Proposition \ref{prop102+201} has two immediate corollaries.
\begin{corol} \label{corol102+201}
     Let $\sigma$ be a \{102, 201\}-avoiding integer sequence such that $m = \max(\sigma)$, and $s = \sec(\sigma)$. Any occurrence of the pattern 101 in $\sigma$ is a subsequence $(m,s,m)$.
\end{corol}
\begin{corol} \label{corolunimod}
    An integer sequence avoids \{101, 102, 201\} if and only if it is unimodal.
\end{corol}

We first study inversion sequences avoiding \{101, 102, 201\}, then those which avoid \{102, 201\} and contain the pattern 101.
The enumeration of inversion sequences avoiding 101, 102 and 201 was already solved and their generating function is given in \cite{Callan_Mansour}\footnote{This article employs the same generating tree construction (that is, inserting each entry of a sequence from left to right). However, it does not present an explicit description of the labels in terms of statistics of inversion sequences, since the algorithmic approach from \cite{Kotsireas_Mansour_Yildirim_2024} was used to derive the succession rule.}. We nevertheless describe a generating tree for this class, for the sake of completeness, and in order to better introduce the more difficult case of sequences containing 101.

Let $\mathfrak A_{n,m} = \{\sigma \in \I_n(10) \; : \; m = \max(\sigma)\}$ be the set of nondecreasing inversion sequences of size $n$ and maximum $m$ for $n \geqslant 1$, and let $\mathfrak A_{0,0} = \{\varepsilon\}$ by convention. Let $\mathfrak C^{(1)}_{n,\ell} = \{\sigma \in \I_n(101,102,201) \; : \; \ell = \sigma_n < \max(\sigma) \}$ be the set of $\{101,102,201\}$-avoiding inversion sequences of size $n$, last value $\ell$, and such that $\ell$ is not the maximum.  Let $\mathfrak A = \coprod_{n,m \geqslant 0} \mathfrak A_{n,m}$, and $\mathfrak C^{(1)} = \coprod_{n > \ell \geqslant 0} \mathfrak C^{(1)}_{n,\ell}$. From Corollary \ref{corolunimod}, $\mathfrak C^{(1)}$ is the set of unimodal inversion sequences which decrease after reaching their maximum, and we have $\I(101,102,201) = \mathfrak A \sqcup \mathfrak C^{(1)}$. Let $\mathfrak a_{n,m} = |\mathfrak A_{n,m}|$, and $\mathfrak c^{(1)}_{n,\ell} = |\mathfrak C^{(1)}_{n,\ell}|$.

\begin{lem} \label{lemunimod}
For all $n > m \geqslant 0$,
$$\mathfrak a_{n,m} = \frac{\binom{n+m-1}{m} (n-m)}{n}.$$
For all $n \geqslant 3, \ell \geqslant 0$,
$$\mathfrak c^{(1)}_{n,\ell} = \sum_{i = \ell}^{n-2} \mathfrak c^{(1)}_{n-1,i} + \sum_{i = \ell+1}^{n-2}  \mathfrak a_{n-1,i}.$$
\end{lem}
\begin{proof}
    Let $n \geqslant 1$, $m \in [0,n-1]$, and $\sigma \in \mathfrak A_{n,m}$. By definition of $\mathfrak A_{n,m}$, we have $\sigma_n = m$. Removing the last entry $\sigma_n$ from $\sigma$ yields a sequence in $\mathfrak A_{n-1,i}$ for some $i \in [0,m]$. Conversely, $\sigma$ can be obtained by appending the value $m$ to some sequence in $\mathfrak A_{n-1,i}$. Hence, there is a bijection between the sets $\mathfrak A_{n,m}$ and $\coprod_{i = 0}^m \mathfrak A_{n-1,i}$. This yields the equation $\mathfrak a_{n,m} = \sum_{i = 0}^m \mathfrak a_{n-1,i}$,
    from which we can easily prove that $\mathfrak a_{n,m} = \frac{\binom{n+m-1}{m} (n-m)}{n}$ for all $n > m \geqslant 0$. This corresponds to Catalan's Triangle (OEIS A009766).

    Let $n \geqslant 3, \ell \geqslant 0$, $\sigma \in \mathfrak C^{(1)}_{n,\ell}$, and $i = \sigma_{n-1} \geqslant \ell$. Let $\sigma' = (\sigma_j)_{j \in [1,n-1]}$ be the sequence obtained by removing the last entry from $\sigma$. If $\sigma'$ is nondecreasing, then $i = \max(\sigma) > \ell$, and $\sigma' \in \mathfrak A_{n-1,i}$. Otherwise, $\sigma' \in \mathfrak C^{(1)}_{n,\ell}$. As in the previous case, this is a bijection between $\mathfrak C^{(1)}_{n, \ell}$ and $\left (\coprod_{i = \ell}^{n-2} \mathfrak C^{(1)}_{n-1,i} \right ) \sqcup \left (\coprod_{i = \ell+1}^{n-2} \mathfrak A_{n-1,i} \right )$.
\end{proof}
\begin{rem}
    The bijections from the proof of Lemma \ref{lemunimod} can be turned into  a succession rule, by labelling $(n,m)$ each sequence in $\mathfrak A_{n,m}$, and labelling $(\ell)$ each sequence in $\mathfrak C^{(1)}_{n,\ell}$:
    $$\Omega_{\{101, 102, 201\}} = \left \{ \begin{array}{rclll}
    (0,0) \\
    (n,m) & \leadsto & (n+1,i) & \text{for} & i \in [m,n] \\
    && (i) & \text{for} & i \in [0,m-1] \vspace{5pt}\\
    (\ell) & \leadsto & (i) & \text{for} & i \in [0,\ell].
    \end{array} \right.$$
\end{rem}
It remains to count inversion sequences in $\I(102, 201)$ which contain 101.
By Proposition \ref{prop102+201} and Corollary \ref{corol102+201}, an inversion sequence $\sigma$ avoids \{102, 201\} and contains 101 if and only if it can be split into three factors $\alpha \cdot \beta \cdot \gamma = \sigma$ such that:
\begin{enumerate}
    \item $\alpha$ is a nondecreasing inversion sequence such that $\max(\alpha) < \max(\sigma)$,
    \item $\beta$ is a word over the alphabet $\{\max(\sigma), \sec(\sigma)\}$ which contains 101, and such that $\beta_1 = \max(\sigma)$,
    \item $\gamma$ is a (possibly empty) nonincreasing word over the alphabet $[0,\sec(\sigma)-1]$.
\end{enumerate}
We distinguish several sets of sequences according to how much of the first occurrence of the pattern 101 has already appeared.
Let
\begin{itemize}
    \item $\mathfrak B^{(1)}_{n,s} = \{\sigma \in \I_n(10) \; : \; s \in [\sec(\sigma), \max(\sigma)-1]\}$, which can be seen as the sequences in which only the first 1 in a pattern 101 has yet appeared.
    \item $\mathfrak B^{(2)}_{n,s} = \{\sigma \cdot (s)^i \; : \; i \in [1,n-2], \, \sigma \in \mathfrak B^{(1)}_{n-i,s} \}$, which can be seen as the sequences in which only the first 10 in a pattern 101 has yet appeared.
    \item $\mathfrak B^{(3)}_{n,s} = \{\sigma \in \I_n(102,201) \; : \; s = \sec(\sigma) \leqslant \sigma_n, \, \sigma \text{\; contains 101} \}$ be the set of $\{102, 201\}$-avoiding inversion sequences of size $n$ and second largest value $s$ which contain 101, and for which the factor denoted $\gamma$ earlier is empty.
    \item $\mathfrak C^{(2)}_{n,\ell} = \{\sigma \in \I_n(102,201\} \; : \; \ell = \sigma_n < \sec(\sigma), \, \sigma \text{\; contains 101}\}$ be the set of $\{102, 201)$-avoiding inversion sequences of size $n$ and last value $\ell$ which contain 101, and for which the factor denoted $\gamma$ earlier is nonempty.
\end{itemize}
Let $\mathfrak B^{(2)} = \coprod_{n,s \geqslant 0} \mathfrak B^{(2)}_{n,s}$, $\mathfrak B^{(3)} = \coprod_{n,s \geqslant 0} \mathfrak B^{(3)}_{n,s}$, and $\mathfrak C^{(2)} = \coprod_{n,\ell \geqslant 0} \mathfrak C^{(2)}_{n,\ell}$.
In particular, $\{\sigma \in \I(102,201) \; : \; \sigma \; \text{contains} \; 101\} = \mathfrak B^{(3)} \sqcup \mathfrak C^{(2)}$.
\begin{rem} \label{remintersect}
The subsets of $\I(102,201)$ we have defined are not disjoint. More precisely, the following holds.
\begin{enumerate}
    \item The sets $\mathfrak B^{(1)}_{n,s}$ intersect for different values of $s$. For instance, the sequence $(0,0,2)$ belongs to both $\mathfrak B^{(1)}_{3,0}$ and $\mathfrak B^{(1)}_{3,1}$.
    \item Each sequence in a set $\mathfrak B^{(1)}_{n,s}$ also belongs to a set $\mathfrak A_{n,m}$ for some $m > s$.
    \item Each set $\mathfrak B^{(2)}_{n,s}$ is a subset of $\mathfrak C^{(1)}_{n,s}$.
\end{enumerate}
\end{rem}
Let $\mathfrak b^{(i)}_{n,s} = |\mathfrak B^{(i)}_{n,s}|$ for $i \in \{1,2,3\}$, and $\mathfrak c^{(2)}_{n,s} = |\mathfrak C^{(2)}_{n,s}|$.

\begin{lem} \label{lem101}
For all $0 \leqslant s \leqslant n-2$,
$$\mathfrak b^{(1)}_{n,s} = \mathfrak b^{(1)}_{n-1,s} + (n-1-s) \mathfrak a_{n,s}.$$
For all $0 \leqslant s \leqslant n-3$,
$$\mathfrak b^{(2)}_{n,s} = \mathfrak b^{(2)}_{n-1,s} + \mathfrak b^{(1)}_{n-1,s}.$$
For all $0 \leqslant s \leqslant n-4$,
$$\mathfrak b^{(3)}_{n,s} = 2 \mathfrak b^{(3)}_{n-1,s} + \mathfrak b^{(2)}_{n-1,s}.$$
For all $0 \leqslant \ell \leqslant n-6$,
$$\mathfrak c^{(2)}_{n,\ell} = \sum_{i = \ell}^{n-5} \mathfrak c^{(2)}_{n-1,i} + \sum_{i = \ell+1}^{n-5}  \mathfrak b^{(3)}_{n-1,i}.$$
\end{lem}
\begin{proof}
    We give a bijective proof of each identity by removing the last entry of an inversion sequence.

    Let $n \geqslant 2, s \in [0,n-2]$, $\sigma \in \mathfrak B^{(1)}_{n,s}$. In particular, $\sigma_n = \max(\sigma)$. Let $\sigma' = (\sigma_j)_{j \in [1,n-1]}$ be the sequence obtained by removing the last entry from $\sigma$, and let $i = \sigma_{n-1}$. If $i = \sigma_n$, then $\sigma' \in \mathfrak B^{(1)}_{n-1,s}$. Otherwise, $i = \sec(\sigma) \leqslant s$, $\sigma' \in \mathfrak A_{n-1,i}$, and $\sigma_n \in [s+1,n-1]$. This yields a bijection between $\mathfrak B^{(1)}_{n,s}$ and $\mathfrak B^{(1)}_{n-1,s} \coprod \big ( [s+1,n-1] \times \coprod_{i=0}^s \mathfrak A_{n-1,i} \big )$. We know from the proof of Lemma \ref{lemunimod} that $\coprod_{i=0}^s \mathfrak A_{n-1,i}$ is in bijection with $\mathfrak A_{n,s}$.

    Let $n \geqslant 3, s \in [0,n-3]$, and $\sigma \in \mathfrak B^{(2)}_{n,s}$. In particular, $\sigma_{n} = s$, and $\sigma_{n-1} \in \{\max(\sigma), s\}$. Let $\sigma' = (\sigma_j)_{j \in [1,n-1]}$ be the sequence obtained by removing the last entry from $\sigma$, and let $i = \sigma_{n-1}$. If $i = s$, then $\sigma' \in \mathfrak B^{(2)}_{n-1,s}$. Otherwise $i = \max(\sigma)$, and $\sigma' \in \mathfrak B^{(1)}_{n-1,s}$.

   Let $n \geqslant 4, s \in [0,n-4]$, and $\sigma \in \mathfrak B^{(3)}_{n,s}$. Let $\sigma' = (\sigma_j)_{j \in [1,n-1]}$ be the sequence obtained by removing the last entry from $\sigma$. If $\sigma'$ contains 101, then $\sigma' \in \mathfrak B^{(3)}_{n-1,s}$, and $\sigma_{n} \in \{\max(\sigma), s\}$. Otherwise $\sigma' \in \mathfrak B^{(2)}_{n-1,s}$, and $\sigma_{n} = \max(\sigma)$.

   Let $n \geqslant 6, \ell \in [0,n-6]$, and $\sigma \in \mathfrak C^{(2)}_{n,\ell}$. Let $\sigma' = (\sigma_j)_{j \in [1,n-1]}$ be the sequence obtained by removing the last entry from $\sigma$, and let $i = \sigma_{n-1} \geqslant \ell$. If $i < \sec(\sigma)$, then $\sigma' \in \mathfrak C^{(2)}_{n-1,i}$. Otherwise $i > \ell$, and $\sigma' \in \mathfrak B^{(3)}_{n-1,i}$.
\end{proof}

Let $\mathfrak C_{n,\ell} = \mathfrak C^{(1)}_{n,\ell} \sqcup \mathfrak C^{(2)}_{n, \ell}$ and $\mathfrak c_{n,\ell} = |\mathfrak C_{n,\ell}| = \mathfrak c^{(1)}_{n,\ell} + \mathfrak c^{(2)}_{n, \ell}$. Let also $\mathfrak C = \mathfrak C^{(1)} \sqcup \mathfrak C^{(2)}$.
\begin{thm} \label{thm102+201}
    For all $n \geqslant 1$,
    $$|\I_n(102,201)| = \left (\sum_{m=0}^{n-1} \mathfrak a_{n,m} \right ) + \left ( \sum_{s=0} ^{n-4} \mathfrak b^{(3)}_{n,s} \right ) + \left ( \sum_{\ell = 0}^{n-3} \mathfrak c_{n,\ell} \right ).$$
\end{thm}
\begin{proof}
    Merging our results for inversion sequences which avoid or contain 101, we obtain
    \begin{equation*}\I(102,201) = \left ( \mathfrak A \sqcup \mathfrak C^{(1)} \right ) \sqcup \left ( \mathfrak B^{(3)} \sqcup \mathfrak C^{(2)} \right ) = \mathfrak A \sqcup \mathfrak B^{(3)} \sqcup \mathfrak C. \qedhere\end{equation*}
\end{proof}

\begin{rem} \label{remC102+201}
The equations for counting $\mathfrak c^{(1)}_{n,\ell}$ and $\mathfrak c^{(2)}_{n,\ell}$ from Lemmas \ref{lemunimod} and \ref{lem101} can be merged into a single recurrence relation for $\mathfrak c_{n,\ell}$. For all $0 \leqslant \ell \leqslant n-3$,
$$\mathfrak c_{n,\ell} = \mathfrak c_{n-1,\ell} + \sum_{i = \ell+1}^{n-2} \mathfrak a_{n-1,i} + \mathfrak b^{(3)}_{n-1,i} + \mathfrak c_{n-1,i}.$$
\end{rem}

\begin{rem} \label{remSuccession102+201}
    By labelling
    \begin{itemize}
        \item $(a,n,m)$ each sequence of $\mathfrak A_{n,m}$,
        \item $(b^{(i)},s)$ each sequence of $\mathfrak B^{(i)}_{n,s}$ for $i \in \{1,2,3\}$
        \item $(c,\ell)$ each sequence of $\mathfrak C_{n,\ell}$,
    \end{itemize}
    the recurrence relations from Lemmas \ref{lemunimod}, \ref{lem101}, and Remark \ref{remC102+201} correspond to the succession rule
    
$$\Omega_{\{102,201\}} = \left \{ \begin{array}{rclll}
    (a,0,0) \\
    (a,n,m) & \leadsto & (a,n+1,i) & \text{for} & i \in [m,n] \\
    && (b^{(1)},i)^{n-i} & \text{for} & i \in [m,n-1] \\
    && (c,i) & \text{for} & i \in [0,m-1] \vspace{5pt} \\
    (b^{(1)},s) & \leadsto & (b^{(1)},s) \, (b^{(2)},s) \vspace{5pt}\\
    (b^{(2)},s) & \leadsto & (b^{(2)},s) \, (b^{(3)},s) \vspace{5pt}\\
    (b^{(3)},s) & \leadsto & (b^{(3)},s)^2 \\
    && (c,i) & \text{for} & i \in [0,s-1] \vspace{5pt}\\
    (c, \ell) & \leadsto & (c,i) & \text{for} & i \in [0, \ell].
\end{array}
\right .$$
\end{rem}

Let $\mathcal T$ be the generating tree described by the succession rule $\Omega_{\{102,201\}}$.
We now explain how to relate $\mathcal T$ with the definitions of Section \ref{trees}.
First, note that $\mathcal T$ is not isomorphic to a combinatorial generating tree for the class $\I(102,201)$, because of the intersections mentioned in Remark \ref{remintersect}.
By construction, $\mathcal T$ is isomorphic to a combinatorial generating tree for the disjoint union $\mathcal U$ of the sets $\mathfrak A$, $\{\coprod_{n \geqslant 0} \mathfrak B^{(1)}_{n,s}\}_{s \geqslant 0}$, $\mathfrak B^{(2)}$, $\mathfrak B^{(3)}$ and $\mathfrak C$. We can view the combinatorial class $\mathcal U$ as a subset of $\mathbb N \times \I(102,201)$:
$$\mathcal U \cong \left (\{0\} \times \big (\mathfrak A \sqcup \mathfrak B^{(3)} \sqcup \mathfrak C \big ) \right ) \sqcup \left (\{1\} \times \mathfrak B^{(2)} \right ) \sqcup \left (\coprod_{s \geqslant 0} \left ( \{s+2\} \times \coprod_{n \geqslant 0} \mathfrak B^{(1)}_{n,s} \right ) \right ),$$
and define the size of each object $(k, \sigma) \in \mathcal U$ as the size of $\sigma$. By the proof of Theorem \ref{thm102+201}, the objects of $\mathcal U$ of the form $(0,\sigma)$ are trivially in bijection with $\I(102,201)$. The remaining objects of $\mathcal U$, of the form $(k, \sigma)$ for $k \geqslant 1$, can be called \emph{phantom objects} as they are ultimately not counted. These  phantom objects are nevertheless necessary in our construction, since removing all the nodes of $\mathcal T$ corresponding to phantom objects would disconnect\footnote{More precisely, each node of $\mathcal T$ corresponding to an object in $\{0\} \times \mathfrak B^{(3)}$ or $\{0\} \times \mathfrak C^{(2)}$ has an ancestor which corresponds to a phantom object in $\{1\} \times \mathfrak B^{(2)}$.} $\mathcal T$. The tree $\mathcal T$ is represented in Figure \ref{GT102+201}, on page \pageref{GT102+201}.

\begin{rem} \label{remPhantom}
The phantom objects can be understood as follows. In an object $(1, \sigma) \in \mathcal U$, $\sigma \in \I(101,102,201)$ is a prefix of an inversion sequence avoiding 102 and 201, whose maximum $m$ and second maximum $s$ form an occurrence $(m,s)$ of the pattern 10 which must be completed into an occurrence $(m,s,m)$ of 101 in the future. In particular, we may only insert the values $m$ or $s$ at the end of $\sigma$; otherwise, the subsequence $(m,s,m)$ could never appear, by Corollary \ref{corol102+201}. 
Similarly, in an object $(s+2, \sigma) \in \mathcal U$, $\sigma \in \I(10)$ is a prefix of an inversion sequence avoiding 102 and 201, that has maximum $m > s$ and which must contain an occurrence $(m,s,m)$ of 101 in the future. In particular, we may only insert the values $m$ or $s$ at the end of $\sigma$; otherwise, the subsequence $(m,s,m)$ could never appear, by Corollary \ref{corol102+201}.
\end{rem}

This construction may appear more complicated than necessary. Indeed, we could simply consider the generating tree growing on the right for $\I(102,201)$ instead, and this would not involve phantom objects. However, that would be less efficient, since this tree requires more labels. Specifically, the number of distinct labels appearing at level $n$ in this tree would be quadratic\footnote{We would need to record both the values of the maximum and the second maximum of sequences in $\mathfrak A$.} in $n$, whereas our construction only involves a linear number of labels.

\begin{landscape}
\begin{figure}
\centering
\begin{footnotesize}
\begin{tikzpicture}[level distance=19mm]
\tikzstyle{level 2}=[sibling distance=70mm]
\tikzstyle{level 3}=[sibling distance=13mm]
\tikzstyle{level 6}=[sibling distance=16mm]
\node[align=center]{\large $\varepsilon$ \\ $(a,0,0)$}
child {node[align=center]{\normalsize 0 \\ $(a,1,0)$}
    child {node[align=center]{00 \\ $(a,2,0)$}
        child {node[align=center]{000 \\ $(a,3,0)$}}
        child {node[align=center]{001 \\ $(a,3,1)$}
            child {node[align=center]{0011 \\ $(a,4,1)$}}
            child {node[align=center]{0012 \\ $(a,4,2)$}}
            child {node[align=center]{0013 \\ $(a,4,3)$}
                child {node[align=center]{00133 \\ $(a,5,3)$}}
                child {node[align=center]{00134 \\ $(a,5,4)$}}
                child {node[align=center]{\textcolor{gray}{00134} \\ $(b^{(1)}, 3)$}}
                child {node[align=center]{00130 \\ $(c,0)$}}
                child {node[align=center]{00131 \\ $(c,1)$}}
                child {node[align=center]{00132 \\ $(c,2)$}
                    child {node[align=center]{001320 \\ $(c,0)$}}
                    child {node[align=center]{001321 \\ $(c,1)$}}
                    child {node[align=center]{001322 \\ $(c,2)$}}
                }
            }
            child {node[align=center]{\textcolor{gray}{0012} \\ $(b^{(1)}, 1)$}}
            child {node[align=center]{\textcolor{gray}{0013} \\ $(b^{(1)}, 1)$}}
            child {node[align=center]{\textcolor{gray}{0013} \\ $(b^{(1)}, 2)$}}
            child {node[align=center]{0010 \\ $(c,0)$}}
        }
        child {node[align=center]{002 \\ $(a,3,2)$}}
        child {node[align=center]{\textcolor{gray}{001} \\ $(b^{(1)}, 0)$}}
        child {node[align=center]{\textcolor{gray}{002} \\ $(b^{(1)}, 0)$}}
        child {node[align=center]{\textcolor{gray}{002} \\ $(b^{(1)}, 1)$}}
    }
    child {node[align=center]{01 \\ $(a,2,1)$}
        child {node[align=center]{011 \\ $(a,3,1)$}}
        child {node[align=center]{012 \\ $(a,3,2)$}
            child {node[align=center]{0122 \\ $(a,4,2)$}}
            child {node[align=center]{0123 \\ $(a,4,3)$}}
            child {node[align=center]{\textcolor{gray}{0123} \\ $(b^{(1)}, 2)$}
                child {node[align=center]{\textcolor{gray}{01233} \\ $(b^{(1)}, 2)$}}
                child {node[align=center]{\textcolor{gray}{01232} \\ $(b^{(2)}, 2)$}
                    child {node[align=center]{\textcolor{gray}{012322} \\ $(b^{(2)}, 2)$}}
                    child {node[align=center]{012323 \\ $(b^{(3)}, 2)$}
                        child {node[align=center]{0123232 \\ $(b^{(3)}, 2)$}}
                        child {node[align=center]{0123233 \\ $(b^{(3)}, 2)$}}
                        child {node[align=center]{0123230 \\ $(c,0)$}}
                        child {node[align=center]{0123231 \\ $(c,1)$}}
                    }
                }
            }
            child {node[align=center]{0120 \\ $(c,0)$}}
            child {node[align=center]{0121 \\ $(c,1)$}
                child {node[align=center]{01210 \\ $(c,0)$}}
                child {node[align=center]{01211 \\ $(c,1)$}}
            }
        }
        child {node[align=center]{\textcolor{gray}{012} \\ $(b^{(1)}, 1)$}}
        child {node[xshift = 12mm, align=center] {010 \\ $(c,0)$}
            child {node[align=center]{0100 \\ $(c,0)$}}
        }
    }
    child {node[xshift = -15mm, align=center] {\textcolor{gray}{01} \\ $(b^{(1)},0)$}
        child {node[align=center]{\textcolor{gray}{011} \\ $(b^{(1)}, 0)$}}
        child {node[align=center]{\textcolor{gray}{010} \\ $(b^{(2)}, 0)$}
            child {node[align=center]{\textcolor{gray}{0100} \\ $(b^{(2)}, 0)$}}
            child {node[align=center]{0101 \\ $(b^{(3)}, 0)$}
                child {node[align=center]{01010 \\ $(b^{(3)}, 0)$}}
                child {node[align=center]{01011 \\ $(b^{(3)}, 0)$}}
            }
        }
    }
};
\end{tikzpicture}
\end{footnotesize}
\caption{Part of the generating tree described by the succession rule $\Omega_{\{102,201\}}$. Phantom objects are colored in gray.}  
\label{GT102+201}
\end{figure}
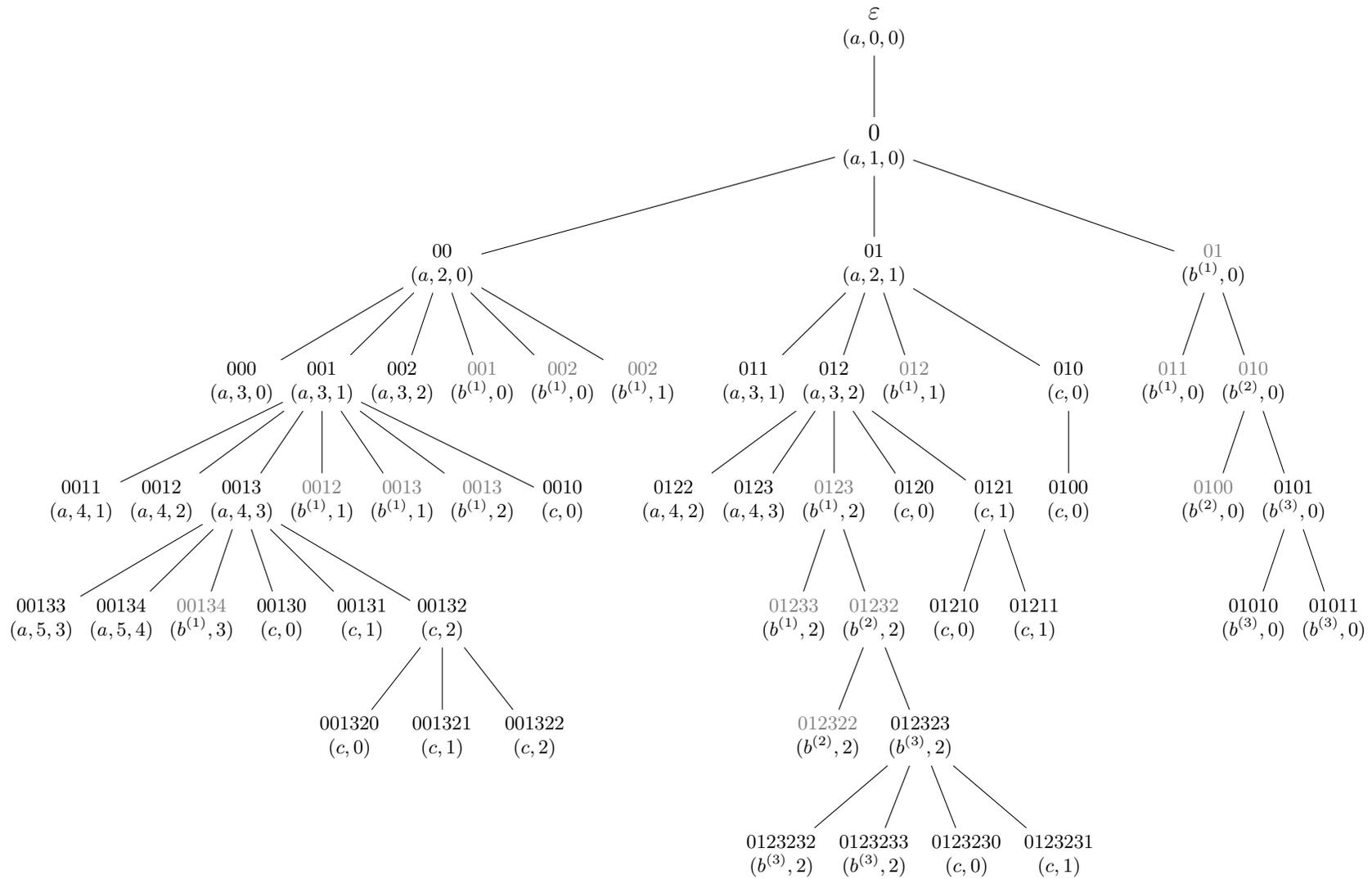
\end{landscape}

\begin{rem}
    We could directly obtain an explicit expression for $|\I_n(102, 201)|$ from the observations of Proposition \ref{prop102+201} and Corollaries \ref{corol102+201} and \ref{corolunimod}, although using it for enumeration is not as efficient as our generating tree approach.
    The sequences of $\I_n(101, 102, 201)$ are counted by
    $$1 + \sum_{m=1}^{n-1} \sum_{\ell=m+1}^n \frac{\binom{\ell+m-1}{m} (\ell-m)}{\ell} \cdot \binom{n-\ell+m-1}{n-\ell},$$
    as shown in \cite{Callan_Mansour}.
    The sequences of $\I_n(102, 201)$ which contain 101 are counted by
    $$\sum_{f=2}^{n-2} \sum_{\ell=f+2}^n \sum_{s=0}^{f-2} (f-s-1) \cdot \frac{\binom{f+s-1}{s} (f-s)}{f} \cdot (2^{\ell-f-1}-1) \cdot \binom{n-\ell+s}{n-\ell}.$$
    The summands count the number of sequences $\sigma \in \I_n(102, 201)$ such that $m = \max(\sigma)$, $s = \sec(\sigma)$, $f = \firstmax(\sigma)$, and $\ell = \lastmax(\sigma)$.
\end{rem}

\section{Removing the maximum} \label{sectionMax}

\subsection{Method}
Generally speaking, removing a term from an inversion sequence does not always create another inversion sequence. In fact, removing the $i$-th term $\sigma_i$ from an inversion sequence $\sigma \in \mathbb \I_n$ for some $i \in [1,n]$ yields an inversion sequence if and only if $\sigma_j < j-1$ for all $j \in [i+1,n]$. In particular if $\sigma_i = \max(\sigma)$, then for all $j \in [i+1,n]$ we have $\sigma_j \leqslant \sigma_i < i \leqslant j-1$, and the above condition is satisfied. Hence, removing a maximum from an inversion sequence always yields an inversion sequence.

Following this idea, we can decompose inversion sequences by removing their maxima one by one. For this decomposition to be unambiguous, we must choose the order of removal of the different occurrences of the maximum (this order is not always trivial). Observe that this decomposition defines a combinatorial generating tree $\mathcal T$ for inversion sequences. In that tree, the parent of each nonempty inversion sequence is obtained by removing its ``next" maximum in the chosen order.

For any set of patterns $P$, a combinatorial generating tree for $P$-avoiding inversion sequences can be obtained by removing all nodes of $\mathcal T$ whose label contains a pattern in $P$. Indeed, this procedure does not disconnect $\mathcal T$ since the parent (in $\mathcal T$) of any nonempty $P$-avoiding inversion sequence also avoids $P$.

Let us quickly go over each pattern that appears in this section, and establish necessary and sufficient conditions for an integer sequence to avoid any occurrence of the given pattern which involves the maximum of the sequence.
\begin{itemize}
    \item 000: The maximum does not appear more than twice.
    \item 100: All values appearing after the first maximum are distinct, or equal to the maximum.
    \item 101: All occurrences of the maximum are consecutive.
    \item 110: All maxima after the first one appear in a single factor, at the end of the sequence.
    \item 102: The subsequence to the left of the last maximum is nondecreasing, ignoring other occurrences of the maximum.
    \item 201: The subsequence to the right of the first maximum is nonincreasing, ignoring other occurrences of the maximum.
    \item 210: The subsequence to the right of the first maximum is nondecreasing, ignoring other occurrences of the maximum.
\end{itemize}
Since our generating tree construction inserts a maximum, the children of a node labelled $\sigma$ in $\mathcal T'$ are the children of $\sigma$ in $\mathcal T$ which satisfy the above conditions (for the patterns in $P$).
Such facts will be used in proofs throughout this section without explicit reference.

\subsection{The pair \{000, 102\}} \label{000+102}
In this subsection, we use a generalization of succession rules called \emph{doubled succession rules}, introduced in \cite{Ferrari_Pergola_Pinzani_Rinaldi}, which allows a node at level $n$ in a generating tree to have children at levels $n+1$ or $n+2$, denoted by $\overset{1} \leadsto$ and $\overset{2} \leadsto$ in the succession rule. We find a doubled succession rule describing a generating tree for $\I(000,102)$ (up to isomorphism), and later use this rule to compute the generating function of $\I(000,102)$ in the appendix.

For all $\sigma \in \I(000,102)$, let $\inc(\sigma)$ be the size of the longest nondecreasing factor at the beginning of $\sigma$, and let $\sites(\sigma) = \inc(\sigma) - \max(\sigma)$. In particular, $\sites(\varepsilon) = 1$.

\begin{rem}
    Let $\sigma \in \I(000,102) \backslash \{\varepsilon\}$ be a nonempty inversion sequence. Since $\sigma$ avoids the pattern 102, the factor $(\sigma_i)_{i \in [1, \firstmax(\sigma)]}$ is nondecreasing. In other words, the maximum of $\sigma$ is reached before its first descent, and therefore $\sigma_{\inc(\sigma)} = \max(\sigma)$.
    This also implies that $\inc(\sigma) > \max(\sigma)$, hence $\sites(\sigma)$ is positive.
\end{rem}

Given $n \geqslant 1, \sigma \in \mathbb N^n$, $p \in [1,n+1]$ and $v \in \mathbb N$, we denote by $[\sigma]_p^v$ the sequence obtained from $\sigma$ by inserting the value $v$ at position $p$, that is $[\sigma]_p^v = (\sigma_i)_{i \in [1,p-1]} \cdot v \cdot (\sigma_i)_{i \in [p,n]}$.
\begin{lem} \label{lem000+102}
    Let $n \geqslant 0$, $\sigma \in \I_n(000,102)$, $m = \max(\sigma)$, $i \in [1,n+1]$, and $k > 0$. Let $\sigma' = [\sigma]_i^{m+k}$ be the sequence obtained from $\sigma$ by inserting $m+k$ at position $i$. Then $\sigma' \in \I_{n+1}(000,102)$ if and only if $i \in [m+k+1, \inc(\sigma)+1]$.
\end{lem}
\begin{proof}
    Since $\sigma$ does not contain the value $m+k$, inserting it cannot create an occurrence of the pattern 000. The value $m+k$ must be inserted in the interval $[m+k+1, n+1]$ in order to create an inversion sequence. By definition of $\inc(\sigma)$, $(\sigma_i)_{i \in [1, \inc(\sigma)]}$ is nondecreasing, hence inserting $m+k$ in any position $[m+k+1, \inc(\sigma)+1]$ does not create an occurrence of 102. On the contrary, inserting $m+k$ in any position $[\inc(\sigma)+2, n+1]$ creates an occurrence of 102, since $(\sigma_{\inc(\sigma)}, \sigma_{\inc(\sigma)+1})$ is an occurrence of $10$ and $m+k > \sigma_{\inc(\sigma)}$.
\end{proof}
It follows from Lemma \ref{lem000+102} that the statistic $\sites(\sigma)$ can be interpreted either as the number of values (denoted $k$ earlier) greater than $m$ which could be inserted in $\sigma$, or the number of positions (denoted $i$ earlier) where a value greater than $m$ can be inserted in $\sigma$, so that the resulting sequence is in $\I_{n+1}(000,102)$.

\begin{thm} \label{thm000+102}
The enumeration of inversion sequences avoiding the patterns 000 and 102 is given by the doubled succession rule
$$\Omega_{\{000,102\}} = \begin{cases}
(1)\\
(s) \overset{1} \leadsto (j)^{s+1-j} &\text{for} \quad j \in [1, s]\\
(s) \overset{2} \leadsto (j+1)^{s+1-j} (j)^{\binom{s+1-j}{2}} &\text{for} \quad j \in [1, s].
\end{cases}$$
\end{thm}
\begin{proof}
We consider the combinatorial generating tree (with jumps) for $\I(000,102)$ such that the parent of any $\sigma \in \I(000,102) \backslash \{\varepsilon\}$ is obtained from $\sigma$ by removing all occurrences of $\max(\sigma)$.
Due to the avoidance of 000, $\max(\sigma)$ appears either once or twice in $\sigma$, so the difference between the level of $\sigma$ and that of its parent is either 1 or 2.
We then replace each label $\sigma$ by the value $\sites(\sigma)$ defined above.
The axiom of the resulting succession rule is (1) since $\sites(\varepsilon) = 1$.

In the rest of this proof, let $n \geqslant 0$, $\sigma \in \I_n(000,102)$ and $m = \max(\sigma)$.

Let $i,k \geqslant 1$ be two integers, and $\sigma' = [\sigma]_i^{m+k}$. From Lemma \ref{lem000+102}, $\sigma' \in \I_{n+1}(000,102)$ if and only if $i \in [m+2, \inc(\sigma)+1]$ and $k \in [1,i-m-1]$. Let us now assume that is the case. By definition of $\sigma'$, we have $\inc(\sigma') = i$ and $\max(\sigma') = m+k$. In particular, $\sites(\sigma') = i - (m+k) \in [1, \sites(\sigma)]$.

Let $j \in [1, \sites(\sigma)]$. Let us count how many choices of $i$ and $k$ satisfy the identity $\sites([\sigma]_i^{m+k}) = j$.
The ordered pair $(i,k)$ is a solution whenever $j = i-(m+k)$, i.e. whenever $k = i-m-j$. Since $k$ is positive, there is one solution for each value of $i$ such that $i > m + j$. The number of solutions $(i,k)$ is the number of values $i$ in the interval $[m+j+1, \inc(\sigma)+1]$, so there are $\inc(\sigma) + 1 - (m+j) = \sites(\sigma)+1-j$ solutions. This corresponds to $(s) \overset{1} \leadsto (j)^{s+1-j}$ in the succession rule.
Now let us count in how many ways we may insert two occurrences of a value greater than $m$ in $\sigma$.
In the previous case, inserting a second $m+k$ adjacent to the first one yields a sequence $\sigma'' = [\sigma']_i^{m+k}$ which satisfies $\sites(\sigma'') = i+1 - (m+k) = \sites(\sigma')+1$. We obtain $\sites(\sigma) + 1 - j$ sequences of size $n+2$ and label $j+1$ for each $j \in [1, \sites(\sigma)]$. This corresponds to $(s) \overset{2} \leadsto (j+1)^{s+1-j}$ in the succession rule.

In order to insert a value $m+k > m$ in two non adjacent positions in $\sigma$, we choose two positions $i_1 < i_2 \in [m+k+1, \inc(\sigma)+1]$ and define $\sigma'' = [[\sigma]_{i_2}^{m+k}]_{i_1}^{m+k}$, which satisfies $\sites(\sigma'') = i_1 - (m+k) \in [1, \sites(\sigma)]$.
Let $j \in [1, \sites(\sigma)]$, and let us count how many choices of $i_1$, $i_2$ and $k$ satisfy $\sites(\sigma'') = j$. The triple $(i_1,i_2,k)$ is a solution whenever $j = i_1-(m+k)$, i.e. whenever $k = i_1 - m - j$. Since $k$ is positive, there is one solution for each $(i_1, i_2)$ such that $i_1 > m + j$. The number of solutions is the number of ordered pairs $(i_1, i_2)$ such that $i_1 < i_2$ and $i_1, i_2 \in [m+j+1, \inc(\sigma)+1]$, so there are $\binom{\sites(\sigma)+1-j}{2}$ solutions. This corresponds to $(s) \overset{2} \leadsto (j)^{\binom{s+1-j}{2}}$ in the succession rule.
\end{proof}
\noindent The generating tree described by $\Omega_{\{000,102\}}$ is represented in Figure \ref{GT000+102}.

\begin{figure}[h]
\centering
\begin{tikzpicture}[level distance=20mm]
\begin{scope}
\tikzstyle{level 1}=[sibling distance=85mm]
\node[align=center]{\large $\varepsilon$ \\ $(1)$}
child[level distance = 2cm] {node[align=center]{0 \\ $(1)$} [sibling distance = 25mm]
    child {node[align=center]{01 \\ $(1)$}
        [sibling distance = 10mm]
        child {node[align=center]{012 \\ $(1)$}
            child {node[align=center]{0123 \\ $(1)$}}
        }
        child[level distance = 4cm] {node[align=center]{0122 \\ $(2)$}}
    }
    child[level distance = 4cm] {node[align=center]{011 \\ $(2)$}
        [sibling distance = 10mm, level distance = 2cm]
        child {node[align=center]{0112 \\ $(2)$}}
        child {node[align=center]{0113 \\ $(1)$}}
        child {node[align=center]{0121 \\ $(1)$}}
    }
}
child[level distance = 4cm] {node[align=center]{00 \\ $(2)$}
    [sibling distance = 10mm, level distance = 2cm]
    child {node[align=center, xshift = -10mm]{001 \\ $(2)$}
        child {node[align=center]{0012 \\ $(2)$}}
        child {node[align=center]{0013 \\ $(1)$}}
        child {node[align=center]{0021 \\ $(1)$}}
    }
    child {node[align=center]{002 \\ $(1)$}
        child {node[align=center]{0023 \\ $(1)$}}
    }
    child {node[align=center]{010 \\ $(1)$}
        child {node[align=center]{0120 \\ $(1)$}}
    }
    child[level distance = 4cm] {node[align=center]{0011 \\ $(3)$}}
    child[level distance = 4cm] {node[align=center]{0022 \\ $(2)$}}
    child[level distance = 4cm] {node[align=center]{0101 \\ $(1)$}}
    child[level distance = 4cm] {node[align=center]{0110 \\ $(2)$}}
};
\end{scope}
\end{tikzpicture}
\caption{First five levels of the generating tree described by the succession rule $\Omega_{\{000,102\}}$.}  
\label{GT000+102}
\end{figure}
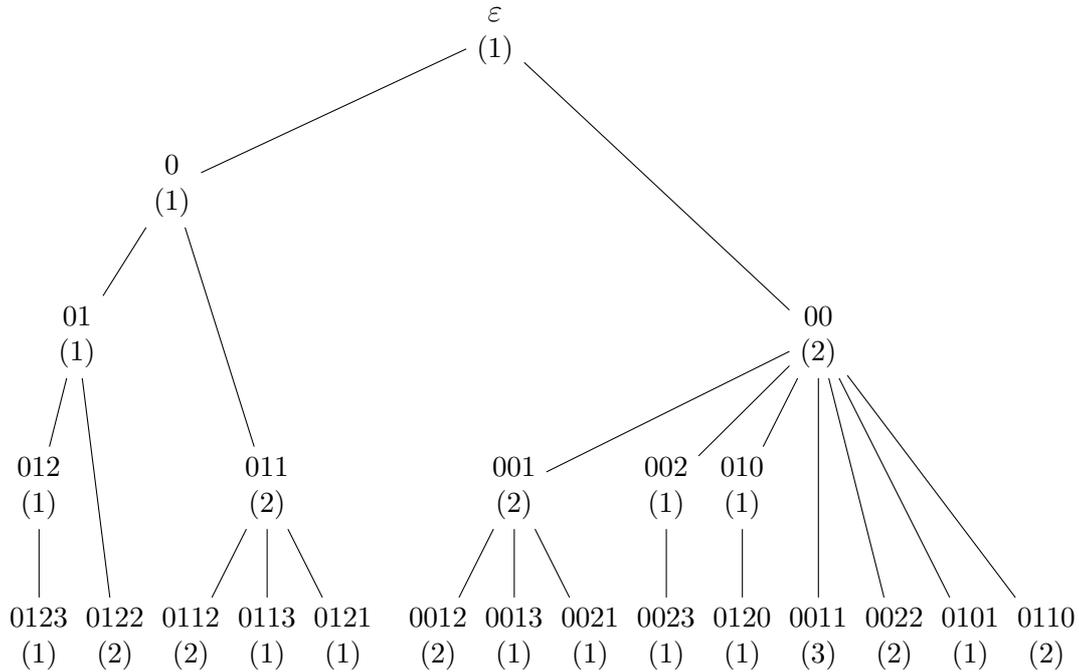

\subsection{The pair \{102, 210\}} \label{102+210}
In this section, we present a generating tree for the class $\I(102,210)$. We later use this construction to find the generating function of $\I(102,210)$ in Theorem \ref{thmGF102+210}.

For any $n \geqslant 0$ and any $\sigma \in \mathbb N^n$, let $\des(\sigma) = \{ i \in [1,n-1] \; : \; \sigma_i > \sigma_{i+1}\}$ be the set of \emph{descent tops} of $\sigma$.

\begin{prop} \label{propdescent}
Let $\sigma$ be a $\{102, 210\}$-avoiding integer sequence. Then any descent top of $\sigma$ is the position of a maximum of $\sigma$, i.e. if $d \in \des(\sigma)$ then $\sigma_d = \max(\sigma)$.
\end{prop}
\begin{proof}
    Assume, for the sake of contradiction, that $i \in \des(\sigma)$ and $\sigma_i < \max(\sigma)$. Since $\des(\sigma)$ is nonempty, $|\sigma| \geqslant 2$. Let $j \in [1,|\sigma|]$ be such that $\sigma_j = \max(\sigma)$. In particular, $j \notin \{i,i+1\}$. If $j < i$, then $(\sigma_j, \sigma_i, \sigma_{i+1})$ is an occurrence of 210. If $j > i+1$, then $(\sigma_i, \sigma_{i+1}, \sigma_j)$ is an occurrence of 102.
\end{proof}
\begin{prop} \label{propinsert102+210}
    Let $\sigma$ be a $\{102, 210\}$-avoiding integer sequence, and let $v > \max(\sigma)$.
    \begin{itemize}
        \item If $\sigma$ is nondecreasing, then inserting the value $v$ anywhere in $\sigma$ does not create any occurrence of 102 or 210.
        \item If $\sigma$ has exactly one descent $\sigma_d > \sigma_{d+1}$, then inserting the value $v$ at some position $p$ in $\sigma$ yields a $\{102, 210\}$-avoiding sequence if and only if $p = d+1$.
        \item If $\sigma$ has at least two descents, then inserting the value $v$ anywhere in $\sigma$ creates an occurrence of 102 or 210.
    \end{itemize}
\end{prop}
\begin{proof}
    We go case-by-case.
    \begin{itemize}
        \item Since $v$ is greater than all values of $\sigma$, the value $v$ could only take the role of the 2 in a pattern 102 or 210. Hence, in order for the insertion of the value $v$ to create an occurrence of either pattern, $\sigma$ must already contain the pattern 10.
        \item If $d$ is the only descent top of $\sigma$, then all occurrences $(\sigma_i, \sigma_j)$ of the pattern 10 in $\sigma$ must satisfy $i \leqslant d$ and $j > d$. In particular, inserting $v$ at position $d+1$ cannot create an occurrence of either pattern 102 or 210. If $v$ is inserted at any other position, the resulting sequence still contains a descent top of value $\sigma_d < v$, so it contains 102 or 210 by Proposition \ref{propdescent}.
        \item Inserting $v$ anywhere in $\sigma$ yields a sequence which still contains a descent top of value $\max(\sigma) < v$, and therefore contains 102 or 210 by Proposition \ref {propdescent}. \qedhere
    \end{itemize}
\end{proof}

It is now clear we want to partition $\I(102,210)$ into three subsets. For all $n \geqslant 0$, let
\begin{itemize}
    \item $\mathfrak A_n = \{\sigma \in \I_n(102, 210) \; : \; |\des(\sigma)| = 0 \} = \I_n(10)$ be the set of nondecreasing inversion sequences of size $n$,
    \item $\mathfrak B_n = \{\sigma \in \I_n(102, 210) \; : \; |\des(\sigma)| = 1 \}$ be the set of $\{102, 210\}$-avoiding inversion sequences of size $n$ having exactly one descent,
    \item $\mathfrak C_n = \{\sigma \in \I_n(102, 210) \; : \; |\des(\sigma)| \geqslant 2 \}$ be the set of $\{102, 210\}$-avoiding inversion sequences of size $n$ having at least two descents.
\end{itemize}
We define a statistic $\bounce : \I(102,210) \to \mathbb N$, which behaves differently on each subset.
\begin{itemize}
    \item For all $\sigma \in \mathfrak A_n$, let $\bounce(\sigma) = n - \sigma_n$, with the convention $\bounce(\varepsilon) = 0$.
    \item For all $\sigma \in \mathfrak B_n$, let $\bounce(\sigma) = d - \sigma_d$, where $d$ is the only element of $\des(\sigma)$.
    \item For all $\sigma \in \mathfrak C_n$, let $\bounce(\sigma) = \firstmax(\sigma) - \max(\sigma)$.
\end{itemize}
Note that when $\sigma$ is nonempty, $\bounce(\sigma)$ counts the difference between some position of the maximum of $\sigma$ and the value of this maximum; however, the exact position varies according to the number of descents of $\sigma$. In particular, $\bounce(\sigma) > 0$ if $\sigma \neq \varepsilon$.

Let $\mathfrak A_{n,k} = \{ \sigma \in \mathfrak A_n \; : \; k = \bounce(\sigma)\}$, $\mathfrak B_{n,k} = \{ \sigma \in \mathfrak B_n \; : \; k = \bounce(\sigma)\}$, and $\mathfrak C_{n,k} = \{ \sigma \in \mathfrak C_n \; : \; k = \bounce(\sigma)\}$. We also consider some subsets of $\mathfrak B_{n,k}$ and $\mathfrak C_{n,k}$.

\begin{itemize}
    \item Let $\mathfrak B^{(1)}_{n,k} = \{\sigma \in \mathfrak B_{n,k} \; : \; (\sigma_i)_{i \in [1,n] \backslash \firstmax(\sigma)} \in \mathfrak A_{n-1}\}$ be the subset of sequences of $\mathfrak B_{n,k}$ such that removing the first maximum yields a sequence in $\mathfrak A_{n-1}$. Let $\mathfrak B^{(2)}_{n,k} = \mathfrak B_{n,k} \backslash \mathfrak B^{(1)}_{n,k}$.
    \item Let $\mathfrak C^{(1)}_{n,k} = \{\sigma \in \mathfrak C_{n,k} \; : \; (\sigma_i)_{i \in [1,n] \backslash \firstmax(\sigma)} \in \mathfrak B_{n-1}\}$ be the subset of sequences of $\mathfrak C_{n,k}$ such that removing the first maximum yields a sequence in $\mathfrak B_{n-1}$. Let $\mathfrak C^{(2)}_{n,k} = \mathfrak C_{n,k} \backslash \mathfrak C^{(1)}_{n,k}$.
\end{itemize}
We give an example for each of the five cases.
\begin{itemize}
    \item $\alpha = (0,0,1,3,3,3,4,6,6) \in \mathfrak A_{9,3}$, since $\des(\alpha) = \emptyset$ and $\bounce(\alpha) = 9-6$.
    \item $\beta = (0,0,0,1,1,4,1,3,4) \in \mathfrak B^{(1)}_{9,2}$, since $\des(\beta) = \{6\}$ and  $\bounce(\beta) = 6-4$.
    \item $\beta' = (0,0,0,1,2,2,2,0,1) \in \mathfrak B^{(2)}_{9,5}$, since $\des(\beta') = \{7\}$ and $\bounce(\beta') = 7-2$.
    \item $\gamma = (0,0,1,1,2,4,3,4,3) \in \mathfrak C^{(1)}_{9,2}$, since $\des(\gamma) = \{6,8\}$ and $\bounce(\gamma) = 6-4$.
    \item $\gamma' = (0,1,1,3,3,1,3,2,3) \in \mathfrak C^{(2)}_{9,1}$, since $\des(\gamma') = \{5,7\}$ and  $\bounce(\gamma') = 4-3$.
\end{itemize}

\begin{thm} \label{thm102+210}
    The enumeration of inversion sequences avoiding the patterns 102 and 210 is given by the succession rule
    $$\Omega_{\{102,210\}} = \left \{ \begin{array}{rclll}
    (a,0) \\
    (a,k) & \leadsto & (a,i) & \text{for} & i \in [1,k+1] \\
    && (b^{(1)},i)^{k-i} & \text{for} & i \in [1,k-1] \vspace{5pt}\\
    (b^{(1)},k) & \leadsto & (b^{(1)},k) \\
    && (b^{(2)},i) & \text{for} & i \in [1,k+1] \\
    && (c^{(1)},i) & \text{for} & i \in [1,k-1] \vspace{5pt}\\
    (b^{(2)},k) & \leadsto & (b^{(2)},i) & \text{for} & i \in [1,k+1]  \vspace{5pt}\\
    (c^{(1)},k) & \leadsto & (c^{(1)},k) \\
    && (c^{(2)},i) & \text{for} & i \in [1,k] \vspace{5pt}\\
    (c^{(2)},k) & \leadsto & (c^{(2)},i) & \text{for} & i \in [1,k].
\end{array} \right.$$
\end{thm}
\begin{proof}
    We describe a combinatorial generating tree for $\I(102,210)$ which ``grows" inversion sequences by inserting entries in increasing order of value.
    We then label the sequences of $\I(102,210)$ according to which set they belong to: each sequence in a set $\mathfrak A_{n,k}$ is labelled $(a,k)$, each sequence in a set $\mathfrak B^{(i)}_{n,k}$ is labelled $(b^{(i)},k)$, and each sequence in a set $\mathfrak C^{(i)}_{n,k}$ is labelled $(c^{(i)},k)$, for $i \in \{1,2\}$. In particular, the empty sequence $\varepsilon$ is labelled $(a,0)$, which is the axiom of the succession rule.
    
    The order of insertion of the different occurrences of each value is not simple. We first explain it informally, before describing the full construction in detail. Let $\sigma \in \I(102,210)$ be an inversion sequence of size $n \geqslant 1$ and maximum $m$. Let $\sigma' \in \I(102,210)$ be the subsequence obtained by removing all values $m$ from $\sigma$.
    We generate $\sigma$ from $\sigma'$ by inserting entries of value $m$ in the following order.
    \begin{enumerate}
        \item The rightmost descent top of $\sigma$ (if $\sigma$ has at least one descent).
        \item Any entries of value $m$ appearing in a single factor at the end of $\sigma$, from left to right.
        \item The second rightmost descent top of $\sigma$ (if $\sigma$ has at least two descents).
        \item Any entries appearing in a constant factor of value $m$ which contains the rightmost descent top of $\sigma$ (if $\sigma$ has at least one descent), inserted from left to right, and to the right of the entry inserted at step 1.
        \item Any remaining entries of value $m$, from right to left.
    \end{enumerate}
    For instance if $\sigma = (0,0,0,1,2,2,6,6,2,3,6,4,6,6,6,4,4,6,6,6,4,6,6) \in \I_{23}(102,210)$, then the order of insertion of the 11 entries of maximal value $6$ is shown on the bottom row of the example below.
    \begin{center}
    \renewcommand{\arraystretch}{1}
    \setlength{\tabcolsep}{5pt}
    \begin{tabular}{ccccccccccccccccccccccc}
        0 & 0 & 0 & 1 & 2 & 2 & \textcolor{red} 6 & \textcolor{red} 6 & 2 & 3 & \textcolor{red} 6 & 4 & \textcolor{red} 6 & \textcolor{red} 6 & \textcolor{red} 6 & 4 & 4 & \textcolor{red} 6 & \textcolor{red} 6 & \textcolor{red} 6 & 4 & \textcolor{red} 6 & \textcolor{red} 6 \\
        &&&&&& 11 & 10 &&& 9 && 8 & 7 & 4 &&& 1 & 5 & 6 && 2 & 3
    \end{tabular}\\
    \end{center}
    This order of insertion may not be intuitive, but the resulting generating tree construction has some advantages compared to a simpler order:
    \begin{itemize}
        \item most importantly, it only requires one parameter (the value of the $\bounce$ statistic),
        \item it does not involve ``phantom objects" as in Section \ref{102+201},
        \item the resulting succession rule does not involve the size $n$ of the sequences.
    \end{itemize}

    We now turn to a step-by-step description of the generating tree, which matches the succession rule $\Omega_{\{102,210\}}$. Let $\sigma \in \I(102,210)$, $n = |\sigma|$, $k = \bounce(\sigma)$, and ${m = \max(\sigma)}$.
    \begin{itemize}
        \item Suppose $\sigma \in \mathfrak A_{n,k}$. In particular, if $\sigma$ is nonempty then $\max(\sigma) = \sigma_n = n-k$.
        \begin{itemize}
            \item Inserting any value $v \in [n-k, n]$ at the end of $\sigma$ yields a sequence in $\mathfrak A_{n+1, n+1-v}$, where $n+1-v \in [1,k+1]$. This production alone generates all nondecreasing inversion sequences.
            \item If $\sigma$ is nonempty, inserting any value $v \in [n-k+1, n-1]$ at any position ${p \in [v+1,n]}$ creates a sequence in $\mathfrak B^{(1)}_{n+1, p-v}$. For each $i = p-v \in [1,k-1]$, there are $k-i$ sequences in $\mathfrak B^{(1)}_{n+1, i}$ generated by $\sigma$ (one sequence for each value $v \in [n-k+1,n-i]$, for $p = v+i$).
        \end{itemize}
        \item Suppose $\sigma \in \mathfrak B^{(1)}_{n,k}$. In particular, $\des(\sigma) = \{k+m\}$.
        \begin{itemize}
            \item Inserting the value $m$ at the end of $\sigma$ yields a sequence in $\mathfrak B^{(1)}_{n+1,k}$.
            \item Inserting any value $v \in [m, m + k]$ between the top and bottom of the only descent of $\sigma$ (i.e. at position $k+m+1$) yields a sequence in $\mathfrak B^{(2)}_{n+1, k - v + m + 1}$.
            \item Inserting the value $m$ at any position $p \in [m+1, k+m-1]$ yields a sequence in $\mathfrak C^{(1)}_{n+1,p-m}$.
        \end{itemize}
        \item Suppose $\sigma \in \mathfrak B^{(2)}_{n,k}$.  In particular, $\des(\sigma) = \{k+m\}$. As before, inserting any value $v \in [m, m + k]$ between the top and bottom of the only descent of $\sigma$ (i.e. at position $k+m+1$) yields a sequence in $\mathfrak B^{(2)}_{n+1, k - v + m + 1}$.
        \item Suppose $\sigma \in \mathfrak C^{(1)}_{n,k}$. In particular, $\firstmax(\sigma) = k+m$.
        \begin{itemize}
            \item Inserting the value $m$ right-adjacent to the rightmost descent top of $\sigma$ (i.e. at position $\max(\des(\sigma))+1)$ yields a sequence in $\mathfrak C^{(1)}_{n+1,k}$.
            \item Inserting the value $m$ at any position $p \in [m+1, \firstmax(\sigma)]$ yields a sequence in $\mathfrak C^{(2)}_{n+1,p-m}$.
        \end{itemize}
        \item Suppose $\sigma \in \mathfrak C^{(2)}_{n,k}$. In particular, $\firstmax(\sigma) = k+m$. As before, inserting the value $m$ at any position $p \in [m+1, \firstmax(\sigma)]$ yields a sequence in $\mathfrak C^{(2)}_{n+1,p-m}$.
    \end{itemize}
    Notice that inserting some value $v \geqslant m$ at any other position would either yield a sequence which already appears in this construction, or create an occurrence of a pattern 102 or 210.
\end{proof}
\noindent The generating tree described by $\Omega_{\{102,210\}}$ is represented in Figure \ref{GT102+210}, on page \pageref{GT102+210}.

\begin{landscape}
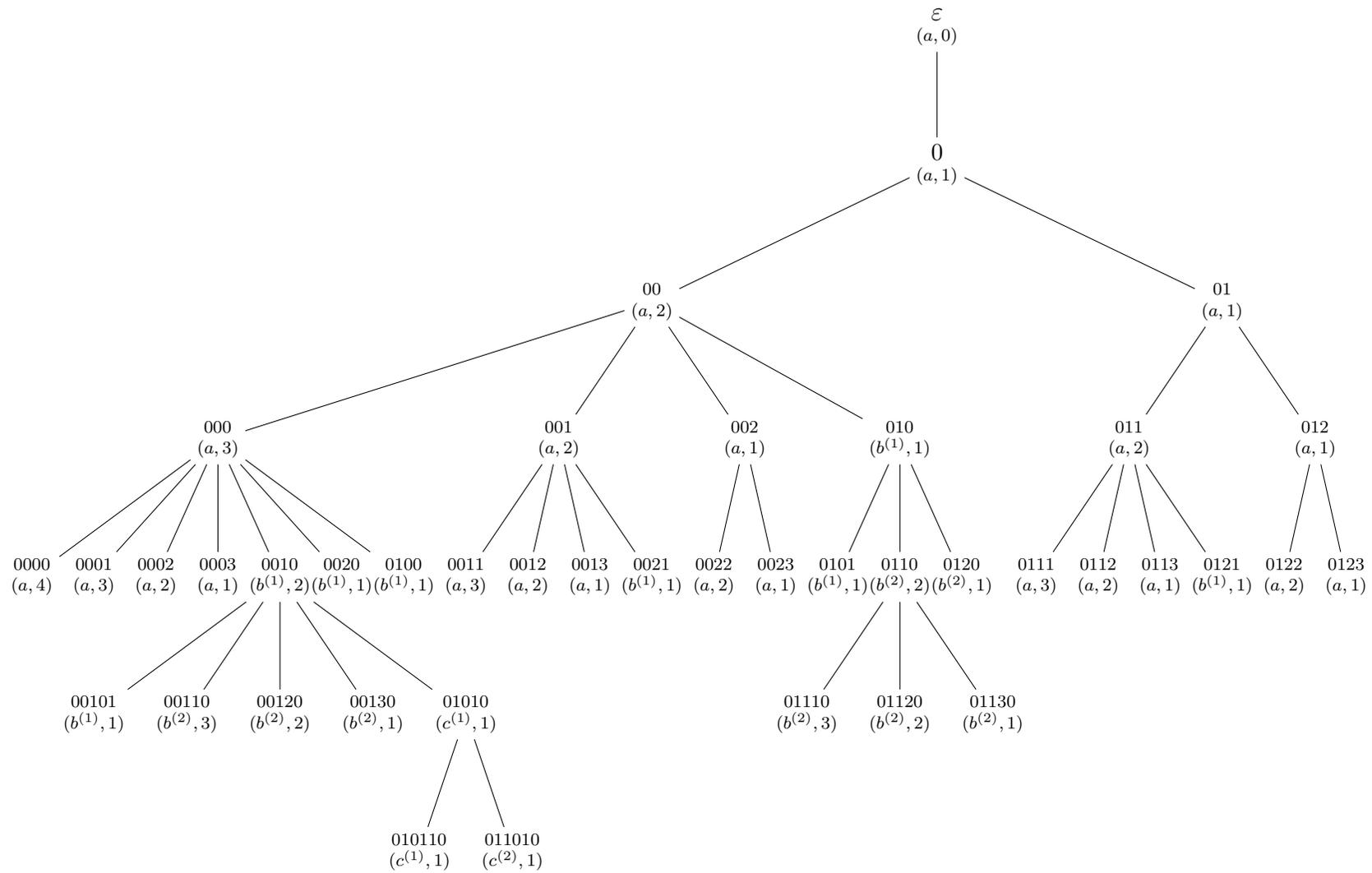
\begin{figure}
\begin{scriptsize}
\begin{tikzpicture}[level distance=22mm]
\begin{scope}
\tikzstyle{level 2}=[sibling distance=92mm]
\tikzstyle{level 3}=[sibling distance=30mm]
\tikzstyle{level 4}=[sibling distance=10mm]
\tikzstyle{level 5}=[sibling distance=15mm]
\node[align=center]{\large $\varepsilon$ \\ $(a,0)$}
child {node[align=center]{\normalsize 0 \\ $(a,1)$}
    child {node[align=center]{00 \\ $(a,2)$}
        child {node[align=center, xshift = -25mm]{000 \\ $(a,3)$}
            child {node[align=center]{0000 \\ $(a,4)$}}
            child {node[align=center]{0001 \\ $(a,3)$}}
            child {node[align=center]{0002 \\ $(a,2)$}}
            child {node[align=center]{0003 \\ $(a,1)$}}
            child {node[align=center]{0010 \\ $(b^{(1)},2)$}
                child {node[align=center]{00101 \\ $(b^{(1)},1)$}}
                child {node[align=center]{00110 \\ $(b^{(2)},3)$}}
                child {node[align=center]{00120 \\ $(b^{(2)},2)$}}
                child {node[align=center]{00130 \\ $(b^{(2)},1)$}}
                child {node[align=center]{01010 \\ $(c^{(1)},1)$}
                    child {node[align=center]{010110 \\ $(c^{(1)},1)$}}
                    child {node[align=center]{011010 \\ $(c^{(2)},1)$}}
                }
            }
            child {node[align=center]{0020 \\ $(b^{(1)},1)$}}
            child {node[align=center]{0100 \\ $(b^{(1)},1)$}}
        }
        child {node[align=center]{001 \\ $(a,2)$}
            child {node[align=center]{0011 \\ $(a,3)$}}
            child {node[align=center]{0012 \\ $(a,2)$}}
            child {node[align=center]{0013 \\ $(a,1)$}}
            child {node[align=center]{0021 \\ $(b^{(1)},1)$}}
        }
        child {node[align=center]{002 \\ $(a,1)$}
            child {node[align=center]{0022 \\ $(a,2)$}}
            child {node[align=center]{0023 \\ $(a,1)$}}
        }
        child {node[align=center, xshift = -5mm]{010 \\ $(b^{(1)},1)$}
            child {node[align=center]{0101 \\ $(b^{(1)},1)$}}
            child {node[align=center]{0110 \\ $(b^{(2)},2)$}
                child {node[align=center]{01110 \\ $(b^{(2)},3)$}}
                child {node[align=center]{01120 \\ $(b^{(2)},2)$}}
                child {node[align=center]{01130 \\ $(b^{(2)},1)$}}
            }
            child {node[align=center]{0120 \\ $(b^{(2)},1)$}}
        }
    }
    child {node[align=center]{01 \\ $(a,1)$}
        child {node[align=center]{011 \\ $(a,2)$}
            child {node[align=center]{0111 \\ $(a,3)$}}
            child {node[align=center]{0112 \\ $(a,2)$}}
            child {node[align=center]{0113 \\ $(a,1)$}}
            child {node[align=center]{0121 \\ $(b^{(1)},1)$}}
        }
        child {node[align=center]{012 \\ $(a,1)$}
            child {node[align=center]{0122 \\ $(a,2)$}}
            child {node[align=center]{0123 \\ $(a,1)$}}
        }
    }
};
\end{scope}
\end{tikzpicture}
\end{scriptsize}
\caption{Part of the generating tree described by the succession rule $\Omega_{\{102,210\}}$. }  
\label{GT102+210}
\end{figure}
\end{landscape}

\subsection{The pairs \{000, 201\} or \{000, 210\}} \label{000+201}
A bijection between $\I(000,201)$ and $\I(000,210)$ was established in \cite[Theorem 8.1]{Yan_Lin_2020}. In this section we enumerate $\I(000,201)$, although the same construction can also be applied to $\I(000,210)$ with a little more work, and results in the same equations.
Let $\mathfrak A_{n,m,p} = \{\sigma \in \I_n(000,201) \; : \; m = \max(\sigma), \; p = \firstmax(\sigma) = \lastmax(\sigma)\}$ be the set of $\{000, 201\}$-avoiding inversion sequences of size $n$, maximum $m$, and such that $m$ appears only once and at position $p$. Let $\mathfrak a_{n,m,p} = |\mathfrak A_{n,m,p}|$. By definition, $\mathfrak a_{n,m,p} = 0$ if $n < p$ or $p \leqslant m$. When $m = 0$, $\mathfrak a_{n,0,p} = \delta_{n,1} \delta_{p,1}$.
The numbers $(\mathfrak a_{n,m,p})_{n,m,p \in \mathbb N}$ can then be calculated using the following lemma.
\begin{lem} \label{lem000+201}
For all $1 \leqslant m < p \leqslant n$,
$$\mathfrak a_{n,m,p} = \sum_{s = 0}^{m-1} \sum_{q = s+1}^{p} \mathfrak a_{n-1,s,q} + (p-q+\delta_{p,q}-\delta_{p,n}) \mathfrak a_{n-2,s,q}$$

\end{lem}
\begin{proof}
    Let $m \geqslant 1, p > m, n \geqslant p$ be three integers, and let $\sigma \in \mathfrak A_{n,m,p}$. Let $\sigma'$ be the sequence obtained by removing the value $m$ from $\sigma$. Let $s = \max(\sigma')$ be the second largest value in $\sigma$, and $q = \firstmax(\sigma')$. Note that if $q < p$ then $\sigma_q = s$, and if $q \geqslant p$ then $\sigma_{q+1} = s$. In particular, if $q > p$ then the subsequence $(\sigma_p, \sigma_{p+1}, \sigma_{q+1})$ is an occurrence of the pattern 201. Hence $q$ is less than or equal to $p$.
    
    Since $\sigma'$ avoids the pattern 000, its maximum $s$ appears either once or twice.
    \begin{itemize}
        \item If $s$ appears only once in $\sigma'$, then $\sigma' \in \mathfrak A_{n-1,s,q}$. Note that inserting the value $m$ at position $p$ in any sequence in $\mathfrak A_{n-1,s,q}$ for $s < m$ and $q \leqslant p$ does not create an occurrence of 000 or 201, so this is a bijection.
        \item If $s$ appears twice, then removing the second occurrence yields an inversion sequence $\sigma'' \in \mathfrak A_{n-2,s,q}$. The possible positions for the second $s$ in $\sigma'$ are:
        \begin{itemize}
            \item Before $p$:
            $\begin{cases}
                [q+1,p-1] & \text{if} \quad q < p \\
                \emptyset & \text{if} \quad q = p
            \end{cases}$ , so there are $p-q-1+\delta_{p,q}$ choices.
            \item After $p$:
            $\begin{cases}
                \{p+1\} & \text{if} \quad p < n \\
                \emptyset & \text{if} \quad p = n
            \end{cases}$ , so there are $1-\delta_{p,n}$ choices.
        \end{itemize}
        This adds up to $p-q+\delta_{p,q}-\delta_{p,n}$ possible positions for the second $s$.
    \end{itemize}
    Summing over all possible values of $m$ and $p$ completes the proof.
\end{proof}
\begin{thm} \label{thm000+201}
For all $n \geqslant 2$,
$$|\I_n(000,201)| = \sum_{m = 0}^{n-1} \sum_{p = m+1}^n \mathfrak a_{n,m,p} + (n-p) \mathfrak a_{n-1,m,p}.$$
\end{thm}
\begin{proof}
    Let $n \geqslant 2$, $\sigma \in \I_n(000,201)$, $m = \max(\sigma)$, and $p = \firstmax(\sigma)$. Since $\sigma$ avoids 000, the maximum $m$ appears either once or twice in $\sigma$.
    \begin{itemize}
        \item If $m$ appears only once, then $\sigma \in \mathfrak A_{n,m,p}$.
        \item If $m$ appears twice, then removing the second $m$ yields a sequence $\sigma' \in \mathfrak A_{n-1,m,p}$. There are $n-p$ possible positions for the second $m$, hence $n-p$ different sequences $\sigma$ correspond to the same $\sigma'$.
    \end{itemize}
    Summing over all possible values of $m$ and $p$ completes the proof.
\end{proof}

\medskip

\begin{rem}
    The pairs of patterns studied in the next four sections contain either 101 or 110. In a sequence that avoids 101 (resp. 110), all occurrences of the maximum after the leftmost one must appear in a single factor at the end of the sequence (resp. consecutive to the leftmost maximum).
    This allows us to restrict the enumeration to inversion sequences which contain only one occurrence of the maximum without loss of generality, since each repetition of that maximum can only be placed at a single position.
\end{rem}

\subsection{The pair \{100, 110\}}
Let $\mathfrak A_{n,m,p} = \{\sigma \in \I_n(100,110) \; : \; m = \max(\sigma), \; p = \firstmax(\sigma) = \lastmax(\sigma)\}$ be the set of $\{100, 110\}$-avoiding inversion sequences of size $n$, maximum $m$, and such that $m$ appears only once and at position $p$. Let $\mathfrak a_{n,m,p} = |\mathfrak A_{n,m,p}|$. In particular, $\mathfrak a_{n,m,p} = 0$ if $n < p$ or $p \leqslant m$, and $\mathfrak a_{n,0,p} = \delta_{n,1} \delta_{p,1}$.
\begin{lem} \label{lem100+110}
Let $1 \leqslant m < p \leqslant n$. If $p \in \{n-1,n\}$,
$$\mathfrak a_{n,m,p} = \sum_{s=0}^{m-1} \sum_{q=s+1}^{n-1} \sum_{\ell = q}^{n-1} \mathfrak a_{\ell,s,q}.$$
If $p < n-1$,
$$\mathfrak a_{n,m,p} = \sum_{s=0}^{m-1} (n-p) \cdot \mathfrak a_{n-1,s,p} + \sum_{q=s+1}^{p-1} \mathfrak a_{n-1,s,q} + \mathfrak a_{n-2,s,q}.$$
\end{lem}
\begin{proof}
    Let $1 \leqslant m < p \leqslant n$, and $\sigma \in \mathfrak A_{n,m,p}$. Let $\sigma' = (\sigma_i)_{i \in [1,n] \backslash \{p\}}$ be the sequence obtained by removing the value $m$ from $\sigma$. Let $s = \max(\sigma')$ and $q = \firstmax(\sigma')$. Let $\sigma''$ be the sequence obtained by removing all terms of value $s$ except the first one from $\sigma'$, so that $\sigma'' \in \mathfrak A_{\ell,s,q}$ for some $\ell \in [q,n-1]$.

    Since $\sigma$ avoids 100, all terms of $\sigma$ after position $p$ must have distinct values. In particular, $s$ appears at most once after $p$. Since $\sigma'$ avoids 110, every $s$ after the first one must appear in a single factor at the end of $\sigma'$. At this point, there are two cases:
    \begin{enumerate}
        \item If $p \in \{n-1,n\}$, then the value $s$ could appear any number of times without creating an occurrence of 100 in $\sigma$. In that case, $\sigma''$ can be any sequence in $\mathfrak A_{\ell,s,q}$ for any $\ell \in [q,n-1]$: starting from any sequence in $\mathfrak A_{\ell,s,q}$ for any $s < m$, $q > s$, and $\ell < n$, we can append $n-\ell-1$ terms of value $s$, then insert $m$ at position $p \in \{n-1,n\}$ without creating an occurrence of 100 or 110, and this is a bijection. Thus, we obtain the first equation of the lemma.
        \item If $p < n-1$:
        \begin{itemize}
            \item If $q < p$: 
            \begin{itemize}
                \item If $\sigma$ contains a single term of value $s$, then $\sigma'' = \sigma' \in \mathfrak A_{n-1,s,q}$.
                \item If $\sigma$ contains two terms of value $s$, then the second one is $\sigma_n$ due to the avoidance of 110. In that case, $\sigma'' \in \mathfrak A_{n-2,s,q}$.
            \end{itemize}
            In either case, this is a bijection for the same reason as before. Note that $s$ cannot appear more than twice in $\sigma$, otherwise $(\sigma_p, \sigma_{n-1}, \sigma_n) = (m,s,s)$ would be an occurrence of 100.
            \item If $q \geqslant p$, then $s$ appears only once in $\sigma$ in order to avoid 100.
            In that case, $\sigma'' = \sigma'$ can be any sequence in $\mathfrak A_{n-1,s,q}$ such that all values $(\sigma''_i)_{i \in [p,n-1]}$ are distinct. Moving the value $s$ in $\sigma''$ from position $q$ to position $p$ yields a sequence $\tau \in \mathfrak A_{n-1,s,p}$. Conversely, starting from any sequence $\tau \in \mathfrak A_{n-1,s,p}$ and moving the value $s$ from position $p$ to position $q \in [p,n-1]$ yields a sequence $\sigma'' \in \mathfrak A_{n-1,s,q}$ such that all values $(\sigma''_i)_{i \in [p,n-1]}$ are distinct. Since there are $n-p$ possible values of $q$, each sequence $\tau \in \mathfrak A_{n-1,s,p}$ is obtained from $n-p$ different sequences $\sigma''$. \qedhere
        \end{itemize}
    \end{enumerate}
\end{proof}
\begin{thm} \label{thm100+110}
For all $n \geqslant 2$,
$$|\I_n(100,110)| = \sum_{\ell = 1}^n \sum_{m = 0}^{\ell-1} \sum_{p = m+1}^\ell \mathfrak a_{\ell,m,p}.$$
\end{thm}
\begin{proof}
     Let $n \geqslant 2$, $\sigma \in \I_n(100, 110)$, $m = \max(\sigma)$, and $p = \firstmax(\sigma)$. Since $\sigma$ avoids $110$, all terms of value $m$ after the first one $\sigma_p$ are in a factor $(\sigma_i)_{i \in [\ell+1, n]}$ for some $\ell \in [p,n]$. Let $\sigma'$ be the sequence obtained by removing all terms of value $m$ from $\sigma$, except the first one $\sigma_p$. By definition, $\sigma' \in \mathfrak A_{\ell, m, p}$. Since we can rebuild $\sigma$ by appending $n-\ell$ terms of value $m$ to $\sigma'$, this is a bijection between $\I_n(100,110)$ and $\coprod_{0 \leqslant m < p \leqslant \ell \leqslant n} \mathfrak A_{\ell,m,p}$.
\end{proof}

\subsection{The pair \{100, 101\}}
Let $\mathfrak A_{n,m,p} = \{\sigma \in \I_n(100,101) \; : \; m = \max(\sigma), \; p = \firstmax(\sigma) = \lastmax(\sigma)\}$ be the set of $\{100, 101\}$-avoiding inversion sequences of size $n$, maximum $m$, and such that $m$ appears only once and at position $p$. Let $\mathfrak a_{n,m,p} = |\mathfrak A_{n,m,p}|$. In particular, $\mathfrak a_{n,m,p} = 0$ if $n < p$ or $p \leqslant m$, and $\mathfrak a_{n,0,p} = \delta_{n,1} \delta_{p,1}$.

\begin{lem} \label{lem100+101}
Let $1 \leqslant m < p \leqslant n$. If $p \in \{n-1,n\}$,
$$\mathfrak a_{n,m,p} = \sum_{s=0}^{m-1} \sum_{q=s+1}^{n-1} \sum_{\ell = q}^{n-1} \mathfrak a_{\ell,s,q}.$$
If $p < n-1$,
$$\mathfrak a_{n,m,p} = \sum_{s=0}^{m-1} (n-p) \cdot \mathfrak a_{n-1,s,p} + \sum_{q=s+1}^{p-1} \sum_{\ell = n-p+q-1}^{n-1} \mathfrak a_{\ell,s,q}.$$
\end{lem}
\begin{proof}
    This is quite similar to the proof of Lemma \ref{lem100+110}, except that now the terms of value $s$ in $\sigma'$ appear in a single factor starting at position $q$, instead of being split into $\sigma'_q$ and a factor at the end of $\sigma'$. This makes a difference in only one case: if $p < n-1$ and $q < p$, then there may be more than two terms of value $s$. More precisely, the terms of value $s$ in $\sigma'$ are in a factor $(\sigma_i)_{i \in [q, q+k]}$ for some $k \in [0, p-q]$. In particular $\sigma'' = (\sigma'_i)_{i \in [1,n-1] \backslash [q+1,q+k]} \in \mathfrak A_{\ell,s,q}$, so $\ell = n-1-k \in [n-p+q-1, n-1]$.
\end{proof}

\begin{thm} \label{thm100+101}
For all $n \geqslant 2$,
$$|\I_n(100,101)| = \sum_{\ell = 1}^n \sum_{m = 0}^{\ell-1} \sum_{p = m+1}^\ell \mathfrak a_{\ell,m,p}.$$
\end{thm}
\begin{proof}
Very similar to the proof of Theorem \ref{thm100+110}, but now the factor of repetitions of the maximum is placed immediately after the first maximum, instead of being at the end of the sequence.
\end{proof}

\subsection{The pair \{110, 201\}}

Let $\mathfrak A_{n,m,p} = \{\sigma \in \I_n(110,201) \; : \; m = \max(\sigma), \; p = \firstmax(\sigma) = \lastmax(\sigma)\}$ be the set of $\{110, 201\}$-avoiding inversion sequences of size $n$, maximum $m$, and such that $m$ appears only once and at position $p$. Let $\mathfrak a_{n,m,p} = |\mathfrak A_{n,m,p}|$. In particular, $\mathfrak a_{n,m,p} = 0$ if $n < p$ or $p \leqslant m$, and $\mathfrak a_{n,0,p} = \delta_{n,1} \delta_{p,1}$.

\begin{lem}
Let $1 \leqslant m < p \leqslant n$. If $p \in \{n-1,n\}$,
$$\mathfrak a_{n,m,p} = \sum_{s=0}^{m-1} \sum_{q=s+1}^{n-1} \sum_{\ell = q}^{n-1} \mathfrak a_{\ell,s,q}.$$
If $p < n-1$,
$$\mathfrak a_{n,m,p} = \sum_{s=0}^{m-1} \sum_{q=s+1}^{p} \mathfrak a_{n-1,s,q} + \sum_{\ell = q}^{p-1+\delta_{p,q}} \mathfrak a_{\ell,s,q}.$$
\end{lem}
\begin{proof}
Let $1 \leqslant m < p \leqslant n$, and $\sigma \in \mathfrak A_{n,m,p}$. Let $\sigma' = (\sigma_i)_{i \in [1,n] \backslash \{p\}}$ be the sequence obtained by removing the value $m$ from $\sigma$. Let $s = \max(\sigma')$ and $q = \firstmax(\sigma')$. In particular $q \leqslant p$, otherwise $(\sigma_p, \sigma_{p+1}, \sigma_{q+1})$ would be an occurrence of 201. Let $\sigma''$ be the sequence obtained by removing all terms of value $s$ except the first one from $\sigma'$, so that $\sigma'' \in \mathfrak A_{\ell,s,q}$ for some $\ell \in [q,n-1]$.
\begin{itemize}
    \item If $p \in \{n-1,n\}$, we obtain the same equation as in Lemma \ref{lem100+110}, for the same reason.
    \item If $p < n-1$,
    \begin{itemize}
        \item If $s$ appears only once in $\sigma'$, then $\sigma'' = \sigma'$ can be any sequence in $\mathfrak A_{n-1,s,q}$
        \item If $s$ appears more than once, then every $s$ after the first one must be in a single factor at the end of $\sigma'$ to avoid 110. Since $\sigma$ avoids 201, the factor $(\sigma_i)_{i \in [p+1,n]}$ is nonincreasing, which means every term in this factor has value $s$. The sequence $\sigma''$ obtained has length $\ell \in [q, p-1]$ if $q < p$, or $\ell = p$ if $q = p$. \qedhere
    \end{itemize}
\end{itemize}
\end{proof}

\begin{thm} \label{thm110+201}
For all $n \geqslant 2$,
$$|\I_n(110,201)| = \sum_{\ell = 1}^n \sum_{m = 0}^{\ell-1} \sum_{p = m+1}^\ell \mathfrak a_{\ell,m,p}.$$
\end{thm}
\begin{proof}
Identical to the proof of Theorem \ref{thm100+110}, since inserting a term of maximal value anywhere after the first maximum cannot create an occurrence of 201.
\end{proof}

\subsection{The pair \{101, 210\}} \label{101+210}

Let $\mathfrak A_{n,m,p} = \{\sigma \in \I_n(101,210) \; : \; m = \max(\sigma), \; p = \firstmax(\sigma) = \lastmax(\sigma)\}$ be the set of $\{101, 210\}$-avoiding inversion sequences of size $n$, maximum $m$, and such that $m$ appears only once and at position $p$. Let $\mathfrak a_{n,m,p} = |\mathfrak A_{n,m,p}|$. In particular, $\mathfrak a_{n,m,p} = 0$ if $n < p$ or $p \leqslant m$, and $\mathfrak a_{n,0,p} = \delta_{n,1} \delta_{p,1}$.

\begin{lem}
Let $1 \leqslant m < p \leqslant n$. If $p \in \{n-1,n\}$,
$$\mathfrak a_{n,m,p} = \sum_{s=0}^{m-1} \sum_{q=s+1}^{n-1} \sum_{\ell = q}^{n-1} \mathfrak a_{\ell,s,q}.$$
If $p < n-1$,
$$\mathfrak a_{n,m,p} = \sum_{s=0}^{m-1} \left ( \sum_{q = p}^{n-1} \mathfrak a_{q,s,p} \right ) + \sum_{q=s+1}^{p-1} \mathfrak a_{q,s,q} + \sum_{\ell = n-p+q}^{n-1} \mathfrak a_{\ell,s,q}.$$
\end{lem}
\begin{proof}
Let $1 \leqslant m < p \leqslant n$, and $\sigma \in \mathfrak A_{n,m,p}$. Let $\sigma' = (\sigma_i)_{i \in [1,n] \backslash \{p\}}$ be the sequence obtained by removing the value $m$ from $\sigma$. Let $s = \max(\sigma')$ and $q = \firstmax(\sigma')$. Let $\sigma''$ be the sequence obtained by removing all terms of value $s$ except the first one from $\sigma'$, so that $\sigma'' \in \mathfrak A_{\ell,s,q}$ for some $\ell \in [q,n-1]$.
\begin{itemize}
    \item If $p \in \{n-1,n\}$, we obtain the same equation as in Lemma \ref{lem100+101}, for the same reason.
    \item If $p < n-1$:
    \begin{itemize}
        \item If $q \geqslant p$, then every term in the factor $(\sigma_i)_{i \in [q+1, n]}$ has value $s$, because $\sigma$ avoids 210. In that case $\sigma''$ is a sequence in $\mathfrak A_{q,s,q}$ such that $(\sigma''_i)_{i \in [p,q]}$ is nondecreasing. The set of such sequences $\sigma''$ is in bijection with $\mathfrak A_{q,s,p}$ (to see that, simply move the value $s$ from position $q$ to position $p$).
        
        \item If $q < p$, let $k+1$ be the number of terms of value $s$ in $\sigma$. In particular, $k = |\sigma'| - |\sigma''| = n-1-\ell$. Note that the terms of value $s$ in $\sigma'$ are exactly $(\sigma'_i)_{i \in [q, q+k]}$.
        \begin{itemize}
            \item If $k \geqslant p-q$, then $\sigma_{p+1} = s$, which implies that every term $(\sigma_i)_{i \in [q, n] \backslash \{p\}}$ has value $s$ because $\sigma$ avoids 210. In particular $k = n-1-q$, $\ell = q$, and we obtain a sequence $\sigma'' \in \mathfrak A_{q,s,q}$.
            \item If $k < p-q$, then the value $s$ does not appear after position $p$, and we obtain a sequence $\sigma'' \in \mathfrak A_{\ell,s,q}$ for $\ell = n-1-k \in [n-p+q,n-1]$. \qedhere
        \end{itemize}
    \end{itemize}
\end{itemize}
\end{proof}

\begin{thm} \label{thm101+210}
For all $n \geqslant 2$,
$$|\I_n(101,210)| = \sum_{\ell = 1}^n \sum_{m = 0}^{\ell-1} \sum_{p = m+1}^\ell \mathfrak a_{\ell,m,p}.$$
\end{thm}
\begin{proof}
Identical to the proof of Theorem \ref{thm100+101}, since inserting a term of maximal value anywhere after the first maximum cannot create an occurrence of 210.
\end{proof}

\subsection{A generating tree perspective on the previous constructions}

The inversion sequences studied in Sections \ref{000+201} to \ref{101+210} could also be generated by inserting their values one by one, in increasing order (and inserting any repetitions of a value from left to right). This yields generating trees which differ only slightly from our construction. We chose not to present them that way, since it would require distinguishing several subsets of inversion sequences, as we did in Sections \ref{102+201} and \ref{102+210}.
In all five cases, the resulting succession rules would have two integer parameters, which always record the number of positions where a new maximum may be inserted, and the number of possible values of a new maximum. The first parameter is influenced by pattern-avoidance, and the second one is used to maintain the ``subdiagonal" property characterizing inversion sequences.

Note that Sections \ref{000+102} and \ref{102+210} use such generating tree constructions, but only have one integer parameter.
\begin{itemize}
    \item For the pair of patterns $\{000, 102\}$, the two parameters are one and the same, and correspond to the statistic we denoted $\sites$, as explained in the paragraph just before Theorem \ref{thm000+102}.
    \item For the pair $\{102, 210\}$:
    \begin{itemize}
        \item a new maximum cannot be inserted in a sequence with several descents; it follows that such sequences only require one parameter to record the number of positions where repetitions of the current maximum could be inserted.
        \item in a sequence which has 0 or 1 descents, the positions where a new maximum may be inserted are trivial (for 0 descents any position is fine, and for 1 descent the only available position is between the top and the bottom of the descent), so only the number of possible values of a new maximum is required.
    \end{itemize}
\end{itemize}

\section{Splitting at the first maximum} \label{sectionSplitMax}
\subsection{Method} 
The method described in this section is a generalization of the construction used in \cite{Mansour_Shattuck_2015} for 120-avoiding inversion sequences.

Let $\sigma \in \I_n$ be an inversion sequence of size $n \geqslant 1$, $m = \max(\sigma)$, and $p = \firstmax(\sigma)$. By definition, $p > m$. Let $\alpha = (\sigma_i)_{i \in [1,p-1]}$, $\beta = (\sigma_i)_{i \in [p,n]}$ be two factors of $\sigma$, which satisfy $\alpha \cdot \beta = \sigma$.
The case $m = 0$ corresponds to constant inversion sequences, whose enumeration is trivial. When $m > 0$, we use this decomposition to count inversion sequences $\sigma$ from the number of possible choices for $\alpha$ and $\beta$.
Note that $\alpha$ is an inversion sequence of size $p-1$ and maximum less than $m$, and $\beta$ is a word of length $n-p+1$ over the alphabet $[0,m]$ such that $\beta_1 = m$. More precisely, the following holds (where $\W_{n,k}$ is the set of words of length $n$ over the alphabet $[0,k-1]$).
\begin{rem} \label{remsplit}
    For all $0 \leqslant m < p < n$, $\{\sigma \in \I_n \; : \; m = \max(\sigma), \, p = \firstmax(\sigma)\} = \{\alpha \cdot \beta \; : \; \alpha \in \I_{p-1}, \, \max(\alpha) < m, \, \beta \in \W_{n-p+1,m+1} ,\, \beta_1 = m\}$.
\end{rem}
Now we can express pattern avoidance on $\sigma$ in terms of necessary and sufficient conditions on $\alpha$ and $\beta$.
\begin{rem} \label{remsplitpattern}
    Let $\tau$ be a pattern of length $k \geqslant 2$. Then $\sigma$ avoids $\tau$ if and only if all three following conditions are satisfied:
\begin{enumerate}
    \item $\alpha$ avoids $\tau$,
    \item $\beta$ avoids $\tau$,
    \item There are no $i_1 < i_2 < \dots < i_k \in [1,n]$ such that $i_1 < p$, $i_k \geqslant p$ and $(\sigma_{i_j})_{j \in [1,k]}$ is an occurrence of $\tau$.
\end{enumerate}
\end{rem}
Condition 3 is difficult to work with in general, but it can be turned into simple conditions on $\alpha$ and $\beta$ for certain patterns, such as 120 or 010.
\begin{prop} \label{propsplit120}
If $\tau = 120$, condition 3 is equivalent to $\max(\alpha) \leqslant \min(\beta)$.
\end{prop}
\begin{proof}
    If $\max(\alpha) > \min(\beta)$, then $(\max(\alpha), m, \min(\beta))$ is an occurrence of 120.
    Conversely, if there are three integers $i_1 < i_2 < i_3$ such that $i_1 < p$, $i_3 \geqslant p$ and $(\sigma_{i_1}, \sigma_{i_2}, \sigma_{i_3})$ is an occurrence of 120, then $\max(\alpha) \geqslant \sigma_{i_1} > \sigma_{i_3} \geqslant \min(\beta)$.
\end{proof}
\begin{prop} \label{propsplit010}
    If $\tau = 010$, condition 3 is equivalent to $\vals(\alpha) \cap \vals(\beta) = \emptyset$.
\end{prop}
\begin{proof}
    If $v \in \vals(\alpha) \cap \vals(\beta)$, then $(v,m,v)$ is an occurrence of 010.
    Conversely, if there are three integers $i_1 < i_2 < i_3$ such that $i_1 < p$, $i_3 \geqslant p$ and $(\sigma_{i_1}, \sigma_{i_2}, \sigma_{i_3})$ is an occurrence of 010, then $\sigma_{i_1} = \sigma_{i_3} \in \vals(\alpha) \cap \vals(\beta)$.
\end{proof}
Sections \ref{symmetry} to $\ref{000+120}$ are focused on the pattern 120, while Sections \ref{010} to \ref{010+110} are focused on the pattern 010.

We always count the possible choices for $\alpha$ by recurrence, since $\alpha$ is an inversion sequence that is shorter than $\sigma$ and must avoid the same patterns.
On the other hand, the words $\beta$ form a different family of objects which must be enumerated independently. This could be solved by applying the same decomposition around the first maximum to words, since Remark \ref{remsplit} has an analogue for words\footnote{More precisely, the set of words $\omega \in \W_{n,k}$ such that $m = \max(\omega)$ and $p = \firstmax(\omega)$ (note that $p$ does not have to be greater than $m$ now) is in bijection with the set of ordered pairs $(\alpha, \beta)$ such that $\alpha \in \W_{p-1,m}$ and $\beta \in \W_{n-p+1,m+1}$ satisfies $\beta_1 = m$. Remark \ref{remsplitpattern} and Propositions \ref{propsplit120} and \ref{propsplit010} apply to these words as well.}. We use generating trees instead, since it results in simpler and more efficient recurrence formulas.

Depending on the pair of patterns studied, we consider different types of words in our decomposition. We always denote $\beta$ the word defined as above, $\gamma = (\beta_i)_{i \in [2,n-p]}$ the same word without the first letter $m$ (so that $\sigma = \alpha \cdot m \cdot \gamma$), and $\gamma' = (\beta_i)_{i \in [1,n-p], \beta_i \neq m}$ the word obtained by removing all letters $m$ from $\beta$.

\subsection{A useful symmetry} \label{symmetry}

Given two integers $n,k \in \mathbb N$, a \emph{composition} of $n$ into $k$ (possibly empty) \emph{parts} is an integer sequence $\lambda = (\lambda_1, \dots, \lambda_k) \in \mathbb N^k$ such that $\sum_{i=1}^k \lambda_i = n$.
Let $n,k \in \mathbb N$ and let $\lambda$ be a composition of $n$ into $k$ parts. We denote by $\W_\lambda \subseteq \W_{n,k}$ the set of words in which each letter $\ell \in [0,k-1]$ appears exactly $\lambda_{\ell+1}$ times (i.e. permutations of the multiset $\{\ell^{\lambda_{\ell+1}} \; : \; \ell \in [0, k-1]\}$). As before, we denote by $\W_\lambda(P)$ the subset of words avoiding some set of patterns $P$. The following theorem is a result of \cite{Atkinson_Linton_Walker} (up to symmetry).
\begin{thm} \label{thmSymmetry}
    Let $k \in \mathbb N$. The function $\mathbb N^k \to \mathbb N, \; (\lambda_1, \dots, \lambda_k) \mapsto |\W_\lambda(120)|$ is symmetric.
\end{thm}
\noindent In other words, if $\lambda, \mu \in \mathbb N^k$ are two sequences such that there exists a permutation $\pi$ of $[1,k]$ which satisfies $\mu_i = \lambda_{\pi(i)}$ for all $i \in [1,k]$, then the identity $|\W_\lambda(120)| = |\W_\mu(120)|$ holds.
This result was generalized in \cite{Albert_Aldred_Atkinson_Handley_Holton} and \cite{Savage_Wilf} to show that the number $|\W_\lambda(120)|$ remains unchanged when the pattern 120 is replaced by any permutation of $\{0,1,2\}$.

\subsection{The pair \{011, 120\}} \label{011+120}
Let $\mathfrak A_{n,m} = \{\sigma \in \I_n(011,120) \; : \; m = \max(\sigma)\}$ be the set of $\{011,120\}$-avoiding inversion sequences of size $n$ and maximum $m$. Let $\mathfrak B_{n,k} = \{ \omega \in \W_{n,k}(120) \; : \; \omega_i = \omega_j \implies i = j\}$ be the set of 120-avoiding words of length $n$ over the alphabet $[0,k-1]$ which do not contain any repeated letter. Let $\mathfrak C_{n,k} = \{\omega \in \W_{n,k} (120) \; : \; \omega_i = \omega_j \neq 0 \implies i = j\}$ be the set of 120-avoiding words of length $n$ over the alphabet $[0,k-1]$ which do not contain any nonzero repeated letter.
Let $\mathfrak a_{n,m} = |\mathfrak A_{n,m}|$, $\mathfrak b_{n,k} = |\mathfrak B_{n,k}|$, and $\mathfrak c_{n,k} = |\mathfrak C_{n,k}|$.

\begin{thm} \label{thm011+120}
For all $0 < m < n$,
$$\mathfrak a_{n,m} = \sum_{p=m+1}^n \mathfrak c_{n-p,m} + \sum_{j=1}^{m-1} \mathfrak a_{p-1,j} \cdot \mathfrak b_{n-p,m-j-1}.$$
\end{thm}
\begin{proof}
    Let $\sigma \in \mathfrak A_{n,m}$, $p = \firstmax(\sigma)$, $\alpha = (\sigma_i)_{i \in [1,p-1]}$, $\gamma = (\sigma_i)_{i \in [p+1,n]}$, and let $j = \max(\alpha)$. In particular, $j < m$.
    \begin{itemize}
        \item If $j = 0$, then $\alpha = (0)^{p-1}$, and $\gamma \in \mathfrak C_{n-p,m}$ since $\sigma_1 = 0$ implies that any nonzero repeated letter creates an occurrence of the pattern 011.
        \item If $j \geqslant 1$, then $\alpha \in \mathfrak A_{p-1,j}$. By Proposition \ref{propsplit120}, $\gamma$ is a 120-avoiding word of length $n-p$ over the alphabet $[j+1,m-1]$ (since repeating $j$ or $m$ would create an occurrence of 011) in which all letters are distinct (to avoid 011). Subtracting $j+1$ from every letter in $\gamma$ yields a word in $\mathfrak B_{n-p,m-j-1}$. \qedhere
    \end{itemize}
    Observing that these two constructions are bijective concludes the proof.
\end{proof}

Let us now enumerate the families of words $(\mathfrak B_{n,k})_{n,k \in \mathbb N}$ and $(\mathfrak C_{n,k})_{n,k \in \mathbb N}$.
\begin{prop} \label{prop011+120}
    For all $n,k \geqslant 0$,
    $$\mathfrak b_{n,k} = \frac{\binom{k}{n} \binom{2n}{n} }{n+1}.$$
\end{prop}
\begin{proof}
    There are $\binom{k}{n}$ possible choices for the set of $n$ distinct letters in $[0,k-1]$ which appear in a word of $\mathfrak B_{n,k}$. Once this set of letters is decided, choosing their order amounts to choosing a 120-avoiding permutation of size $n$, counted by the Catalan number ${C_n = \frac{1}{n+1} \binom{2n}{n}}$.
\end{proof}

\begin{lem} \label{lem011+120}
For all $n \geqslant 1, k \geqslant 2$,
$$\mathfrak c_{n,k} = 2 \mathfrak c_{n,k-1} + \mathfrak c_{n-1,k} - \mathfrak c_{n-1,k-1} - \frac{\binom{k-2}{n} \binom{2n}{n}}{n+1}.$$
\end{lem}
\begin{proof}
    Let $n \geqslant 0, k \geqslant 2$, and $\omega \in \mathfrak C_{n,k}$. We consider three different cases, based on the position of the letter $1$ in $\omega$.
    \begin{itemize}
        \item If $\omega$ does not contain the letter 1, then subtracting 1 from every nonzero letter in $\omega$ yields a word in $\mathfrak C_{n,k-1}$, and this is a bijection. As such, the words of $\mathfrak C_{n,k}$ which do not contain the letter 1 are counted by $\mathfrak c_{n,k-1}$.
        \item Otherwise, $\omega$ contains exactly one occurrence of the letter 1. This implies that $n \geqslant 1$ and $k \geqslant 2$.
        \begin{itemize}
            \item If there is any occurrence of the letter 0 appearing to the right of the letter 1 in $\omega$, then every such occurrence must be in a single factor adjacent to the letter 1 (otherwise, $\omega$ would contain the pattern 120). Hence, removing the last letter 0 in $\omega$ can yield any word in $\mathfrak C_{n-1,k}$ which contains the letter 1, and this is a bijection. We know from the previous case that the words of $\mathfrak C_{n-1,k}$ which do not contain the letter 1 are counted by $\mathfrak c_{n-1,k-1}$, hence the number of possible words $\omega$ in this case is $\mathfrak c_{n-1,k} - \mathfrak c_{n-1,k-1}$.
            \item Otherwise, every letter 0 in $\omega$ appears to the left of the letter 1. Subtracting 1 from every nonzero letter in $\omega$ then yields a word in $\mathfrak C_{n,k-1}$ which contains the letter 0. Note this defines a bijection: its inverse takes any word in $\mathfrak C_{n,k-1}$ which contains the letter 0, adds 1 to every nonzero letter, and replaces the last letter 0 by the letter 1. Since the words of $\mathfrak C_{n,k-1}$ which do not contain the letter 0 are trivially in bijection with $\mathfrak B_{n,k-2}$, we find there are $\mathfrak c_{n,k-1} - \mathfrak b_{n,k-2}$ possible choices for the word $\omega$ in this case. \qedhere
        \end{itemize}
    \end{itemize}
\end{proof}
We also give an explicit expression for $\mathfrak c_{n,k}$ in Lemma \ref{lemExpl100+120}.

\subsection{The pair \{100, 120\}}
Let $\mathfrak A_{n,m} = \{\sigma \in \I_n(100,120) \; : \; m = \max(\sigma)\}$ be the set of $\{100,120\}$-avoiding inversion sequences of size $n$ and maximum $m$. Let $\mathfrak B_{n,k} = \{\omega \in \W_{n,k} (120) \; : \; \omega_i = \omega_j \neq k-1 \implies i = j\}$ be the set of 120-avoiding words of length $n$ over the alphabet $[0,k-1]$ which do not contain any non-maximal repeated letter.
Let $\mathfrak a_{n,m} = |\mathfrak A_{n,m}|$, and $\mathfrak b_{n,k} = |\mathfrak B_{n,k}|$.
\begin{thm} \label{thm100+120}
    For all $0 < m < n$,
    $$\mathfrak a_{n,m} = \sum_{p = m+1}^n \sum_{j = 0}^{m-1} \mathfrak a_{p-1,j} \cdot \mathfrak b_{n-p,m-j+1}.$$
\end{thm}
\begin{proof}
    Let $\sigma \in \mathfrak A_{n,m}$, $p = \firstmax(\sigma)$, $\alpha = (\sigma_i)_{i \in [1,p-1]}$, $\gamma = (\sigma_i)_{i \in [p+1,n]}$, and let ${j = \max(\alpha)}$. In particular, $\alpha \in \mathfrak A_{p-1,j}$. By Proposition \ref{propsplit120}, $\gamma \in [j,m]^{n-p}$. Since $\sigma$ avoids 100, $\gamma$ cannot contain any repeated value lower than $m$. Subtracting $j$ from each letter of $\gamma$ yields a word in $\mathfrak B_{n-p,m-j+1}$.
\end{proof}

As in Section \ref{011+120}, let $\mathfrak C_{n,k} = \{\omega \in \W_{n,k} (120) \; : \; \omega_i = \omega_j \neq 0 \implies i = j\}$ be the set of 120-avoiding words of length $n$ over the alphabet $[0,k-1]$ which do not contain any nonzero repeated letter, and let $\mathfrak c_{n,k} = |\mathfrak C_{n,k}|$. The number of words $\mathfrak c_{n,k}$ was counted in Lemma \ref{lem011+120}.
\begin{lem}
    For all $n,k \geqslant 0,$ $\mathfrak b_{n,k} = \mathfrak c_{n,k}$.
\end{lem}
\begin{proof}
Let $\mathfrak B_{n,k,r} = \{ \omega \in \mathfrak B_{n,k} \; : \; r = |\{i \in [1,n] \; : \; \omega_i = k-1\}|\}$ be the subset of $\mathfrak B_{n,k}$ of words in which the largest letter $k-1$ appears exactly $r$ times. Likewise, let $\mathfrak C_{n,k,r} = \{ \omega \in \mathfrak C_{n,k} \; : \; r = |\{i \in [1,n] \; : \; \omega_i = 0\}|\}$ be the subset of $\mathfrak C_{n,k}$ of words in which the letter 0 appears exactly $r$ times. By definition, we have
$$\mathfrak B_{n,k} = \coprod_{r = 0}^n \mathfrak B_{n,k,r} \quad \text{and} \quad \mathfrak C_{n,k} = \coprod_{r = 0}^n \mathfrak C_{n,k,r}.$$
Let $\mathfrak D_{n,k} = \{\lambda \in \{0,1\}^k \; : \; n = \sum_{i = 1}^k \lambda_i\}$ be the set of compositions of $n$ into $k$ parts of size 0 or 1.
In the notation of Section \ref{symmetry}, for all $r \geqslant 0$, 
$$\mathfrak B_{n,k,r} = \coprod_{\lambda \in \mathfrak D_{n-r,k-1}} \W_{\lambda \cdot r}(120) \quad \text{and} \quad \mathfrak C_{n,k,r} = \coprod_{\lambda \in \mathfrak D_{n-r,k-1}} \W_{r \cdot \lambda}(120).$$
By Theorem \ref{thmSymmetry}, for all $r \geqslant 0$ and $\lambda \in \mathfrak D_{n-r,k-1}$, $|\W_{\lambda \cdot r}(120)| = |\W_{r \cdot \lambda}(120)|$. This implies that $|\mathfrak B_{n,k,r}| = |\mathfrak C_{n,k,r}|$ for all $n,k,r \geqslant 0$, hence $\mathfrak b_{n,k} = \mathfrak c_{n,k}$.
\end{proof}

We also provide an explicit expression of $\mathfrak b_{n,k}$.
\begin{lem} \label{lemExpl100+120}
For all $n \geqslant 0, k \geqslant 1$,
$$\mathfrak b_{n,k} = \frac{1}{n+1} \sum_{d=0}^{\min(n,k)} \binom{k-1}{d} \binom{n+d}{n} (n-d+1).$$ 
\end{lem}
\begin{proof}
Let $\mathfrak E_{n,k} = \{\omega \in \mathfrak B_{n,k+1} \; : \; [0,k-1] \subseteq \vals(\omega) \}$ be the subset of words of $\mathfrak B_{n,k+1}$ which contain all letters $[0,k-1]$ (and may contain the letter $k$). In particular, for all $n \geqslant 0$, $\mathfrak E_{n,k}$ is empty if $k > n$, and $\mathfrak E_{n,0} = \{(0)^n\}$. Note that any word of $\mathfrak E_{n,k}$ contains exactly one occurrence of each letter in $[0,k-1]$, and $n-k$ occurrences of the letter $k$. Let $\mathfrak e_{n,k} = |\mathfrak E_{n,k}|$.

The same construction as in Proposition \ref{propbinomwords} gives the following identity, seeing that $\binom{k-1}{d}$ is the number of increasing functions $[0,d] \to [0,k-1]$ such that $d \mapsto k-1$.
For all  $n \geqslant 0, k \geqslant 1$,
$$\mathfrak b_{n,k} = \sum_{d=0}^{\min(n,k)} \binom{k-1}{d} \mathfrak e_{n,d}.$$

Let us now count the words of $\mathfrak E_{n,k}$.
Let $1 \leqslant k \leqslant n$, and $\omega \in \mathfrak E_{n,k}$.
    \begin{itemize}
        \item If $\omega_n = k$, then removing $\omega_n$ yields a word in $\mathfrak E_{n-1,k}$ and this is a bijection.
        \item If $\omega_n < k$, then no letter $k$ may appear after the letter $k-1$ in $\omega$, otherwise the subsequence $(k-1,k,\omega_n)$ would be an occurrence of 120. Replacing every letter $k$ in $\omega$ by $k-1$ yields a word $\omega' \in \mathfrak E_{n,k-1}$ (since each letter in $[0,k-2]$ still appears exactly once in $\omega'$), and this is a bijection. Indeed, it can be reversed by taking any word $\omega' \in \mathfrak E_{n,k-1}$ and replacing every letter $k-1$ except the last one by the letter $k$ (notice that $\omega' \in \mathfrak E_{n,k-1}$ must contain the letter $k-1$ since $k \leqslant n$).
    \end{itemize}
    Hence, $(\mathfrak e_{n,k})_{n,k \geqslant 0}$ satisfies the recurrence relation $\mathfrak e_{n,k} = \mathfrak e_{n-1,k} + \mathfrak e_{n,k-1}$ for all $1 \leqslant k \leqslant n$, with initial conditions $\mathfrak e_{n,0} = 1$ for all $n \geqslant 0$ and $\mathfrak e_{0,k} = 0$ for all $k \geqslant 1$.
    From this recurrence, we can easily prove by induction that for all $0 \leqslant k \leqslant n$,
    $$\mathfrak e_{n,k} = \frac{\binom{n+k}{n} (n-k+1)}{n+1},$$ which shows that $(\mathfrak e_{n,k})_{n,k \geqslant 0}$ is Catalan's triangle (entry A009766 of the OEIS).
\end{proof}

\subsection{The pair \{120, 201\}}
Let $\mathfrak A_{n,m} = \{\sigma \in \I_n(120,201) \; : \; m = \max(\sigma)\}$ be the set of $\{120,201\}$-avoiding inversion sequences of size $n$ and maximum $m$. Let $\mathfrak B_{n,k} = \{\omega \in \W_{n,k+1}(120) \; : \; \omega_i < \omega_j \neq k \implies j < i\}$ be the set of 120-avoiding words of length $n$ over the alphabet $[0,k]$ whose non-maximal letters appear in nonincreasing order. Let $\mathfrak a_{n,m} = |\mathfrak A_{n,m}|$, and $\mathfrak b_{n,k} = |\mathfrak B_{n,k}|$.
\begin{thm} \label{thm120+201}
For all $0 < m < n$,
$$\mathfrak{a}_{n,m} = \sum_{p = m+1}^n \sum_{j = 0}^{m-1} \mathfrak{a}_{p-1,j} \cdot \mathfrak{b}_{n-p, m-j}.$$
\end{thm}
\begin{proof}
     Let $\sigma \in \mathfrak A_{n,m}$, $p = \firstmax(\sigma)$, $\alpha = (\sigma_i)_{i \in [1,p-1]}$, $\gamma = (\sigma_i)_{i \in [p+1,n]}$, and let $j = \max(\alpha)$. In particular, $\alpha \in \mathfrak A_{p-1,j}$. By Proposition \ref{propsplit120}, $\gamma \in [j,m]^{n-p}$.
     Since $\sigma$ avoids 201 and $\sigma_p = m$, all letters different from $m$ in $\gamma$ appear in nonincreasing order. By subtracting $j$ from all letters in $\gamma$, we obtain a word in the set $\mathfrak B_{n-p,m-j}$. We easily observe that this construction is a bijection.
\end{proof}
\begin{rem} \label{remnonincwords}
    For all $n \geqslant 0$ and $k \geqslant 1$, there are $\binom{k+n-1}{n}$ nonincreasing words of length $n$ over the alphabet $[0,k-1]$. Indeed, the set of nonincreasing words of length $n$ over the alphabet $[0,k-1]$ is in bijection with the set of compositions of $n$ into $k$ (possibly empty) parts. More precisely, for each composition of $n$ into $k$ parts $\lambda$, there is a single nonincreasing word in the set $\W_\lambda$ (using the notation from Section $\ref{symmetry}$).
\end{rem}
\begin{lem}
    For all $n,k \geqslant 1$,
    $$\mathfrak b_{n,k} = \binom{k+n+1}{n} + k \left (2^n - 1 - \frac{n(n+1)}{2} \right ) - n.$$
\end{lem}
\begin{proof}
    Let $n,k \geqslant 1$, and $\omega \in \mathfrak B_{n,k}$.
    \begin{itemize}
        \item If the letter $k$ does not appear in $\omega$, then $\omega$ can be any nonincreasing word of length $n$ over the alphabet $[0,k-1]$. By Remark \ref{remnonincwords}, there are $\binom{k+n-1}{n}$ such words.
        \item Otherwise, let $p = \firstmax(\omega)$ be the position of the first letter $k$ in $\omega$.
        \begin{itemize}
            \item If $p = 1$, then $\omega_1 = k$ and $(\omega_i)_{i \in [2,n]}$ can be any word of $\mathfrak B_{n-1,k}$.

            \item If $p \geqslant 2$, let $\alpha = (\omega_i)_{i \in [1,p-1]}$ and $\gamma = (\omega_i)_{i \in [p+1,n]}$. In particular, $\alpha$ is a nonincreasing word of length $p-1$ over the alphabet $[0,k-1]$. By Remark \ref{remnonincwords}, there are $\binom{k+p-2}{p-1}$ possible words $\alpha$. Note that any letter $v \in \vals(\gamma)$ such that $v \neq k$ must satisfy both $v \leqslant \min(\alpha)$ (by the nonincreasing property) and $v \geqslant \max(\alpha)$ (to avoid 120). We distinguish two cases:
            \begin{itemize}
                \item If $\alpha$ is constant, then $\alpha = (v)^{p-1}$ for some $v \in [0,k-1]$, and $\gamma \in \{v, k\}^{n-p}$. In that case, there are $k$ possible choices for $\alpha$, and for each $\alpha$ there are $2^{n-p}$ possible choices for $\gamma$.
                \item Otherwise, there are $\binom{k+p-2}{p-1} - k$ remaining choices for $\alpha$. Since $\alpha$ is not constant, $\min(\alpha) < \max(\alpha)$, hence no letter in $[0,k-1]$ may appear in $\gamma$, so $\gamma = (k)^{n-p}$.
            \end{itemize}
        \end{itemize}
    \end{itemize}
    We obtain the following equation for all $n,k \geqslant 1$:
    \begin{align*}
        \mathfrak b_{n,k} &= \binom{k+n-1}{n} + \mathfrak b_{n-1,k} + \sum_{p=2}^n \left (k \cdot 2^{n-p} + \binom{k+p-2}{p-1} - k \right ) \\
        &= \binom{k+n-1}{n} + \mathfrak b_{n-1,k} + k (2^{n-1} - 1) + \binom{k+n-1}{n-1} - 1 - k (n-1) \\
        &= \mathfrak b_{n-1,k} + \binom{k+n}{n} + k (2^{n-1} - n) - 1.
    \end{align*}
    We conclude the proof with a telescoping sum. For all $n,k \geqslant 1$,
    \begin{align*}
        \mathfrak b_{n,k} &= \mathfrak b_{0,k} + \sum_{i=1}^n \left (\mathfrak b_{i,k} - \mathfrak b_{i-1,k} \right ) \\
        &= 1 + \sum_{i=1}^n \left ( \binom{k+i}{i} + k (2^{i-1} - i) - 1 \right ) \\
        &= \binom{k+n+1}{n} + k \left (2^n - 1 - \frac{n(n+1)}{2} \right ) - n. \qedhere
    \end{align*}
\end{proof}

\subsection{The pair \{110, 120\}}
Let $\mathfrak A_{n,m} = \{\sigma \in \I_n(110,120) \; : \; m = \max(\sigma)\}$ be the set of $\{110,120\}$-avoiding inversion sequences of size $n$ and maximum $m$. Let $\mathfrak B_{n,k} = \coprod_{\ell \in [0,n]} \W_{\ell,k} (110,120)$ be the set of $\{110,120\}$-avoiding words of length at most $n$ over the alphabet $[0,k-1]$.
Let $\mathfrak a_{n,m} = |\mathfrak A_{n,m}|$, and $\mathfrak b_{n,k} = |\mathfrak B_{n,k}|$.
\begin{thm} \label{thm110+120}
    For all $0 < m < n$,
    $$\mathfrak a_{n,m} = \sum_{p=m+1}^n \sum_{j=0}^{m-1} \mathfrak a_{p-1,j} \cdot \mathfrak b_{n-p,m-j}.$$
\end{thm}
\begin{proof}
    Let $\sigma \in \mathfrak A_{n,m}$, $p = \firstmax(\sigma)$, $\alpha = (\sigma_i)_{i \in [1,p-1]}$, $\gamma = (\sigma_i)_{i \in [p+1,n]}$, and let $j = \max(\alpha)$. In particular, $\alpha \in \mathfrak A_{p-1,j}$. By Proposition \ref{propsplit120}, $\gamma \in [j,m]^{n-p}$. Since $\sigma$ avoids 110, if $\gamma_q = m$ for some $q \in [1,n-p]$ then $\gamma_i = m$ for all $i \in [q,n-p]$. In other words, all occurrences of $m$ in $\gamma$ (if any) must appear in a single factor at the end of $\gamma$. Let $k \in [0,n-p]$ be the number of occurrences of $m$ in $\gamma$, and $\gamma' = (\gamma_i)_{i \in [1,n-p-k]}$ be the word obtained by removing all letters $m$ from $\gamma$. By subtracting $j$ from all letters in $\gamma'$, we obtain a word in the set $\mathfrak B_{n-p,m-j}$. We easily observe that this construction is a bijection.
\end{proof}
Let $\mathfrak C_{n,k} = \overline \W_{n,k} (110,120)$ be the set of \{110, 120\}-avoiding words of length $n$ which contain all letters of the alphabet $[0,k-1]$. Let $\mathfrak c_{n,k} = |\mathfrak C_{n,k}|$. By definition of $\mathfrak B_{n,k}$ and Proposition \ref{propbinomwords}, for all $n,k \geqslant 0$,
$$\mathfrak b_{n,k} = \sum_{\ell = 0}^n \sum_{d=0}^{\min(\ell,k)} \binom{k}{d} \mathfrak c_{\ell,d}.$$

\begin{lem} \label{lem110+120}
For all $1 \leqslant k \leqslant n$,
$$\mathfrak{c}_{n,k} = \frac{1}{k} \binom{n-1}{k-1} \binom{n+k}{k-1}.$$
\end{lem}
\begin{proof}
For all $0 \leqslant k \leqslant n$ and $\omega \in \mathfrak C_{n,k}$, let $\dec(\omega)$ be the size of the longest strictly decreasing factor at the beginning of $\omega$. In particular, if $n > 0$ then $\omega_{\dec(\omega)}$ is the first letter 0 in $\omega$ (otherwise the first 0 would appear after a weak increase, creating an occurrence of 110 or 120). Also note that $\dec(\omega) \leqslant k$ since all letters of a strictly decreasing factor are distinct. Let $\mathfrak C_{n,k,s} = \{\omega \in \mathfrak C_{n,k} \; : \; s = \dec(\omega)\}$, and $\mathfrak c_{n,k,s} = |\mathfrak C_{n,k,s}|$.

We consider the combinatorial generating tree for $\coprod_{n \geqslant k \geqslant 0} \mathfrak C_{n,k}$ obtained by inserting letters in decreasing order of value, and inserting repeats of a letter from right to left. This amounts to inserting the letter 0 before the first 0 in a word $\omega \in \mathfrak C_{n,k}$, or inserting the letter 0 in $\omega+1$.
This indeed defines a generating tree since any nonempty word $\omega$ has a unique parent of size $|\omega|-1$, obtained by removing the first 0 from $\omega$, and subtracting 1 from every letter if the resulting word does not contain the letter 0.

For any word $\omega \in \mathfrak C_{n,k}$, the positions where the letter 0 may be inserted before the first 0 in $\omega$ are $[1,\dec(\omega)]$, and this clearly cannot create an occurrence of 110 or 120. 
The positions where the letter 0 may be inserted in $\omega+1$ without creating an occurrence of 110 or 120 are $[1,\dec(\omega)+1]$.
This implies that each word of $\mathfrak C_{n,k,s}$ has one child in $\mathfrak C_{n+1,k,i}$ for all $i \in [1,s]$, and one child in $\mathfrak C_{n+1,k+1,i}$ for all $i \in [1,s+1]$. We obtain the following equation for all $2 \leqslant i \leqslant k \leqslant n$:
$$\mathfrak c_{n,k,i} = \sum_{s = i}^{n-1} \mathfrak c_{n-1,k,s} + \sum_{s = i-1}^{n-1} \mathfrak c_{n-1,k-1,s}.$$
In particular, $\mathfrak c_{n,k,i} - \mathfrak c_{n,k,i+1} = \mathfrak c_{n-1,k,i} + \mathfrak c_{n-1,k-1,i-1}$, so a simpler recurrence relation holds: for all $2 \leqslant s \leqslant k \leqslant n$,
$$\mathfrak c_{n,k,s} = \mathfrak c_{n,k,s+1} + \mathfrak c_{n-1,k,s} + \mathfrak c_{n-1,k-1,s-1}.$$
From this equation, we can easily prove by induction that for all $1 \leqslant s \leqslant k \leqslant n$,
$$\mathfrak c_{n,k,s} = \frac{s (n+k-s-1)!}{(n-k)! k! (k-s)!}.$$

Finally, the expression of $\mathfrak c_{n,k}$ is obtained by summing over $s$. For all $1 \leqslant k \leqslant n$,
\begin{align*}
    \mathfrak c_{n,k} &= \sum_{s=1}^k \mathfrak c_{n,k,s} \\
    &= \sum_{s=1}^k \frac{s (n+k-s-1)!}{(n-k)! k! (k-s)!} \\
    &= \frac{1}{k} \binom{n-1}{k-1} \sum_{s=1}^k s \binom{n+k-s-1}{n-1} \\
    &= \frac{1}{k} \binom{n-1}{k-1} \sum_{i=1}^k \sum_{s=i}^k \binom{n+k-s-1}{n-1} \\
    &= \frac{1}{k} \binom{n-1}{k-1} \sum_{i=1}^k \binom{n+k-i}{n} \\
    &= \frac{1}{k} \binom{n-1}{k-1} \binom{n+k}{k-1}. \qedhere
\end{align*}
\end{proof}

\subsection{The pair \{010, 120\}} \label{010+120}
Let $\mathfrak A_{n,m} = \{\sigma \in \I_n(010,120) \; : \; m = \max(\sigma)\}$ be the set of $\{010,120\}$-avoiding inversion sequences of size $n$ and maximum $m$. Let $\mathfrak B_{n,k} = \W_{n,k} (010,120)$ be the set of \{010, 120\}-avoiding words of length $n$ over the alphabet $[0,k-1]$. Let $\mathfrak a_{n,m} = |\mathfrak A_{n,m}|$, and $\mathfrak b_{n,k} = |\mathfrak B_{n,k}|$.
\begin{thm} \label{thm010+120}
For all $0 < m < n$,
$$\mathfrak{a}_{n,m} = \sum_{p = m+1}^n \sum_{j = 0}^{m-1} \mathfrak{a}_{p-1,j} \cdot \mathfrak{b}_{n-p, m-j}.$$
\end{thm}
\begin{proof}
    It follows from Remarks \ref{remsplit} and \ref{remsplitpattern} and Propositions \ref{propsplit120} and \ref{propsplit010} that for $0 < m < n$, $\mathfrak{A}_{n,m}$ is the set of inversion sequences $\alpha \cdot m \cdot \gamma$ such that:
    \begin{itemize}
        \item $\alpha \in \mathfrak{a}_{p-1,j}$ for some $p \in [m+1,n]$ and $j \in [0,m-1]$,
        \item $\gamma \in [0,m]^{n-p}$ avoids 010 and 120 and satisfies $\min(\gamma) > \max(\alpha) = j$.
    \end{itemize}
    In other words, $\gamma$ is a \{010, 120\}-avoiding word of length $n-p$ over the alphabet $[j+1,m]$, and the set of such words is clearly in bijection with $\mathfrak B_{n-p, m-j}$.
\end{proof}
Let $\mathfrak C_{n,k} = \overline \W_{n,k} (010,120)$ be the set of \{010, 120\}-avoiding words of length $n$ which contain all letters of the alphabet $[0,k-1]$. Let $\mathfrak c_{n,k} = |\mathfrak C_{n,k}|$. By Proposition \ref{propbinomwords}, for all $n,k \geqslant 0$,
$$\mathfrak b_{n,k} = \sum_{d=0}^{\min(n,k)} \binom{k}{d} \mathfrak c_{n,d}.$$
\begin{lem} \label{lem010+120}
For all $n,k \geqslant 1$,
$$\mathfrak{c}_{n,k} = \frac{1}{k} \binom{n-1}{k-1} \binom{n+k}{k-1}.$$
\end{lem}
\begin{proof}
We consider the combinatorial generating tree for $\coprod_{n \geqslant k \geqslant 0} \mathfrak C_{n,k}$ obtained by inserting letters in increasing order of value, and inserting repeats of a letter from right to left. In other words, at each step of the construction, we either insert a letter $k$ in a word $\omega \in \mathfrak C_{n,k}$, or insert a letter $k-1$ before the first $k-1$ in $\omega$.

For all $1 \leqslant k \leqslant n$ and $\omega \in \mathfrak C_{n,k}$, let
$$\Sites(\omega) = \{q \in [1, \firstmax(\omega)] \; : \; \max((\omega_i)_{i \in [1,q-1]}) < \min((\omega_i)_{i \in [q,n]})\}$$
be the set of positions where the letter $k-1$ may be inserted before the first letter $k-1$ in $\omega$ without creating an occurrence of 010 or 120. 
The set of positions where the letter $k$ may be inserted in $\omega$ without creating an occurrence of 010 or 120 is $\Sites(\omega) \sqcup \{n+1\}$. For all $1 \leqslant s \leqslant k \leqslant n$, let $\mathfrak C_{n,k,s} = \{\omega \in \mathfrak C_{n,k} \; : \; s = |\Sites(\omega)|\}$.

Let $1 \leqslant s \leqslant k \leqslant n$ and $\omega \in \mathfrak C_{n,k,s}$. Let $q_1 < q_2 < \dots < q_s$ be the elements of $\Sites(\omega)$. Then for all $i \in [1,s]$, inserting the letter $k-1$ at position $q_i$ yields a word of $\mathfrak C_{n+1,k,i}$, and inserting the letter $k$ at position $q_i$ yields a word of $\mathfrak C_{n+1,k+1,i}$. Additionally, inserting the letter $k$ at position $n+1$ yields a word in $\mathfrak C_{n+1,k+1,s+1}$. Hence the following equation holds: for all $2 \leqslant i \leqslant k \leqslant n$,
$$\mathfrak c_{n,k,i} = \sum_{s = i}^{n-1} \mathfrak c_{n-1,k,s} + \sum_{s = i-1}^{n-1} \mathfrak c_{n-1,k-1,s}.$$
The same recurrence relation was found in the proof of Lemma \ref{lem110+120}, with the same initial condition $\mathfrak c_{1,1,1} = 1$. This exhibits a bijection between the sets of words $\overline \W_{n,k}(010, 120)$ and $\overline \W_{n,k}(110, 120)$ (which trivially extends to a bijection between $\W_{n,k}(010, 120)$ and $\W_{n,k}(110, 120)$), and concludes our proof.
\end{proof} 

\subsection{The pair \{101, 120\}} \label{101+120}

Let $\mathfrak A_{n,m} = \{\sigma \in \I_n(101, 120) \; : \; m = \max(\sigma)\}$ be the set of $\{101, 120\}$-avoiding inversion sequences of size $n$ and maximum $m$. Since the sequences in $\mathfrak A_{n,m}$ avoid 101, all occurrences of their maximum $m$ are consecutive. Let $\mathfrak A'_{n,m} = \{\sigma \in \mathfrak A_{n,m} \; : \; \sigma_n = m\}$ be the subset of $\mathfrak A_{n,m}$ of sequences whose last value is equal to their maximum.
Let $\mathfrak B_{n,k} = \coprod_{\ell \in [0,n]} \W_{\ell,k}(101,120)$ be the set of $\{101, 120\}$-avoiding words of length at most $n$ over the alphabet $[0,k-1]$. Let $\mathfrak a_{n,m} = |\mathfrak A_{n,m}|$, $\mathfrak a'_{n,m} = |\mathfrak A'_{n,m}|$, and $\mathfrak b_{n,k} = |\mathfrak B_{n,k}|$.
\begin{thm} \label{thm101+120}
For all $0 < m < n$,
$$\mathfrak{a}_{n,m} = \sum_{p = m+1}^n \sum_{j = 0}^{m-1} \mathfrak{a}'_{p-1,j} \cdot \mathfrak{b}_{n-p, m-j} + (\mathfrak{a}_{p-1,j} - \mathfrak{a}'_{p-1,j}) \cdot \mathfrak{b}_{n-p, m-j-1},$$
$$\mathfrak{a}'_{n,m} = \sum_{p = m+1}^{n} \sum_{j = 0}^{n-1} \mathfrak{a}_{p-1,j}.$$
\end{thm}
\begin{proof}
     Let $\sigma \in \mathfrak A_{n,m}$, $p = \firstmax(\sigma)$, $\alpha = (\sigma_i)_{i \in [1,p-1]}$, $\gamma = (\sigma_i)_{i \in [p+1,n]}$, and let $j = \max(\alpha)$. In particular, $\alpha \in \mathfrak A_{p-1,j}$ and $\gamma \in [j,m]^{n-p}$. Let $\ell = \lastmax(\sigma) \in [p,n]$ be the position of the last value $m$ in $\sigma$. Since $\sigma$ avoids 101, all occurrences of the maximum $m$ are consecutive, hence their positions are $[p, \ell]$. Let $\gamma' = (\sigma_i)_{i \in [\ell+1,n]}$ be the word obtained by removing all letters $m$ from $\gamma$.
     \begin{itemize}
         \item If $\alpha \in \mathfrak A'_{p-1,j}$, then $\gamma'$ can be any $\{101,120\}$-avoiding word of length $n-\ell \in [0,n-p]$ over the alphabet $[j,m-1]$. Subtracting $j$ from all letters in $\gamma'$ yields any word in the set $\mathfrak B_{n-p, m-j}$.
         \item If $\alpha \in \mathfrak A_{p-1,j} \backslash \mathfrak A'_{p-1,j}$, then $\sigma_{p-1} < j$. The subsequence $(j, \sigma_{p-1})$ is an occurrence of the pattern 10, which implies that the value $j$ does not appear in $\gamma'$ (otherwise $\sigma$ would contain 101). In that case, $\gamma'$ is a word over the alphabet $[j+1,m-1]$, and the number of possible words $\gamma'$ is $\mathfrak b_{n-p,m-j-1}$.
     \end{itemize}

     Finally, if $\sigma \in \mathfrak A'_{n,m}$ then $\gamma = (m)^{n-p}$ and $\alpha$ can be any sequence in $\mathfrak A_{p-1,j}$. This yields the second equation.
\end{proof}

It remains to count the words of $\mathfrak B_{n,k}$. Let $F(x,y) = \sum_{\ell,k \geqslant 0} |\W_{\ell,k}(101,120)| x^\ell y^k$ be the ordinary generating function of $\{101,120\}$-avoiding words.
The enumeration of $\{101,120\}$-avoiding words was solved in \cite{Mansour_2006}, with the following expression:
$$F(x,y) = \frac{(1-x)^2 - (1-2x)y - \sqrt{(1-x)^4 - 2(1-x)^2 y + (1-4x^2+4x^3)y^2}}{2xy(1-y)}.$$
The generating function of $\mathfrak B_{n,k}$ is 
\begin{align*}
    B(x,y) &= \sum_{n,k \geqslant 0} \mathfrak b_{n,k} x^n y^k \\
    &= \sum_{n,k \geqslant 0} \sum_{\ell \in [0,n]} |\W_{\ell,k}(101,120)| x^n y^k \\
    &= \sum_{\ell,k \geqslant 0} |\W_{\ell,k}(101,120)| \sum_{n \geqslant \ell} x^n y^k \\
    &= \sum_{\ell,k \geqslant 0} |\W_{\ell,k}(101,120)| \frac{x^\ell y^k}{1-x} \\
    &= \frac{F(x,y)}{1-x}.
\end{align*}

\subsection{The pair \{000, 120\}} \label{000+120}

Given a sequence $\sigma \in \mathbb N^n$ and an integer $v \in \mathbb N$, let $\occ(\sigma, v) = |\{i \in [1,n] \; : \; \sigma_i = v\}|$ be the number of occurrences of the value $v$ in $\sigma$.

Let $\mathfrak A_{n,m,r} = \{\sigma \in \I_n(000, 120) \; : \; m = \max(\sigma), \, r = \occ(\sigma, m) \}$ be the set of $\{000, 120\}$-avoiding inversion sequences of size $n$, maximum $m$, and $r$ occurrences of $m$.
Let $\mathfrak B_{n,k} = \W_{n,k}(000,120)$ be the set of $\{000, 120\}$-avoiding words of length $n$ over the alphabet $[0,k-1]$, and $\mathfrak C_{n,k} = \{ \omega \in \mathfrak B_{n,k} \; : \; \occ(\omega, k-1) = 1\}$ be the subset of words of $\mathfrak B_{n,k}$ in which the largest letter $k-1$ appears exactly once. We also consider the subsets of $\mathfrak B_{n,k}$ and $\mathfrak C_{n,k}$ of words in which the letter 0 appears less than twice: let $\mathfrak B_{n,k,\cancel{00}} = \{ \omega \in \mathfrak B_{n,k} \; : \; \occ(\omega, 0) < 2\}$ and $\mathfrak C_{n,k,\cancel{00}} = \{ \omega \in \mathfrak C_{n,k} \; : \; \occ(\omega, 0) < 2\}$. Let $\mathfrak a_{n,m,r} = |\mathfrak A_{n,m,r}|$, $\mathfrak b_{n,k} = |\mathfrak B_{n,k}|$, $\mathfrak c_{n,k} = |\mathfrak C_{n,k}|$, $\mathfrak b_{n,k,\cancel{00}} = |\mathfrak B_{n,k,\cancel{00}}|$, and $\mathfrak c_{n,k,\cancel{00}} = |\mathfrak C_{n,k,\cancel{00}}|$.

\begin{thm} \label{thm000+120}
For all $0 < m < n$,
$$\mathfrak{a}_{n,m,1} = \sum_{p = m+1}^n \sum_{j = 0}^{m-1} \mathfrak{a}_{p-1,j,1} \cdot \mathfrak{b}_{n-p, m-j, \cancel{00}} + \mathfrak{a}_{p-1,j,2} \cdot \mathfrak{b}_{n-p, m-j-1},$$
$$\mathfrak{a}_{n,m,2} = \sum_{p = m+1}^n \sum_{j = 0}^{m-1} \mathfrak{a}_{p-1,j,1} \cdot \mathfrak{c}_{n-p, m-j+1, \cancel{00}} + \mathfrak{a}_{p-1,j,2} \cdot \mathfrak{c}_{n-p, m-j}.$$
\end{thm}
\begin{proof}
     Let $\sigma \in \mathfrak A_{n,m}$, $p = \firstmax(\sigma)$, $\alpha = (\sigma_i)_{i \in [1,p-1]}$, $\gamma = (\sigma_i)_{i \in [p+1,n]}$, and let $j = \max(\alpha)$. In particular, $\alpha \in \mathfrak A_{p-1,j}$ and $\gamma \in [j,m]^{n-p}$ by Proposition \ref{propsplit120}.
     \begin{itemize}
         \item If $\sigma \in \mathfrak A_{n,m,1}$, then the letter $m$ cannot appear in $\gamma$, so $\gamma \in [j,m-1]^{n-p}$.
         \begin{itemize}
             \item If $\alpha \in \mathfrak A_{p-1,j,1}$ then the letter $j$ appears at most once in $\gamma$ (to avoid the pattern 000). By subtracting $j$ from every letter in $\gamma$, we obtain a word in the set $\mathfrak B_{n-p,m-j, \cancel{00}}$.
             \item If $\alpha \in \mathfrak A_{p-1,j,2}$ then the letter $j$ cannot appear in $\gamma$. By subtracting $j+1$ from every letter in $\gamma$, we obtain a word in the set $\mathfrak B_{n-p,m-j-1}$.
         \end{itemize}
         \item If $\sigma \in \mathfrak A_{n,m,2}$, then the letter $m$ appears exactly once in $\gamma$.
         \begin{itemize}
             \item If $\alpha \in \mathfrak A_{p-1,j,1}$ then subtracting $j$ from every letter in $\gamma$ yields a word in the set $\mathfrak C_{n-p,m-j+1, \cancel{00}}$.
             \item If $\alpha \in \mathfrak A_{p-1,j,2}$ then subtracting $j+1$ from every letter in $\gamma$ yields a word in the set $\mathfrak C_{n-p,m-j}$.
        \end{itemize}
     \end{itemize}
     We easily observe that this construction is a bijection.
\end{proof}

\noindent For $a,b \in \{1,2\}$, let $\mathfrak D_{n,k,a,b} = \{ \omega \in \overline \W_{n,k}(000,120) \; : \; a = \occ(\omega, 0), \, b = \occ(\omega, k-1) \}$ be the set of $\{000,120\}$-avoiding words of length $n$ which contain every letter of the alphabet $[0,k-1]$ and in which the letter $0$ appears $a$ times and the letter $k-1$ appears $b$ times. Let $\mathfrak d_{n,k,a,b} = |\mathfrak D_{n,k,a,b}|$.
\begin{prop}
For all $n,k \geqslant 1$,
$$\mathfrak b_{n,k} = \sum_{d = 1}^{\min(n,k)} \binom{k}{d} \sum_{a,b \in \{1,2\}} \mathfrak d_{n,d,a,b},$$

$$\mathfrak c_{n,k} = \sum_{d = 1}^{\min(n,k)} \binom{k-1}{d-1} \sum_{a \in \{1,2\}} \mathfrak d_{n,d,a,1},$$

$$\mathfrak b_{n,k,\cancel{00}} = \sum_{d = 1}^{\min(n,k)} \left ( \binom{k-1}{d-1} \sum_{b \in \{1,2\}} \mathfrak d_{n,d,1,b} + \binom{k-1}{d} \sum_{a,b \in \{1,2\}} \mathfrak d_{n,d,a,b} \right ),$$

$$\mathfrak c_{n,k,\cancel{00}} = \sum_{d = 1}^{\min(n,k)} \left ( \binom{k-2}{d-2} \mathfrak d_{n,d,1,1} + \binom{k-2}{d-1} \sum_{a \in \{1,2\}} \mathfrak d_{n,d,a,1} \right ).$$
\end{prop}
\begin{proof}
    Let $F_{d,k}$ be the set of increasing functions from $[0,d-1]$ to $[0,k-1]$. The four identities above can be proven using the same construction as in the proof of Proposition \ref{propbinomwords}, observing that
    \begin{enumerate}
        \item $\binom{k}{d}$ = $|F_{d,k}|$,
        \item $\binom{k-1}{d-1} = |\{\varphi \in F_{d,k} \; : \; \varphi(d-1) = k-1\}|$,
        \item $\binom{k-1}{d-1} = |\{\varphi \in F_{d,k} \; : \; \varphi(0) = 0\}|$, and $\binom{k-1}{d} = |\{\varphi \in F_{d,k} \; : \; \varphi(0) > 0\}|$,
        \item $\binom{k-2}{d-2} = |\{\varphi \in F_{d,k} \; : \; \varphi(0) = 0, \, \varphi(d-1) = k-1\}|$, and $\binom{k-2}{d-1} = |\{\varphi \in F_{d,k} \; : \; \varphi(0) > 0, \, \varphi(d-1) = k-1\}|$. \qedhere
    \end{enumerate}
\end{proof}

Let $\mathfrak E_{n,k} = \{\lambda \in \{1,2\}^k \; : \; n = \sum_{i = 1}^k \lambda_i \}$ be the set of compositions of $n$ into $k$ parts of size 1 or 2. Equivalently, $\mathfrak E_{n,k}$ is the set of words of length $k$ over the alphabet $\{1,2\}$ which contain exactly $2k-n$ occurrences of the letter 1 and $n-k$ occurrences of the letter 2. Observe that $\mathfrak E_{n,k}$ is empty if $n < k$ or $n > 2k$.

\begin{rem} \label{rem000+120}
    In the notation of Section \ref{symmetry},
$$\mathfrak D_{n,k,a,b} = \coprod_{\substack{\lambda \in \mathfrak E_{n,k} \\ (\lambda_1, \lambda_k) = (a,b)}} \W_{\lambda}(120).$$
\end{rem}

Let $\mathfrak F_{n,k} = \W_{(2)^{n-k} \cdot (1)^{2k-n}}(120)$ be the set of 120-avoiding words of length $n$ over the alphabet $[0,k-1]$, which contain exactly 2 occurrences of each letter in $[0,n-k-1]$ and 1 occurrence of each letter in $[n-k,k-1]$, and let $\mathfrak f_{n,k} = |\mathfrak F_{n,k}|$.

\begin{prop}
    For all $k \geqslant 2, n \in [k,2k]$ and $a,b \in \{1,2\}$,
    $$\mathfrak d_{n,k,a,b} = \binom{k-2}{n-k+2-a-b} \mathfrak f_{n,k}$$
\end{prop}
\begin{proof}

For all $k \geqslant 2, n \in [k,2k]$ and $a,b \in \{1,2\}$, 
the set of sequences $\lambda \in \mathfrak E_{n,k}$ such that $\lambda_1 = a$ and $\lambda_k = b$ is the set of sequences of the form $a \cdot \lambda' \cdot b$ where $\lambda'$ is a word of length $k-2$ over the alphabet $\{1,2\}$ which contains exactly $n-k+2-a-b$ occurrences of the letter 2. Hence, there are $\binom{k-2}{n-k+2-a-b}$ such sequences.

By Theorem \ref{thmSymmetry}, the number of words $|\W_{\lambda}(120)|$ is independent of the choice of ${\lambda \in \mathfrak E_{n,k}}$, and counted by $\mathfrak f_{n,k}$. Remark \ref{rem000+120} concludes the proof.
\end{proof}

\begin{lem}
    The following recurrence relation holds for all $k \geqslant 1, n \in [k+1,2k]$:
    $$\mathfrak f_{n,k} = \mathfrak f_{n,k+1} - \mathfrak f_{n-1,k},$$
    with initial conditions $\mathfrak f_{n,n} = \frac{1}{n+1}\binom{2n}{n}$ for all $n \geqslant 0$.
\end{lem}
\begin{proof}
    Let $k \geqslant 2, n \in [k,2k-2]$, and $\omega \in \mathfrak F_{n,k}$. In particular, $\omega$ contains a single letter $k-1$ and a single letter $k-2$.
    \begin{itemize}
        \item If the letter $k-2$ is on the left of $k-1$, then $k-1$ is the last letter of $\omega$ (because $\omega$ avoids the pattern 120). In that case, removing the letter $k-1$ yields a word in $\mathfrak F_{n-1,k-1}$, and this is a bijection.
        \item Otherwise, the letter $k-2$ is on the right of $k-1$. In that case, replacing the letter $k-1$ by $k-2$ yields a word in $\W_{(2)^{n-k} \cdot (1)^{2k-n-2} \cdot (2)}(120)$, and this is a bijection. By Theorem \ref{thmSymmetry}, this set is equinumerous with $\W_{(2)^{n-k+1} \cdot (1)^{2k-n-2}}(120) = \mathfrak F_{n,k-1}$.
    \end{itemize}
    This shows that for all $k \geqslant 2$ and $n \in [k, 2k-2]$, $\mathfrak f_{n,k} = \mathfrak f_{n,k-1} + \mathfrak f_{n-1,k-1}$, which is equivalent to the recurrence relation of the lemma.

    For all $n \geqslant 0$, $\mathfrak F_{n,n}$ is trivially in bijection with the set of 120-avoiding permutations of size $n$, which is known to be counted by the Catalan number $C_n = \frac{1}{n+1}\binom{2n}{n}$.
\end{proof}
\noindent The numbers $\mathfrak f_{n,k}$ can be found in entry A059346 of the OEIS.

\subsection{The pattern 010} \label{010}

We denote by $\Ustirling{n}{k}$ the unsigned Stirling numbers of the first kind, which count the number of permutations of size $n$ with $k$ cycles (among other combinatorial interpretations, cf. entry A132393 of the OEIS).

Let $\mathfrak A_{n,m,d} = \{\sigma \in \I_n(010) \; : \; m = \max(\sigma), \, d = \dist(\sigma)\}$ be the set of $010$-avoiding inversion sequences of size $n$, maximum $m$, and having $d$ distinct values. Let $\mathfrak B_{n,k} = \{\omega \in \overline \W_{n,k}(010) \; : \; \omega_1 = k-1\}$ be the set of $010$-avoiding words of length $n$ which contain all letters of the alphabet $[0,k-1]$ and begin with their largest letter. Let $\mathfrak a_{n,m,d} = |\mathfrak A_{n,m,d}|$ and $\mathfrak b_{n,k} = |\mathfrak B_{n,k}|$.

\begin{thm} \label{thm010}
For all $2 \leqslant d \leqslant m + 1 \leqslant n$,
$$\mathfrak{a}_{n,m,d} =  \sum_{i = 0}^{d-1} \binom{m-i}{d-i-1} \sum_{p = m+1}^n \mathfrak b_{n-p+1, d-i} \sum_{j = 0}^{m-1} \mathfrak{a}_{p-1,j,i}.$$
\end{thm}
\begin{proof}
    Let $\sigma \in \mathfrak{A}_{n,m,d}$, $p = \firstmax(\sigma)$, $\alpha = (\sigma_i)_{i \in [1,p-1]}$, and $\beta = (\sigma_i)_{i \in [p, n]}$.
    
    Since $\sigma$ avoids the pattern 010, $\alpha$ and $\beta$ avoid 010, and $\vals(\alpha) \cap \vals(\beta) = \emptyset$ by Proposition \ref{propsplit010}. In particular, $\alpha \in \mathfrak{A}_{p-1,j,i}$ for some $j < m$ and $i < d$, and $\beta$ is a 010-avoiding word of length $n-p+1$, which contains exactly $d-i$ distinct values chosen from the remaining $m+1-i$ (that is, all values in $[0, m]$ except for the $i$ values in $\alpha$), and such that $\beta_1 = m$. Since $m$ is always in $\beta$, there are $\binom{m-i}{d-i-1}$ possible choices for the set $\vals(\beta)$.
    Once the values of $\beta$ are chosen, there are $\mathfrak b_{n-p+1, d-i}$ ways to arrange them into a word avoiding 010 and beginning with its largest letter.
\end{proof}

\begin{lem} \label{lem010}
For all $n, k \geqslant 1$,
$$\mathfrak{b}_{n,k} = \Ustirling{n}{n+1-k}.$$
\end{lem}
\begin{proof}
    Let $n \geqslant k \geqslant 2$, and $\omega \in \mathfrak{B}_{n,k}$. Since $\omega$ avoids the pattern $010$, all letters 0 in $\omega$ are consecutive.
    \begin{itemize}
    \item If $\omega$ contains several letters 0, then removing one of them yields a word $\omega' \in \mathfrak{B}_{n-1,k}$, and this is clearly a bijection.
    \item If $\omega$ contains a single letter 0, then removing it and subtracting 1 from all other letters yields a word $\omega' \in \mathfrak{B}_{n-1,k-1}$. Since $\omega_1 = k-1 > 0$, there are $n-1$ possible positions for the letter 0 in $\omega$, so exactly $n-1$ words $\omega \in \mathfrak{B}_{n,k}$ yield the same $\omega'$.
    \end{itemize}
    Hence the following recurrence relation holds for all $n \geqslant k \geqslant 2$:
    $$\mathfrak{b}_{n,k} = \mathfrak{b}_{n-1,k} + (n-1) \mathfrak{b}_{n-1,k-1}.$$
    This recurrence relation is also satisfied by the Stirling numbers of the first kind $\Ustirling{n}{n+1-k}$, and we can easily verify that initial conditions also match.
\end{proof}

\subsection{The pair of patterns \{000, 010\}} \label{000+010}

Let $\mathfrak A_{n,m,d} = \{ \sigma \in \I_n(000,010) \; : \; m = \max(\sigma), \, d = \dist(\sigma) \}$ be the set of $\{000, 010\}$-avoiding inversion sequences of size $n$, maximum $m$, and having $d$ distinct values. Let $\mathfrak B_{n,k} = \{ \omega \in \overline \W_{n,k}(000,010) \; : \; \omega_1 = k-1\}$ be the set of $\{000, 010\}$-avoiding words of length $n$ which contain all letters of the alphabet $[0,k-1]$ and begin with their largest letter.
Let $\mathfrak a_{n,m,d} = |\mathfrak A_{n,m,d}|$, and $\mathfrak b_{n,k} = |\mathfrak B_{n,k}|$.
Since the sequences of $\mathfrak A_{n,m,d}$ and the words of $\mathfrak B_{n,k}$ avoid the pattern 000, we have $\mathfrak a_{n,m,d} = 0$ if $n > 2d$, and $\mathfrak b_{n,k} = 0$ if $n > 2k$.
\begin{thm} \label{thm000+010}
For all $2 \leqslant d \leqslant m+1 \leqslant n$,
$$\mathfrak a_{n,m,d} = \sum_{i = 0}^{d-1} \binom{m-i}{d-i-1} \sum_{p = m+1}^n \mathfrak b_{n-p+1,d-i} \sum_{j = 0}^{m-1} \mathfrak a_{p-1,j,i}.$$
\end{thm}
\begin{proof}
    Identical to the proof of Theorem \ref{thm010}: the pattern 000 could not spread over $\alpha$ and $\beta$, since the avoidance of 010 already implies that $\alpha$ and $\beta$ do not share any values.
\end{proof}

\begin{lem} \label{lem000+010}
For all $n, k \geqslant 2$,
$$\mathfrak b_{n,k} = (n-1) \mathfrak b_{n-1,k-1} + (n-2) \mathfrak b_{n-2,k-1}.$$
\end{lem}
\begin{proof}
    Let $n \geqslant k \geqslant 2$, and $\omega \in \mathfrak B_{n,k}$. Since $\omega$ avoids $010$, all letters 0 in $\omega$ are consecutive. Since $\omega$ avoids $000$, $\omega$ has at most two letters 0.
    \begin{itemize}
    \item If $\omega$ contains a single letter 0, then removing it and subtracting $1$ from all other letters yields a word $\omega' \in \mathfrak B_{n-1,k-1}$. Since $\omega_1 = k-1 > 0$, there are $n-1$ possible positions for the letter 0 in $\omega$, so exactly $n-1$ words $\omega \in \mathfrak B_{n,k}$ yield the same $\omega'$.
    \item If $\omega$ contains two letters 0, then removing them and subtracting $1$ from all other letters yields a word $\omega' \in \mathfrak B_{n-2,k-1}$. There are $n-2$ possible positions for the two consecutive letters 0 in $\omega$, so exactly $n-2$ words $\omega \in \mathfrak B_{n,k}$ yield the same $\omega'$. \qedhere
    \end{itemize}
\end{proof}

\subsection{Forbidden values}

In the remainder of Section \ref{sectionSplitMax}, if $\alpha$ is a sequence avoiding a set of patterns $P$, we say that a value $v \in \{0, \dots, \max(\alpha)\}$ is \emph{forbidden} by $\alpha$ and $P$ if $\alpha \cdot m v$ contains a pattern in $P$ when $m > \max(\alpha)$. We denote by $\forb(\alpha, P)$ the set of values forbidden by $\alpha$ and $P$, or simply $\forb(\alpha)$ when there is no ambiguity.

In our decomposition around the first maximum of an inversion sequence $\sigma = \alpha \cdot m \cdot \gamma$, if $\gamma$ contains a value in $\forb(\alpha, \tau)$, then $\sigma$ contains $\tau$. This means that in order for $\sigma$ to avoid $\tau$, $\gamma$ must be a word over the alphabet $[0,m] \backslash \forb(\alpha)$. This is in fact a weaker version of condition 3 from Remark \ref{remsplitpattern}, since it only ensures that $\sigma$ does not contain any occurrence of $\tau$ whose last entry only is in $\gamma$.

Note that Propositions \ref{propsplit120}, \ref{propsplit010} can be expressed in terms of forbidden values:
$$\forb(\alpha, 120) = [0, \max(\alpha)-1], \qquad \forb(\alpha, 010) = \vals(\alpha),$$
and these identities hold for any integer sequence $\alpha$ (not only for inversion sequences). In particular, if $\alpha$ is an inversion sequence such that $\max(\alpha) > 0$, then $|\forb(\alpha, 010)| \geqslant 2$.

Earlier, we refined the enumeration of pattern-avoiding inversion sequences $\sigma$ according to one or two parameters (in addition to their size). The first one was the maximum $m$ of $\sigma$, and is required for our decomposition around the first maximum (so that the maximum of the left part $\alpha$ is less than $m$). For every pair of pattern which contained 120, the number of forbidden values was redundant with the maximum of $\sigma$, and therefore unnecessary.
In the previous two cases (the patterns 010 and \{000, 010\}), the second parameter counting the number of distinct values $d$ of $\sigma$ was in fact the number of forbidden values $|\forb(\sigma,010)|$.

For each pair of patterns $P$ which follows, there is no simple equivalent description of $|\forb(\sigma,P)|$, and introducing this parameter allows us to solve the enumeration of $P$-avoiding inversion sequences. 

\subsection{The pairs \{010, 201\} and \{010, 210\}} \label{010+210}
A bijection between $\I(010, 201)$ and $\I(010, 210)$ was established in \cite{Yan_Lin_2020}. In this section we work with the pair of patterns $\{010,210\}$, although our construction can also be applied to inversion sequences avoiding the pair $\{010,201\}$, resulting in the same equations.

\begin{rem} \label{rem010+210}
    Let $\sigma$ be a $\{010, 210\}$-avoiding integer sequence. Let $q$ be the largest value of $\sigma$ such that a larger value appears to its left, or $q = 0$ if there is no such value (i.e. if $\sigma$ is nondecreasing).
    Then we have $\forb(\sigma,210) = [0,q-1]$. Recalling also that $\forb(\sigma,010) = \vals(\sigma)$, we find $\forb(\sigma, \{010,210\}) = [0,q-1] \sqcup \{i \in \vals(\sigma) \; : \; i \geqslant q\}$.
\end{rem}

Let $\mathfrak A_{n,m,f} = \{\sigma \in \I_n(010,210) \; : \; m = \max(\sigma), \, f = |\forb(\sigma)|\}$ be the set of $\{010, 210\}$-avoiding inversion sequences of size $n$, maximum $m$, and having $f$ forbidden values. Let $\mathfrak B_{n,k} = \{ \omega \in \W_{n,k+1}(010) \; : \; \omega_i < \omega_j < k \implies i < j \; \text{ and } \; k-1 \in \vals(\omega)\}$ if $k > 0$, or $\mathfrak B_{n,k} = \{(0)^n\}$ if $k = 0$, be the set of $010$-avoiding words $\omega$ of length $n$ over the alphabet $[0, k]$ such that the subword $\omega'$ obtained by removing all letters $k$ from $\omega$ is nondecreasing, and $k-1$ is the maximum of $\omega'$.
Let $\mathfrak a_{n,m,f} = |\mathfrak A_{n,m,f}|$, and $\mathfrak b_{n,k} = |\mathfrak B_{n,k}|$.

\begin{thm} \label{thm010+210}
For all $2 \leqslant f \leqslant m+1 \leqslant n$,
$$\mathfrak a_{n,m,f} = \sum_{p = m+1}^n \sum_{i=0}^{f-1} \mathfrak b_{n-p,f-i-1} \sum_{j = 0}^{m-1}  \mathfrak a_{p-1,j,i}.$$
\end{thm}
\begin{proof}
    Let $\sigma \in \mathfrak A_{n,m,f}$, $p = \firstmax(\sigma)$, $\alpha = (\sigma_i)_{i \in [1,p-1]}$, and $\gamma = (\sigma_i)_{i \in [p+1,n]}$. Let $\gamma'$ be the subsequence of $\gamma$ obtained by removing all values $m$ from $\gamma$.
    Since $\sigma$ avoids the pattern 010, $\alpha$ and $\gamma$ avoid 010, and $\vals(\alpha) \cap \vals(\gamma) = \emptyset$ by Proposition \ref{propsplit010}. Since $\sigma$ avoids the pattern 210, $\alpha$ avoids 210, and $\gamma'$ is nondecreasing. 
    
    The subsequence $\alpha$ is in $\mathfrak A_{p-1,j,i}$ for some $j < m$ and $i \leqslant f$. Notice that we actually have $i < f$ since $m$ is a forbidden value for $\sigma$ (because of the avoidance of 010), but not for $\alpha$ (since $m > \max(\alpha))$.
    The subsequence $\gamma$ is a 010-avoiding word of length $n-p$ over the alphabet $\Sigma = [0,m] \backslash \forb(\alpha)$ of size $m+1-i$, and such that $\gamma'$ is nondecreasing.
    
    The maximum of $\gamma'$ is the largest letter of $m \cdot \gamma$ such that a larger letter appears to its left, and $m$ is the only value in $m \cdot \gamma$ which is greater than $\max(\gamma')$. Hence, by Remark \ref{rem010+210}, $\forb(\sigma) = \forb(\alpha) \sqcup \{\ell \in \Sigma \; : \: \ell \leqslant \max(\gamma')\} \sqcup \{m\}$ (this still holds if $\gamma'$ is empty, i.e. $\max(\gamma') = -1$).
    Since $|\{\ell \in \Sigma \; : \: \ell \leqslant \max(\gamma')\}| = |\forb(\sigma)| - |\forb(\alpha)| - |\{m\}| = f-i-1$, replacing the letters $\{\ell \in \Sigma \; : \; \ell \leqslant \max(\gamma')\}$ by $[0,f-i-2]$ and the letter $m$ by $f-i-1$ yields a bijection between the words $\gamma$ and the words of $\mathfrak B_{n-p,f-i-1}$. \qedhere
\end{proof}

Let $\mathfrak C_{n,k}$ be the set of $010$-avoiding words $\omega$ of length $n$ over the alphabet $[0, k-1] \sqcup \{\infty\}$ (where $\infty$ is the largest letter) such that the subword $\omega'$ defined by removing all letters $\infty$ from $\omega$ is nondecreasing, and $k-1$ is the maximum of $\omega'$. Clearly, taking any word $\omega \in \mathfrak C_{n,k}$ and replacing each letter $\infty$ in $\omega$ by the letter $k$ yields a word in $\mathfrak B_{n,k}$, and this is a bijection. It is more convenient for the proof of the following lemma to count the words of $\mathfrak C_{n,k}$ rather than those of $\mathfrak B_{n,k}$.

Let $\mathfrak C^{(1)}_{n,k} = \{\omega \in \mathfrak C_{n,k} \; | \; \omega_n = \infty\}$, and $\mathfrak C^{(2)}_{n,k} = \{\omega \in \mathfrak C_{n,k} \; | \; \omega_n \neq \infty\}$, so that $\mathfrak C_{n,k} = \mathfrak C^{(1)}_{n,k} \sqcup \mathfrak C^{(2)}_{n,k}$.
Let $\mathfrak c^{(1)}_{n,k} = |\mathfrak C^{(1)}_{n,k}|$, and $\mathfrak c^{(2)}_{n,k} = |\mathfrak C^{(2)}_{n,k}|$. In particular, $\mathfrak b_{n,k} = \mathfrak c^{(1)}_{n,k} + \mathfrak c^{(2)}_{n,k}$.
\begin{lem} \label{lem010+210}
For all $n \geqslant 2, k \geqslant 1$,
$$\mathfrak b_{n,k} = \mathfrak b_{n,k-1} + 2 \mathfrak b_{n-1,k} - \mathfrak b_{n-2,k} - \mathfrak b_{n-1,k-1} + \mathfrak b_{n-2,k-1}.$$
\end{lem}
\begin{proof}
    Let $n \geqslant 1, k \geqslant 0$, and $\omega \in \mathfrak C_{n,k}$.
    \begin{itemize}
        \item If $\omega \in \mathfrak C^{(1)}_{n,k}$, then $\omega_n = \infty$. Removing $\omega_n$ yields a word in $\mathfrak{C}_{n-1,k}$.
        \item If $\omega \in \mathfrak C^{(2)}_{n,k}$, then by the nondecreasing property, $\omega_n = k-1$. 
        \begin{itemize}
            \item If $\omega_{n-1} = k-1$, then removing $\omega_n$ yields a word in $\mathfrak C^{(2)}_{n-1,k}$.
            \item If $\omega_{n-1} = \infty$, then the avoidance of the pattern 010 ensures $\omega$ cannot contain another letter $k-1$, therefore removing $\omega_n$ yields a word in $\mathfrak C^{(1)}_{n-1,i}$ for some $i < k$.
            \item Otherwise, $\omega_{n-1} = i$ for some $i < k$, and removing $\omega_n$ yields a word in $\mathfrak C^{(2)}_{n-1,i}$.
        \end{itemize}
    \end{itemize}
    The maps described above are all bijections, hence for all $n \geqslant 1, k \geqslant 0$,
    $$\mathfrak c^{(1)}_{n,k} = \mathfrak b_{n-1,k},$$
    and for all $n \geqslant 2, k \geqslant 0$,
    \begin{align*}
    \mathfrak c^{(2)}_{n,k} &= \mathfrak c^{(2)}_{n-1,k} + \sum_{i=0}^{k-1} \mathfrak c^{(1)}_{n-1,i} + \mathfrak c^{(2)}_{n-1,i} \\
    &= \mathfrak b_{n-1,k} - \mathfrak c^{(1)}_{n-1,k} + \sum_{i=0}^{k-1} \mathfrak b_{n-1,i} \\
    &= \mathfrak b_{n-1,k} - \mathfrak b_{n-2,k} + \sum_{i=0}^{k-1} \mathfrak b_{n-1,i}.
    \end{align*}
    By summing $\mathfrak c^{(1)}_{n,k}$ and $\mathfrak c^{(2)}_{n,k}$, we have for all $n \geqslant 2, k \geqslant 0$,
    $$\mathfrak b_{n,k} = 2 \mathfrak b_{n-1,k} - \mathfrak b_{n-2,k} + \sum_{i=0}^{k-1} \mathfrak b_{n-1,i}.$$
    We conclude by telescoping the sum over $i$. For all $n \geqslant 2, k \geqslant 1$, 
    \begin{align*}
        \mathfrak b_{n,k} - \mathfrak b_{n,k-1} &= 2 \mathfrak b_{n-1,k} - \mathfrak b_{n-2,k} + \sum_{i=0}^{k-1} \mathfrak b_{n-1,i} - (2 \mathfrak b_{n-1,k-1} - \mathfrak b_{n-2,k-1} + \sum_{i=0}^{k-2} \mathfrak b_{n-1,i}) \\
        &= 2 \mathfrak b_{n-1,k} - \mathfrak b_{n-2,k} - \mathfrak b_{n-1,k-1} + \mathfrak b_{n-2,k-1}. \qedhere
    \end{align*}
\end{proof}

\subsection{The pair \{010, 110\}} \label{010+110}

\begin{rem} \label{rem010+110}
    Let $\sigma$ be a $\{010, 110\}$-avoiding integer sequence. Let $q$ be the largest repeated value in $\sigma$, or $q = 0$ if there is no such value.
    Then $\forb(\sigma, 110) = [0,q-1]$. Recalling that $\forb(\sigma, 010) = \vals(\sigma)$, we find $\forb(\sigma, \{010, 110\}) = [0,q-1] \sqcup \{ i \in \vals(\sigma) \; : \; i \geqslant q\}$.
\end{rem}

Let $\mathfrak{A}_{n,m,f} = \{ \sigma \in \I_n(010, 110) \; : \; m = \max(\sigma), \, f = |\forb(\sigma)|\}$ be the set of $\{010,110\}$-avoiding inversion sequences of size $n$, maximum $m$, and having $f$ forbidden values. 
Let $\mathfrak B_{n,k,f} = \{ \omega \in \W_{n,k}(010, 110) \; : \; f = |\forb(\omega)|\}$ be the set of $\{010, 110\}$-avoiding words of length $n$ over the alphabet $[0, k-1]$ having $f$ forbidden values. Let $\mathfrak C_{n,k} = \coprod_{\ell \in [0,n]} \W_{\ell,k} (010,110)$ be the set of $\{010, 110\}$-avoiding words of length at most $n$ over the alphabet $[0, k-1]$.
Let $\mathfrak a_{n,m,f} = |\mathfrak A_{n,m,f}|$, $\mathfrak b_{n,k,f} = |\mathfrak B_{n,k,f}|$, and $\mathfrak c_{n,k} = |\mathfrak C_{n,k}|$.
\begin{thm} \label{thm010+110}
For all $2 \leqslant f \leqslant m+1 \leqslant n$,
$$\mathfrak a_{n,m,f} = \sum_{p=m+1}^{n} \sum_{i=0}^{f-1} \sum_{j=0}^{m-1} \mathfrak a_{p-1,j,i} \cdot (\mathfrak b_{n-p, m-i, f-i-1} + \delta_{f,m+1} \cdot \mathfrak{c}_{n-p-1, m-i}).$$
\end{thm}
\begin{proof}
    Let $\sigma \in \mathfrak A_{n,m,f}$, $p = \firstmax(\sigma)$, $\alpha = (\sigma_i)_{i \in [1,p-1]}$, and $\gamma = (\sigma_i)_{i \in [p+1,n]}$.
    Since $\sigma$ avoids the pattern 010, $\alpha$ and $\gamma$ avoid 010, and $\vals(\alpha) \cap \vals(\gamma) = \emptyset$ by Proposition \ref{propsplit010}. Since $\sigma$ avoids the pattern 110, $\alpha$ and $\gamma$ avoid 110, and all letters of $\gamma$ are greater than any repeated letter of $\alpha$.
    
    The subsequence $\alpha$ is in $\mathfrak A_{p-1,j,i}$ for some $j < m$ and $i < f$ (for the same reason as in the proof of Theorem \ref{thm010+210}). The subsequence $\gamma$ is a $\{010,110\}$-avoiding word of length $n-p$ over the alphabet $\Sigma = [0,m] \backslash \forb(\alpha)$ of size $m+1-i$. Additionally, if a letter $m$ appears in $\gamma$, then all letters to its right are also $m$; otherwise $\sigma$ would contain the pattern 110. We distinguish two cases for $\gamma$.
    \begin{itemize}
        \item If $\gamma$ does not contain the letter $m$, then $\gamma$ is a $\{010,110\}$-avoiding word of length $n-p$ over an alphabet of size $m-i$. Let $q$ be the largest repeated letter in $\gamma$, or $q = 0$ if there is no such letter. From Remark \ref{rem010+110}, we observe that the forbidden values of $\sigma$ are $\forb(\sigma) = \forb(\alpha) \sqcup \{\ell \in \Sigma \; : \; \ell < q\} \sqcup \{\ell \in \vals(\gamma) \; : \; \ell \geqslant q\} \sqcup \{m\}$. In particular, $|\{\ell \in \Sigma \; : \; \ell < q\} \sqcup \{\ell \in \vals(\gamma)\; : \; \ell \geqslant q\}| = f-i-1$. Replacing the letters of $\Sigma$ by $[0,m-i]$ (while preserving their order) yields a bijection between the words $\gamma$ and the words of $\mathfrak B_{n-p,m-i,f-i-1}$.
        \item If $\gamma$ contains the letter $m$, then $\gamma$ can be written as $\gamma' \cdot (m)^k$, where $\gamma'$ does not contain the letter $m$, and $k > 0$. The number of such words $\gamma'$ is $\mathfrak c_{n-p-1,m-i}$.
        In this case, $m$ is the largest repeated letter in $\sigma$ as well as the maximum of $\sigma$. This occurs if and only if $f = m+1$ (i.e. all values $[0,m]$ are forbidden by $\sigma$). \qedhere
    \end{itemize}
\end{proof}

By definition,
$$\mathfrak C_{n,k} = \coprod_{\ell = 0}^n \coprod_{f = 0}^k \mathfrak B_{\ell,k,f}, \qquad \text{hence} \qquad \mathfrak c_{n,k} = \sum_{\ell = 0}^n \sum_{f = 0}^k \mathfrak b_{\ell,k,f}.$$
We count the words of $\mathfrak B_{n,k,f}$ in Lemma \ref{lem010+110}. In order to do so, we first construct the words of $\overline \W_{n,k}(010, 110)$ by inserting their letters in increasing order from $0$ to $k-1$ (all such letters necessarily appearing), and inserting repeats of a letter from left to right.
For all $\omega \in \overline \W_{n,k}(010, 110)$, let $\Sites(\omega) = \{ i \in [1,n+1] \; : \; \vals((\omega_j)_{j < i}) \cap \vals((\omega_j)_{j \geqslant i}) = \emptyset\}$ be the set of positions (called \emph{active sites}) where the letter $k$ may be inserted in $\omega$ without creating an occurrence of the pattern 010. In particular, $1$ and $n+1$ are always active sites of $\omega$. Note also that inserting the letter $k$ in $\omega$ cannot create an occurrence of the pattern 110: the letter $k$ cannot take the role of the letter 0 in the pattern 110 since it is greater than all letters of $\omega$, and it cannot take the role of the letter 1 since it only appears once.

Let $\mathfrak D_{n,k,s} = \{ \omega \in \overline \W_{n,k}(010, 110) \; : \; s = |\Sites(\omega)|\}$ be the set of $\{010, 110\}$-avoiding words of length $n$ containing all letters of the alphabet $[0,k-1]$ and having $s$ active sites. Let $\mathfrak D_{n,k,s}^{(1)}$ be the subset of $\mathfrak D_{n,k,s}$ of words containing exactly one letter $k-1$, and $\mathfrak D_{n,k,s}^{(2)}$ be the subset of remaining words (i.e. words containing at least two letters $k-1$ if $k \geqslant 1$, or the empty word if $(n,k,s) = (0,0,1)$), so that $\mathfrak D_{n,k,s} = \mathfrak D_{n,k,s}^{(1)} \sqcup \mathfrak D_{n,k,s}^{(2)}$. Let $\mathfrak d_{n,k,s} = |\mathfrak D_{n,k,s}|$, $\mathfrak d_{n,k,s}^{(1)} = |\mathfrak D_{n,k,s}^{(1)}|$, and $\mathfrak d_{n,k,s}^{(2)} = |\mathfrak D_{n,k,s}^{(2)}|$.

\begin{lem} \label{lem010+110GT}
For all $n,k,s \geqslant 1$,
$$\mathfrak d_{n,k,s}^{(1)} = (s-1) \mathfrak d_{n-1,k-1,s-1}.$$
For all $n,s \geqslant 2, k \geqslant 1$,
$$\mathfrak d_{n,k,s}^{(2)} = \mathfrak d_{n,k,s+1}^{(2)} + \mathfrak d_{n-1,k,s}^{(2)} - \mathfrak d_{n-1,k,s+1}^{(2)} + \mathfrak d_{n-2,k-1,s-1}.$$
\end{lem}
\begin{proof}
Let $n,k,s \geqslant 0$, $\omega \in \mathfrak D_{n,k,s}$, and let $(p_1, \dots, p_s)$ be the values of $\Sites(\omega)$ in increasing order (in particular, $p_1 = 1$ and $p_s = n+1$). We consider three different ways in which the word $\omega$ can ``grow".
\begin{enumerate}
    \item Inserting a letter $k$.  Let $i \in [1,s]$, and let $\omega'$ be the word obtained by inserting the letter $k$ at position $p_i$ in $\omega$. Then $\Sites(\omega') = \{p_j \; : \; j \in [1,i]\} \sqcup \{p_j + 1 \; : \; j \in [i,s]\}$, so $\omega' \in \mathfrak D^{(1)}_{n+1,k+1,s+1}$.
    \item Inserting two occurrences of the letter $k$. This is similar to the previous case, since the rightmost letter $k$ can only be inserted at the end of the word in order to avoid the pattern 110.  Let $i \in [1,s]$, and let $\omega'$ be the word obtained by inserting one letter $k$ at position $p_i$ in $\omega$, and one letter $k$ at the end of the resulting word (at position $n+2$). Then $\Sites(\omega') = \{p_j \; : \; j \in [1,i]\} \sqcup \{n+3\}$, so $\omega' \in \mathfrak D^{(2)}_{n+2,k+1,i+1}$.
    \item Inserting a repeat of letter $k-1$, to the right of the rightmost letter $k-1$, and only if $\omega$ already contains at least two occurrences of the letter $k-1$ (i.e. $\omega \in \mathfrak D^{(2)}_{n,k,s}$ and $k > 0$). In that case, every letter $k-1$ in $\omega$ except (possibly) the leftmost $k-1$ must be in a single factor at the end of $\omega$, in order to avoid 110.
    Let $\omega'$ be the word obtained by inserting the letter $k-1$ at position $n+1$ in $\omega$. Then $\Sites(\omega') = \big (\Sites(\omega) \backslash \{n+1\} \big ) \sqcup \{n+2\}$, so $\omega' \in \mathfrak D^{(2)}_{n+1,k,s}$.
\end{enumerate}
It can be seen that any word in a set $\mathfrak D_{n,k,s}$ for some $n,k,s \geqslant 0$ can be obtained in exactly one way from the construction described above, starting from the empty word $\varepsilon$. In other words, we described a combinatorial generating tree for the class $\coprod_{n,k \geqslant 0} \overline \W_{n,k}(010,110)$. 

From item 1, we have for all $n,k,s \geqslant 1$,
$$\mathfrak d_{n,k,s}^{(1)} = (s-1) \mathfrak d_{n-1,k-1,s-1}.$$

From items 2 and 3, we have for all $n,s \geqslant 2, k \geqslant 1$,
$$\mathfrak d_{n,k,s}^{(2)} = \mathfrak d_{n-1,k,s}^{(2)} + \sum_{j = s-1}^{n-1} \mathfrak d_{n-2,k-1,j},$$
which may be rewritten
\begin{align*}
    \mathfrak d_{n,k,s}^{(2)} - \mathfrak d_{n,k,s+1}^{(2)} &= \mathfrak d_{n-1,k,s}^{(2)} - \mathfrak d_{n-1,k,s+1}^{(2)} + \mathfrak d_{n-2,k-1,s-1}. \qedhere
\end{align*}
\end{proof}

For any integer sequence $\sigma$, let $\rep(\sigma) = \max( v \in \vals(\sigma) \; : \; \exists i \neq j, \; \sigma_i = \sigma_j = v)$ be the largest repeated value in $\sigma$, with the convention $\rep(\sigma) = -1$ if $\sigma$ does not contain any repeated value. Let $\top(\sigma) = |\{i \in \vals(\sigma) \; : \; i > \rep(\sigma)\}|$ be the number of values of $\sigma$ which are greater than its largest repeated value (if no value is repeated, then $\top(\sigma)$ is the number of distinct values of $\sigma$, or equivalently the size of $\sigma$).

We call \emph{unused letters} of a word $\omega \in \W_{n,k}$ the letters in the set $[0,k-1] \backslash \vals(\omega)$. Note this definition relies not only on $\omega$, but also on the alphabet considered (e.g. each word of $\W_{n,k}$ is also in $\W_{n,k+1}$, but has different unused letters).

\begin{lem} \label{lem010+110}
For all $n,k,f \geqslant 0$,
$$\mathfrak b_{n,k,f} = \sum_{t = 0}^{f} \binom{t+k-f}{k-f} \sum_{b = 0}^{f-t} \binom{f-t-1}{b} \sum_{s = 1}^{n-t+1} \frac{(s+t-1)!}{(s-1)!} \mathfrak d_{n-t,f-b-t,s}^{(2)}.$$
\end{lem}
\begin{proof}
    Let $\omega \in \mathfrak B_{n,k,f}$. In particular, by Remark $\ref{rem010+110}$, $f = \rep(\omega) + \top(\omega) + 1$, and $k-f$ is the number of unused letters of $\omega$ greater than $\rep(\omega)$.
    Let $t = \top(\omega)$, $d = \dist(\omega)$, and $b = f - d$ be the number of unused letters in $\omega$ less than $\rep(\omega)$. To summarize, over the alphabet $[0,k-1]$, $\omega$ has:
    \begin{itemize}
        \item $t$ letters greater than $\rep(\omega)$,
        \item $k-f$ unused letters greater than $\rep(\omega)$,
        \item $b$ unused letters less than $\rep(\omega)$,
        \item $\rep(\omega)-b = f-t-1-b$ letters less than $\rep(\omega)$, if $\omega \neq \varepsilon$.
    \end{itemize}

    Replacing the $d$ letters of $\vals(\omega)$ by $[0,d-1]$ while preserving their order (by shifting the values so that there are no more unused letters) yields a word $\omega' \in \overline \W_{n,d}$. Further removing all letters greater than $\rep(\omega')$ from $\omega'$ yields a word $\omega'' \in \mathfrak D^{(2)}_{n-t, f-b-t, s}$ for some $s \in [1,n-t+1]$.
    
    For any $n,k,f,t,b,s$, and $\omega'' \in \mathfrak D^{(2)}_{n-t, f-b-t, s}$, there are $\binom{t+k-f}{k-f} \binom{f-t-1}{b} \frac{(s+t-1)!}{(s-1)!}$ words $\omega \in \mathfrak B_{n,k,f}$ whose image under the above construction is $\omega''$. Indeed, there are
    \begin{itemize}
        \item $\binom{t+k-f}{k-f}$ possible sets of unused letters greater than $\rep(\omega)$,
        \item $\binom{f-t-1}{b}$ possible sets of unused letters less than $\rep(\omega)$ (this still holds if $\omega = \varepsilon$),
        \item $\frac{(s+t-1)!}{(s-1)!}$ possible placements for the $t$ letters greater than $\rep(\omega)$. This can be seen by inserting each of those $t$ letters in increasing order: inserting a letter greater than the maximum in a word corresponds to item 1 in the proof of Lemma \ref{lem010+110GT}, so there are $s$ possible positions for the smallest inserted letter, $s+1$ for the next letter, and so on. \qedhere
    \end{itemize}
\end{proof}

\section{Shifted inversion sequences} \label{sectionShift}

\subsection{Method}
For all $n,s \in \mathbb N$, let $\I_n^s = \{\sigma \in \mathbb N^n \; : \; \sigma_i < i+s \quad \forall i \in [1,n] \}$ be the set of \emph{$s$-shifted} inversion sequences of size $n$.
In particular, $\I_n^s \subseteq \I_n^{s+1}$, and $\I_n^0 = \I_n$. An $s$-shifted inversion sequence of size $n$ can be seen as an inversion sequence of size $n+s$ whose first $s$ entries were removed. More precisely, for all $n, s \geqslant 0$, $\I_{n+s} = \{\sigma \cdot \tau \; : \; (\sigma, \tau) \in \I_s \times \I_n^s\}$. For any set of patterns $P$, we denote by $\I_n^s(P)$ the set of $P$-avoiding $s$-shifted inversion sequences of size $n$.

In this section, we study some patterns for which it is easier to split sequences around their minimum. We use a decomposition of sequences around their first minimum, similar to that of Section \ref{sectionSplitMax}, although a shifted inversion sequence now naturally appears on the right side when this decomposition is applied to an inversion sequence.

In Section \ref{sectionSplitMax}, the right side of the decomposition was a word, since the value of the maximum was fixed. A similar event occurs in the upcoming cases, when decomposing sequences around their first minimum: the sequences of Sections \ref{010+102} and \ref{100+102} avoid the pattern 102, so when the left side of the decomposition is nonempty (i.e. the first value of the sequence is not the minimum), all values of the right side must be less than or equal to all values of the left side, hence the right side is a word on a fixed alphabet, once again. It follows that shifted inversion sequences only appear on the right side of this decomposition when the left side is empty. Naturally, the left side of the decomposition around the first minimum of an inversion sequence is always empty, since the leftmost value of a nonempty inversion sequence is always 0. It follows that words only appear when this decomposition is recursively applied to a shifted inversion sequence.

For instance, applying this decomposition to $(0,0,0,0,4,4,5,0,3,4,2,3,0) \in \I_{13}(102)$ will first remove each leading zero, yielding the sequence $(4,4,5,0,3,4,2,3,0) \in \I_9^4(102)$, then split it into a shifted inversion sequence $(4,4,5)$ and a word $(3,4,2,3,0)$.

\subsection{The pair \{010, 102\}} \label{010+102}
For the generating function of $\I(010,102)$, see \cite{Testart_GT_left}.

Let $\mathfrak A_{n,s} = \I_n^s(010,102)$ be the set of $\{010, 102\}$-avoiding $s$-shifted inversion sequences of size $n$. Let $\mathfrak B_{n,k} = \{\omega \in \W_{n,k}(010,102) \; : \; \max(\omega) = k-1\}$ be the set of $\{010, 102\}$-avoiding words of length $n$ over the alphabet $[0,k-1]$ which contain the letter $k-1$. Let $\mathfrak a_{n,s} = |\mathfrak A_{n,s}|$ and $\mathfrak b_{n,k} = |\mathfrak B_{n,k}|$. In particular, $|\I_n(010,102)| = \mathfrak a_{n,0}$.

\begin{thm} \label{thm010+102}
For all $n \geqslant 1, s \geqslant 0$,
    $$\mathfrak a_{n,s} =  (\sum_{z = 0}^n \mathfrak a_{n-z,s+z-1} ) + ( \sum_{z=1}^{n-1}  \mathfrak a_{n-z,s-1} ) + \sum_{r = 1}^{n-2} \sum_{m = 1}^{s} \mathfrak b_{r,m} \sum_{\ell=0}^{n-r-1} (n-r-\ell-\delta_{\ell,0}) \mathfrak a_{\ell,s-m-1} .$$

\end{thm}
\begin{proof}
    Let $n \geqslant 1, s \geqslant 0$, and $\sigma \in \mathfrak A_{n,s}$. Since $\sigma$ avoids the pattern 010, all occurrences of the value 0 in $\sigma$ must be consecutive. Let $z = |\{i \in [1,n] \; : \; \sigma_i = 0\}|$ be the number of occurrences of the value 0 in $\sigma$. Let $\sigma'$ be the sequence obtained by removing every occurrence of the value 0 from $\sigma$, and subtracting 1 from all remaining values.
    \begin{itemize}
        \item If $z = 0$ or $\sigma_1 = 0$, then $\sigma' \in \mathfrak A_{n-z, s+z-1}$. This is a bijection since $\sigma = (0)^z \cdot (\sigma'+1)$.
        \item If $\sigma_1 \neq 0$ and $\sigma_n = 0$, then $\sigma' \in \mathfrak A_{n-z, s-1}$. This is a bijection since $\sigma = (\sigma'+1) \cdot (0)^z$.
        \item If $z > 0$, $\sigma_0 \neq 0$, and $\sigma_n \neq 0$, let $\alpha$ and $\beta$ be the (uniquely defined) integer sequences such that $\sigma = \alpha \cdot (0)^z \cdot \beta$. In particular, $\alpha$ and $\beta$ are both nonempty, and their values are positive. Let $r$ be the size of $\beta$, $m$ be the maximum of $\beta$, and $\beta' \in \mathfrak B_{r, m}$ be the word obtained by subtracting 1 from each letter of $\beta$.
        
        Since $\sigma$ avoids 102, every value of $\alpha$ is greater than or equal to $m$ (in particular, $m \leqslant \alpha_1 \leqslant s$). All occurrences of the value $m$ in $\alpha$ must be in a single factor at the end of $\alpha$; otherwise, a value greater than $m$ would appear to the right of a value $m$ in $\alpha$, creating an occurrence of the pattern 010 in $\sigma$ (since $\beta$ also contains $m$). Let $q$ be the number of occurrences of the value $m$ in $\alpha$, and let $\ell = |\alpha| - q$ be the number of entries of $\alpha$ of value greater than $m$. Let $\alpha' \in \mathfrak A_{\ell, s-m-1}$ be the sequence obtained by removing every occurrence of the value $m$ from $\alpha$ and subtracting $m+1$ from the remaining values.

        Note that any choice of $q$ and $z$ such that $q+z = n-\ell-r$ does not affect $\alpha'$ and $\beta'$. We have $q \in [0,n-r-\ell-1]$ if $\alpha'$ is nonempty, and $q \in [1,n-r-\ell-1]$ if $\alpha'$ is empty (since $\alpha$ is not empty), so there are $n-r-\ell - \delta_{\ell,0}$ possible values for $q$. For each possible value of $q$, there is one value of $z$ such that $q+z = n-\ell-r$.

        This decomposition is bijective, since for any choice of
        \begin{itemize}
            \item $m \in [1, s]$,
            \item $\ell, q, z, r \in \mathbb N$ such that $\ell+q, z, r \geqslant 1$ and $\ell+q+z+r = n$,
            \item $\alpha' \in \mathfrak A_{\ell,s-m-1}$,
            \item $\beta' \in \mathfrak B_{r,m}$,
        \end{itemize}
         we have $(\alpha'+m+1) \cdot (m)^q \cdot (0)^z \cdot (\beta'+1) \in \mathfrak A_{n,s}$. Specifically, this construction cannot create any occurrence of a pattern 010 or 102 since all values of $\alpha'+m+1$ are greater than all values of $\beta'+1$. \qedhere
    \end{itemize}
\end{proof}

Let $F \, : \, (x,y) \mapsto \sum_{\ell,k \geqslant 0} |\W_{\ell,k}(010,102)| x^\ell y^k$ be the ordinary generating function of $\{010,102\}$-avoiding words.
The enumeration of $\{010,102\}$-avoiding words was solved in \cite{Mansour_2006}, with the following expression:
$$F(x,y) = \frac{(1-x)^2 - (1-2x)y - \sqrt{(1-x)^4 - 2(1-x)^2 y + (1-4x^2+4x^3)y^2}}{2xy(1-y)}.$$
We easily observe that $\mathfrak B_{n,k} = \W_{n,k}(010,102) \backslash \W_{n,k-1}(010,102)$, so the generating function of $\coprod_{n,k \in \mathbb N} \mathfrak B_{n,k}$ is $(x,y) \mapsto (1-y) F(x,y)$. We can then extract the numbers $\mathfrak b_{n,k}$ from this generating function, as in Section \ref{101+120}. The function $F$ is actually the same as in Section \ref{101+120}, since $\{010,102\}$-avoiding words and $\{101,120\}$-avoiding words are in bijection via the complement map $\W_{n,k}(010, 102) \to \W_{n,k}(101, 120)$, $\omega \mapsto (k-1-\omega_i)_{i \in [1,n]}$.

\subsection{The pair \{100, 102\}} \label{100+102}

Let $\mathfrak A_{n,s} = \I_n^s(100,102)$ be the set of $\{100,102\}$-avoiding $s$-shifted inversion sequences of size $n$. Let $\mathfrak A'_{n,s} = \{ \sigma \in \mathfrak A_{n,s} \; : \; 0 \in \vals(\sigma) \text{ and } \sigma_1 \neq 0\}$ be the subset of sequences which contain a 0 but do not begin by 0. Let $\mathfrak B_{n,k} = \{\omega \in \W_{n,k} (102) \; : \; \omega_i = \omega_j \implies i = j\}$ be the set of 102-avoiding words of length $n$ over the alphabet $[0,k-1]$ which do not contain any repeated letter. Let $\mathfrak C_{n,k} = \{\omega \in \W_{n,k} (102) \; : \; \omega_i = \omega_j \neq k-1 \implies i = j\}$ be the set of 102-avoiding words of length $n$ over the alphabet $[0,k-1]$ which do not contain any non-maximal repeated letter. Let $\mathfrak a_{n,s} = |\mathfrak A_{n,s}|$, $\mathfrak a'_{n,s} = |\mathfrak A'_{n,s}|$, $\mathfrak b_{n,k} = |\mathfrak B_{n,k}|$, and $\mathfrak c_{n,k} = |\mathfrak C_{n,k}|$. In particular, $\I_n(100,102) = \mathfrak a_{n,0}$.

\begin{thm} \label{thm100+102}
    For all $n \geqslant 1, s \geqslant 0$,
    $$\mathfrak a_{n,s} = \mathfrak a_{n,s-1} + \mathfrak a_{n-1,s+1} + \mathfrak a'_{n,s},$$
    $$\mathfrak a'_{n,s} = \sum_{p=2}^n \sum_{m=1}^s \left ( \mathfrak c_{n-p,m} + \sum_{r=1}^{p-2} \mathfrak a_{p-1-r, s+r-m-1} \cdot \mathfrak b_{n-p,m} + \sum_{r=0}^{p-2} \mathfrak a'_{p-1-r, s+r-m} \cdot \mathfrak b_{n-p,m-1} \right ). $$
\end{thm}
\begin{proof}
    Let $n \geqslant 1, s \geqslant 0$, and $\sigma \in \mathfrak A_{n,s}$.
    \begin{itemize}
        \item If $\sigma$ does not contain the value 0, then subtracting 1 from every value of $\sigma$ yields a sequence in $\mathfrak A_{n,s-1}$, and this is a bijection.
        \item If $\sigma_1 = 0$, then removing the first term from $\sigma$ yields a sequence in $\mathfrak A_{n-1,s+1}$, and this is a bijection.
        \item Otherwise, $\sigma \in \mathfrak A'_{n,s}$.
    \end{itemize}
    This proves the first equality. Let us now turn to the second one.
    
    Let $n \geqslant 1, s \geqslant 0$, and $\sigma \in \mathfrak A'_{n,s}$. Since $\sigma$ avoids the pattern 100 and $\sigma_1 > 0$, there is exactly one term of value $0$ in $\sigma$. Let $p$ be the position of the only value 0 in $\sigma$. Let $\alpha = (\sigma_i)_{i \in [1,p-1]}$ and $\beta = (\sigma_i)_{i \in [p+1,n]}$ be the two subsequences such that $\sigma = \alpha \cdot 0 \cdot \beta$. In particular, all values of $\alpha$ and $\beta$ are positive. By construction, $\alpha$ is nonempty, but $\beta$ may be empty. Let $m$ be the minimum of $\alpha$. Since $\sigma$ avoids the pattern 102, all values of $\beta$ must be less than or equal to $m$, so $\beta$ is a word of length $n-p$ over the alphabet $[1,m]$. Since $\sigma$ avoids the pattern 100, each value in $[1,m-1]$ appears at most once in $\beta$.  Let $\beta'$ be the word obtained by subtracting 1 from every letter of $\beta$.
    \begin{itemize}
        \item If $\alpha = (m)^{p-1}$ is constant, then $\beta'$ can be any word of $\mathfrak C_{n-p,m}$.
        \item Otherwise, $\alpha$ contains at least one value greater than $m$. Since $\sigma$ avoids the pattern 100, the letter $m$ can appear at most once in $\beta$. In other words, all letters of $\beta$ must be distinct, so $\beta' \in \mathfrak B_{n-p,m}$. Let $r \in [0,p-2]$ be the number of consecutive occurrences of the value $m$ at the beginning of $\alpha$.
        \begin{itemize}
            \item If all values $m$ in $\alpha$ appear in the factor $(m)^r$ at the beginning of $\alpha$, let $\alpha' \in \mathfrak A_{p-1-r,s+r-m-1}$ be the sequence obtained by removing this factor from $\alpha$, and subtracting $m+1$ from all remaining values. Recall $\alpha$ must contain the value $m$ at least once, so $r > 0$. This defines a bijection, since for any $\alpha' \in \mathfrak A_{p-1-r,s+r-m-1}$ and $\beta' \in \mathfrak B_{n-p,m}$, the sequence $(m)^r \cdot (\alpha'+m+1) \cdot 0 \cdot (\beta'+1)$ is in $\mathfrak A'_{n,s}$.
            \item Otherwise, there is exactly one value $m$ in $\alpha$ outside of the factor $(m)^r$ at the beginning of $\alpha$, since having a second value $m$ outside of this factor would imply that the subsequence $(\alpha_{r+1},m,m)$ is an occurrence of the pattern 100. The same reasoning shows that $\beta$ cannot contain the letter $m$, so $\beta' \in \mathfrak B_{n-p,m-1}$.
            Let $\alpha' \in \mathfrak A'_{p-1-r,s+r-m}$ be the sequence obtained by removing the factor $(m)^r$ at the beginning of $\alpha$, and subtracting $m$ from all remaining values. This defines a bijection, since for any $\alpha' \in \mathfrak A'_{p-1-r,s+r-m}$ and $\beta' \in \mathfrak B_{n-p,m-1}$, the sequence $(m)^r \cdot (\alpha'+m) \cdot 0 \cdot (\beta'+1)$ is in $\mathfrak A'_{n,s}$. \qedhere
        \end{itemize}
    \end{itemize}
\end{proof}

Words avoiding the pairs of patterns $\{100, 102\}$ or $\{011, 120\}$ are trivially in bijection via the complement map $\W_{n,k}(100, 102) \to \W_{n,k}(011, 120)$, $\omega \mapsto (k-1-\omega_i)_{i \in [1,n]}$. We studied the pair of patterns $\{011, 120\}$ in Section \ref{011+120}, and it can easily be seen that the complement map defines bijections between the sets $\mathfrak B_{n,k}$, resp. $\mathfrak C_{n,k}$, and their counterparts of Section \ref{011+120}. Therefore, the following identities are deduced from Proposition \ref{prop011+120} and Lemma \ref{lem011+120}:
\begin{itemize}
    \item for all $n,k \geqslant 0$,
    $$\mathfrak b_{n,k} = \frac{\binom{k}{n} \binom{2n}{n} }{n+1},$$
    \item for all $n \geqslant 1, k \geqslant 2$,
    $$\mathfrak c_{n,k} = 2 \mathfrak c_{n,k-1} + \mathfrak c_{n-1,k} - \mathfrak c_{n-1,k-1} - \frac{\binom{k-2}{n} \binom{2n}{n}}{n+1}.$$
\end{itemize}

\section{Generating functions} \label{sectionGF}
In the literature, a succession rule describing a generating tree is often used to calculate the generating function of the associated combinatorial class. Indeed, a succession rule can be turned into a system of functional equations satisfied by the generating function. These equations often involve \emph{catalytic} variables, which correspond to some statistics of the combinatorial objects (the same statistics that label the generating tree). A common tool for solving these equations is the \emph{kernel method} \cite{GFforGT}.

The constructions of pattern-avoiding inversion sequences we presented allow us to compute many terms of their enumeration sequences (usually a few hundred, see Table \ref{table2} for more precise data). From those initial terms, Pantone has conjectured algebraic expressions for three of their generating functions, using his software \cite{guessfunc}.
\begin{conj} \label{conj000+102}
The ordinary generating function $F$ of $\I(000, 102)$ is algebraic with minimal polynomial
$$x^4 F(x)^4 - 2x^3(x - 1)F(x)^3 + x(x^3 - 2x^2 + 4x - 1)F(x)^2 - (2x^2 - 2x + 1)F(x) + 1.$$
\end{conj}
\begin{conj} \label{conj102+201}
The ordinary generating function $F$ of $\I(102, 201)$ is algebraic with minimal polynomial
\begin{align*}
    &x(x-1)^2 (x-2)^2 (2x-1)^2 F(x)^2 + (x-1)(2x-1)(4x^4-9x^3+5x^2+4x-2) F(x) \\
    &- x^5+9x^4-22x^3+25x^2-12x+2.
\end{align*}
\end{conj}
\begin{conj} \label{conj102+210}
The ordinary generating function $F$ of $\I(102, 210)$ is algebraic with minimal polynomial
\begin{align*}
&(4x-1)(x-1)^4 x^3 F(x)^2-(4x-1)(4x^4-22x^3+25x^2-9x+1)(x-1)^2 F(x) \\
&+4x^7-44x^6+165x^5-254x^4+194x^3-75x^2+14x-1.
\end{align*}
\end{conj}
We prove Conjectures \ref{conj102+201} and \ref{conj102+210} by calculating the generating functions of $\I(102, 201)$ and $\I(102, 210)$, in Theorems \ref{thmGF102+201} and \ref{thmGF102+210} below. In a personal communication, Pantone proves Conjecture \ref{conj000+102} by calculating the generating function of $\I(000, 102)$, using a system of equations we derived from the succession rule $\Omega_{\{000,102\}}$ of Theorem \ref{thm000+102}. We present this proof in the appendix.

Pantone also has a conjecture in \cite{Pantone201+210} about inversion sequences avoiding the patterns 010 and 102.
\begin{conj} \label{conj010+102}
     The ordinary generating function $F$ of $\I(010, 102)$ is algebraic with minimal polynomial
     \begin{align*}
         &x(x^2 - x + 1)(x - 1)^2 F(x)^3 + 2x(x - 1)(2x^2 - 2x + 1)F(x)^2\\
         &- (x^4 - 8x^3 + 11x^2 - 6x + 1)F(x) - (2x - 1)(x - 1)^2.
     \end{align*}
\end{conj}
\noindent Conjecture \ref{conj010+102} is now proved in \cite{Testart_GT_left}, using yet another generating tree construction of inversion sequences.

The software \cite{guessfunc} could not guess an algebraic, D-finite, or D-algebraic generating function for any other classes of pattern-avoiding inversion sequences we studied in this article.

We now calculate the generating functions of $\I(102,201)$ and $\I(102,210)$ from our succession rules, using the kernel method.

\subsection{The generating function of \texorpdfstring{$\I$}{}(102, 201)}
We define generating functions corresponding to the numbers $\mathfrak a, \mathfrak b^{(i)}, \mathfrak c$ of Section \ref{102+201}.
Let $A(x,y) = \sum_{n,m \geqslant 0} \mathfrak a_{n,m} x^n y^m$, $B^{(i)}(x,y) = \sum_{n,s \geqslant 0} \mathfrak b^{(i)}_{n,s} x^n y^s$ for $i \in \{1,2,3\}$, and $C(x,y) = \sum_{n,\ell \geqslant 0} \mathfrak c_{n,\ell} x^n y^\ell$. Let $F(x) = \sum_{n \geqslant 0} |\I_n(102,201)| x^n$ be the ordinary generating function of $\I(102,201)$.

Recall the succession rule $\Omega_{\{102,201\}}$ from Remark \ref{remSuccession102+201}:
$$\Omega_{\{102,201\}} = \left \{ \begin{array}{rclll}
    (a,0,0) \\
    (a,n,m) & \leadsto & (a,n+1,i) & \text{for} & i \in [m,n] \\
    && (b^{(1)},i)^{n-i} & \text{for} & i \in [m,n-1] \\
    && (c,i) & \text{for} & i \in [0,m-1] \vspace{5pt} \\
    (b^{(1)},s) & \leadsto & (b^{(1)},s) \, (b^{(2)},s) \vspace{5pt}\\
    (b^{(2)},s) & \leadsto & (b^{(2)},s) \, (b^{(3)},s) \vspace{5pt}\\
    (b^{(3)},s) & \leadsto & (b^{(3)},s)^2 \\
    && (c,i) & \text{for} & i \in [0,s-1] \vspace{5pt}\\
    (c, \ell) & \leadsto & (c,i) & \text{for} & i \in [0, \ell].
\end{array}
\right .$$

\begin{prop} \label{propEq102+201}
    The functions $A, B^{(1)}, B^{(2)}, B^{(3)}$, $C$, and $F$ satisfy the following equations.
    \begin{align*}
        & A(x,y) = 1 + \frac{x}{1-y} \big (A(x,y) - y A(xy,1) \big ) \\
        & B^{(1)}(x,y) = x B^{(1)}(x,y) + \frac{x}{(1-y)} \left ( \hspace{-1.3pt} x \frac{\partial A}{\partial x}(x,y) - y \frac{\partial A}{\partial y}(x,y) + \frac{y}{1-y} \big (A(xy,1)-A(x,y) \big ) \hspace{-1.3pt} \right ) \\
        & B^{(2)}(x,y) = x \big (B^{(1)}(x,y) + B^{(2)}(x,y) \big ) \\
        & B^{(3)}(x,y) = x \big (B^{(2)}(x,y) + 2 B^{(3)}(x,y) \big ) \\
        & C(x,y) = \frac{x}{1-y} \big ( C(x,1) - y C(x,y) + A(x,1) - A(x,y) + B^{(3)}(x,1) - B^{(3)}(x,y) \big ) \\
        & F(x) = A(x,1) + B^{(3)}(x,1) + C(x,1)
    \end{align*}
\end{prop}
\begin{proof}
    We turn the succession rule $\Omega_{\{102,201\}}$ into equations relating the generating functions.
    \begin{align*}
        A(x,y) &= 1 + \sum_{n,m \geqslant 0} \mathfrak a_{n,m} x^{n+1} \sum_{i=m}^n y^i \\
        &= 1 + \sum_{n,m \geqslant 0} \mathfrak a_{n,m} x^{n+1} \frac{y^m - y^{n+1}}{1-y} \\
        &= 1 + \frac{x}{1-y} \big (A(x,y) - y A(xy,1) \big )
    \end{align*}

    \begin{align*}
        B^{(1)}(x,y) &= \left ( \sum_{n,s \geqslant 0} \mathfrak b^{(1)}_{n,s} x^{n+1} y^s \right ) + \left ( \sum_{n,m \geqslant 0} \mathfrak a_{n,m} x^{n+1} \sum_{i=m}^{n-1} (n-i) y^i \right ) \\
        &= x B^{(1)}(x,y) + \sum_{n,m \geqslant 0} \mathfrak a_{n,m} \frac{x^{n+1}}{1-y} \left (n y^m - m y^m + \frac{y^{n+1} - y^{m+1}}{1-y} \right ) \\
        &= x B^{(1)}(x,y) + \frac{x}{(1-y)} \left ( \hspace{-1.3pt} x \frac{\partial A}{\partial x}(x,y) - y \frac{\partial A}{\partial y}(x,y) + \frac{y}{1-y} \big (A(xy,1)-A(x,y) \big ) \hspace{-1.3pt} \right )
    \end{align*}
    
    \begin{align*}
         B^{(2)}(x,y) &= \sum_{n,s \geqslant 0} (\mathfrak b^{(1)}_{n,s} + \mathfrak b^{(2)}_{n,s}) x^{n+1} y^s \\
         &=  x \big (B^{(1)}(x,y) + B^{(2)}(x,y) \big )
    \end{align*}

    \begin{align*}
         B^{(3)}(x,y) &= \sum_{n,s \geqslant 0} (\mathfrak b^{(2)}_{n,s} + 2 \mathfrak b^{(3)}_{n,s}) x^{n+1} y^s \\
         &= x \big (B^{(2)}(x,y) + 2B^{(3)}(x,y) \big )
    \end{align*}

    \begin{align*}
        C(x,y) &= \left ( \sum_{n,\ell \geqslant 0} \mathfrak c_{n,\ell} x^{n+1} \sum_{i = 0}^\ell y^i \right ) + \left ( \sum_{n,k \geqslant 0} (\mathfrak a_{n,k} + \mathfrak b^{(3)}_{n,k}) x^{n+1} \sum_{i=0}^{k-1} y^i \right ) \\
        &= \left ( \sum_{n,\ell \geqslant 0} \mathfrak c_{n,\ell} x^{n+1} \frac{1-y^{\ell+1}}{1-y} \right ) + \left ( \sum_{n,k \geqslant 0} (\mathfrak a_{n,k} + \mathfrak b^{(3)}_{n,k}) x^{n+1} \frac{1-y^k}{1-y} \right )  \\
        &= \frac{x}{1-y} \big ( C(x,1) - y C(x,y) + A(x,1) - A(x,y) + B^{(3)}(x,1) - B^{(3)}(x,y) \big )
    \end{align*}
    The equation $F(x) = A(x,1) + B^{(3)}(x,1) + C(x,1)$ is an immediate consequence of Theorem \ref{thm102+201}.
\end{proof}

\begin{thm} \label{thmGF102+201}
    The generating function of $\I(102,201)$ is
    $$F(x) = \frac{-8x^4 + 18x^3 - 10x^2 - 8x + 4 + 2(2x - 1)(x^2 - 2x + 2)\sqrt{(5x-1)(x-1)}}{4x (2x-1) (x-1) (x-2)^2}.$$
\end{thm}
\begin{proof}
From Proposition \ref{propEq102+201}, we have
    \begin{equation} \label{eqA102+201}
        (1-y-x) A(x,y) = 1-y - xy A(xy,1).
    \end{equation}
    We cancel the kernel $(1-y-x)$ by defining $Y(x) = 1-x$ and replacing $y$ by $Y(x)$:
    $$0 = 1 - (1-x) - x(1-x) A(x(1-x),1),$$
    and obtain
    $$A(x-x^2,1) = \frac{1}{1-x}.$$
    This determines a unique formal power series $A(x,1)$:
    $$A(x,1) = \frac{1-\sqrt{1-4x}}{2x}.$$
    Replacing $A(x,1)$ by this expression in \eqref{eqA102+201} yields
    $$A(x,y) = \frac{1-2y+\sqrt{1-4xy}}{2-2x-2y}.$$

From Proposition \ref{propEq102+201}, we have
$$B^{(1)}(x,y) = \frac{x}{(1-x)(1-y)} \left ( x \frac{\partial A}{\partial x}(x,y) - y \frac{\partial A}{\partial y}(x,y) + \frac{y}{1-y} \big (A(xy,1)-A(x,y) \big ) \right ),$$
$$B^{(2)}(x,y) = \frac{x}{1-x} B^{(1)}(x,y),$$
$$B^{(3)}(x,y) = \frac{x}{1-2x} B^{(2)}(x,y),$$
$$(1 - y + xy) C(x,y) = x \big ( C(x,1) + A(x,1) - A(x,y) + B^{(3)}(x,1) - B^{(3)}(x,y) \big ).$$
We easily obtain expressions for $B^{(1)}(x,y)$, $B^{(2)}(x,y)$, and $B^{(3)}(x,y)$ from that of $A(x,y)$. As for $C$, we apply the kernel method again. The kernel is $1 - y + xy$, we cancel it by setting $y = \frac{1}{1-x}$. We obtain an expression for $C(x,1)$:
$$C(x,1) = A \left (x,\frac{1}{1-x} \right ) - A(x,1) + B^{(3)} \left (x,\frac{1}{1-x} \right ) - B^{(3)}(x,1).$$
Finally, we can write an expression for $F(x)$:
\begin{align*}
    F(x) &= A(x,1) + B^{(3)}(x,1) + C(x,1) \\
    &= \frac{-8x^4 + 18x^3 - 10x^2 - 8x + 4 + 2(2x - 1)(x^2 - 2x + 2)\sqrt{(5x-1)(x-1)}}{4x (2x-1) (x-1) (x-2)^2}. \qedhere
\end{align*}
\end{proof}

\subsection{The generating function of \texorpdfstring{$\I$}{}(102, 210)}
We define generating functions corresponding to the numbers $\mathfrak a, \mathfrak b^{(i)}, \mathfrak c^{(i)}$ of Section \ref{102+210}.
Let $A(x,y) = \sum_{n,k \geqslant 0} \mathfrak a_{n,k} x^n y^k$, $B^{(i)}(x,y) = \sum_{n,k \geqslant 0} \mathfrak b^{(i)}_{n,k} x^n y^k$, $C^{(i)}(x,y) = \sum_{n,k \geqslant 0} \mathfrak c^{(i)}_{n,k} x^n y^k$ for $i \in \{1,2\}$. Let $F(x) = \sum_{n \geqslant 0} |\I_n(102,210)| x^n$ be the ordinary generating function of $\I(102,210)$.

Recall the succession rule $\Omega_{\{102,210\}}$ from Theorem \ref{thm102+210}:
$$\Omega_{\{102,210\}} = \left \{ \begin{array}{rclll}
(a,0) \\
(a,k) & \leadsto & (a,i) & \text{for} & i \in [1,k+1] \\
&& (b^{(1)},i)^{k-i} & \text{for} & i \in [1,k-1] \vspace{5pt}\\
(b^{(1)},k) & \leadsto & (b^{(1)},k) \\
&& (b^{(2)},i) & \text{for} & i \in [1,k+1] \\
&& (c^{(1)},i) & \text{for} & i \in [1,k-1] \vspace{5pt}\\
(b^{(2)},k) & \leadsto & (b^{(2)},i) & \text{for} & i \in [1,k+1]  \vspace{5pt}\\
(c^{(1)},k) & \leadsto & (c^{(1)},k) \\
&& (c^{(2)},i) & \text{for} & i \in [1,k] \vspace{5pt}\\
(c^{(2)},k) & \leadsto & (c^{(2)},i) & \text{for} & i \in [1,k].
\end{array} \right.$$

\begin{prop} \label{propEq102+210}
    The functions $A, B^{(1)}, B^{(2)}, C^{(1)}, C^{(2)}$, and $F$ satisfy the following equations.
    \allowdisplaybreaks
    \begin{align*}  
        & A(x,y) = 1 + \frac{xy}{1-y} \big (A(x,1) - y A(x,y) \big ) \\
        & B^{(1)}(x,y) = xB^{(1)}(x,y) + \frac{xy}{1-y} \left (\frac{\partial A}{\partial y}(x,1) + \frac{1}{1-y} \big (A(x,y)-A(x,1) \big ) \right) \\
        & B^{(2)}(x,y) = \frac{xy}{1-y} \left (B^{(1)}(x,1) + B^{(2)}(x,1) - y \big (B^{(1)}(x,y) + B^{(2)}(x,y) \big ) \right ) \\
        & C^{(1)}(x,y) = xC^{(1)}(x,y) + \frac{x}{1-y} \left ( yB^{(1)}(x,1) - B^{(1)}(x,y) \right ) \\
        & C^{(2)}(x,y) = \frac{xy}{1-y} \left ( C^{(1)}(x,1) + C^{(2)}(x,1) - C^{(1)}(x,y) - C^{(2)}(x,y) \right ). \\
        & F(x) = A(x,1) + B^{(1)}(x,1) + B^{(2)}(x,1) + C^{(1)}(x,1) + C^{(2)}(x,1)
    \end{align*}
\end{prop}
\begin{proof}
    We turn the succession rule $\Omega_{\{102,210\}}$ into equations relating the generating functions.
\begin{align*}
    A(x,y) &= 1 + \sum_{n,k \geqslant 0} \mathfrak a_{n,k} x^{n+1} \sum_{i = 1}^{k+1} y^i \\
    &= 1 + \sum_{n,k \geqslant 0} \mathfrak a_{n,k} x^{n+1} \frac{y-y^{k+2}}{1-y} \\
    &= 1 + \frac{xy}{1-y} \big ( A(x,1) - y A(x,y) \big )
\end{align*}
\begin{align*}
    B^{(1)}(x,y) &= \left (\sum_{n,k \geqslant 0} \mathfrak b^{(1)}_{n,k} x^{n+1} y^k \right ) + \left ( \sum_{n,k \geqslant 0} \mathfrak a_{n,k} x^{n+1} \sum_{i=1}^{k-1} (k-i) y^i \right ) \\
    &= xB^{(1)}(x,y) + \sum_{n,k \geqslant 0} \mathfrak a_{n,k} \frac{x^{n+1} y}{1-y} \left (k + \frac{y^k - 1}{1-y} \right ) \\
    &= xB^{(1)}(x,y) + \frac{xy}{1-y} \left ( \frac{\partial A}{\partial y}(x,1) + \frac{1}{1-y} \big ( A(x,y) - A(x,1) \big ) \right )
\end{align*}
\begin{align*}
    B^{(2)}(x,y) &= \sum_{n,k \geqslant 0} \left (\mathfrak b^{(1)}_{n,k} + \mathfrak b^{(2)}_{n,k} \right ) x^{n+1} \sum_{i=1}^{k+1} y^i \\
    &= \sum_{n,k \geqslant 0} \left (\mathfrak b^{(1)}_{n,k} + \mathfrak b^{(2)}_{n,k} \right ) x^{n+1} \frac{y-y^{k+2}}{1-y} \\
    &= \frac{xy}{1-y} \left (B^{(1)}(x,1) + B^{(2)}(x,1) - y \big (B^{(1)}(x,y) + B^{(2)}(x,y) \big ) \right )
\end{align*}
\begin{align*}
    C^{(1)}(x,y) &= \left ( \sum_{n,k \geqslant 0} \mathfrak c^{(1)}_{n,k} x^{n+1} y^k \right ) + \left ( \sum_{n,k \geqslant 0} \mathfrak b^{(1)}_{n,k} x^{n+1} \sum_{i=1}^{k-1} y^i \right ) \\
    &= xC^{(1)}(x,y) + \sum_{n,k \geqslant 0} \mathfrak b^{(1)}_{n,k} x^{n+1} \frac{y-y^k}{1-y} \\
    &= xC^{(1)}(x,y) + \frac{x}{1-y} \left ( yB^{(1)}(x,1) - B^{(1)}(x,y) \right )
\end{align*}
\begin{align*}
    C^{(2)}(x,y) &= \sum_{n,k \geqslant 0} \left (\mathfrak c^{(1)}_{n,k} + \mathfrak c^{(2)}_{n,k} \right ) x^{n+1} \sum_{i=1}^k y^i \\
    &= \sum_{n,k \geqslant 0} \left (\mathfrak c^{(1)}_{n,k} + \mathfrak c^{(2)}_{n,k} \right ) x^{n+1} \frac{y-y^{k+1}}{1-y} \\
    &= \frac{xy}{1-y} \left ( C^{(1)}(x,1) + C^{(2)}(x,1) - C^{(1)}(x,y) - C^{(2)}(x,y) \right ).
\end{align*}
By definition, $F(x) = A(x,1) + B^{(1)}(x,1) + B^{(2)}(x,1) + C^{(1)}(x,1) + C^{(2)}(x,1)$.
\end{proof}

\begin{thm} \label{thmGF102+210}
    The generating function $F(x)$ of $\I(102,210)$ is
    $$\frac{(4x-1) (4x^4 - 22x^3 + 25x^2 - 9x + 1) - (2x - 1)(x^2 - 5x + 1)(2x^2 - 4x + 1)\sqrt{1-4x}}{2 x^3 (4 x-1) (x-1)^2}.$$
\end{thm}
\begin{proof}
From Proposition \ref{propEq102+210}, we have
\begin{equation} \label{eqA102+210}
    (1-y+xy^2)A(x,y) = 1-y + xy A(x,1)
\end{equation}
\begin{equation} \label{eqB1102+210}
    B^{(1)}(x,y) = \frac{xy}{(1-x)(1-y)} \left (\frac{\partial A}{\partial y}(x,1) + \frac{1}{1-y} \big (A(x,y)-A(x,1) \big ) \right)
\end{equation}
\begin{equation} \label{eqB2102+210}
    (1-y+xy^2)B^{(2)}(x,y) = xy \left (B^{(1)}(x,1) + B^{(2)}(x,1) - y B^{(1)}(x,y) \right )
\end{equation}
\begin{equation} \label{eqC1102+210}
    C^{(1)}(x,y) = \frac{x}{(1-x)(1-y)} \left ( yB^{(1)}(x,1) - B^{(1)}(x,y) \right )
\end{equation}
\begin{equation} \label{eqC2102+210}
    (1-y+xy)C^{(2)}(x,y) = xy \left ( C^{(1)}(x,1) + C^{(2)}(x,1) - C^{(1)}(x,y) \right ).
    \end{equation}
First, we derive an expression of $A(x,y)$ from \eqref{eqA102+210}. To cancel the kernel $(1-y+xy^2)$, there are two solution in $y$:
$$Y_1(x) = \frac{1-\sqrt{1-4x}}{2x}, \qquad Y_2(x) = \frac{1+\sqrt{1-4x}}{2x}.$$
Only $Y_1$ defines a formal power series. We replace $y$ by $Y_1(x)$ in \eqref{eqA102+210}, and obtain
$$A(x,1) = \frac{1-\sqrt{1-4x}}{2x}.$$
Replacing $A(x,1)$ by this expression in \eqref{eqA102+210} yields
$$A(x,y) = \frac{2-y-y\sqrt{1-4x}}{2(1-y+xy^2)}.$$
We can obtain expressions for $B^{(1)}(x,y)$ and $C^1(x,y)$ by substituting our expression of $A(x,y)$ in \eqref{eqB1102+210} and \eqref{eqC1102+210}.

Next, we look for an expression of $B^{(2)}(x,1)$ from \eqref{eqB2102+210}. The kernel is the same as for $A$. Replacing $y$ by $\frac{1-\sqrt{1-4x}}{2x}$ in \eqref{eqB2102+210} yields
$$B^{(2)}(x,1) = \frac{1-\sqrt{1-4x}}{2x} B^{(1)}\left(x,\frac{1-\sqrt{1-4x}}{2x} \right) - B^{(1)}(x,1).$$
If we tried to directly evaluate $B^{(1)} \left (x,\frac{1-\sqrt{1-4x}}{2x} \right )$ from our expressions of $B^{(1)}(x,y)$ and $A(x,y)$, we would obtain the fraction $\frac{0}{0}$. Instead, we can write down an expression for $A\left (x, \frac{1-\sqrt{1-4z}}{2z} \right )$, then consider its limit as $z$ tends to $x$, to show that
$$A \left (x,\frac{1-\sqrt{1-4x}}{2x} \right ) = \frac{1}{2}\left (1+\frac{1}{\sqrt{1-4x}} \right ).$$
We can then obtain an expression of $B^{(1)}\left (x,\frac{1-\sqrt{1-4x}}{2x} \right )$, and therefore of $B^{(2)}(x,1)$.

Finally, we look for an expression of $C^{(2)}(x,1)$ from \eqref{eqC2102+210}. The kernel is $(1-y+xy)$. We cancel it by setting $y = \frac{1}{1-x}$, and obtain
$$C^{(2)}(x,1) = C^{(1)} \left (x,\frac{1}{1-x} \right ) - C^{(1)}(x,1).$$

The generating function of $\I(102,210)$ is
\begin{align*}
    &A(x,1) + B^{(1)}(x,1) + B^{(2)}(x,1) +  C^{(1)}(x,1) +  C^{(2)}(x,1) \\
    =& \frac{(4x-1) (4x^4 - 22x^3 + 25x^2 - 9x + 1) - (2x - 1)(x^2 - 5x + 1)(2x^2 - 4x + 1)\sqrt{1-4x}}{2 x^3 (4 x-1) (x-1)^2}. \qedhere
\end{align*}
\end{proof}

\section{Asymptotics} \label{sectionAsymptotics}
Looking back at the first few terms of the enumeration sequences in Table \ref{table2}, it could appear that among the pairs of patterns we have studied, the one avoided by the most inversion sequences is \{101, 210\}, and the one avoided by the fewest inversion sequences is \{000, 010\}. Interestingly, it can be shown that the enumeration sequence of $\I(101, 210)$ is bounded above by an exponential function, and that of $\I(000, 010)$ is super-exponential (see Propositions \ref{propPermAsymptotics} and \ref{propSuperExp} below). In fact, computation indicates there are more inversion sequences of size $n$ avoiding \{000, 010\} than \{101, 210\} starting at $n = 41$.

Marcus and Tardos \cite[Theorem 9]{Marcus_Tardos} prove that the number of $n \times n$ 0-1 matrices avoiding a permutation matrix $P$ is bounded above by $c_P^n$ for some constant $c_P$. The subset of $n \times n$ 0-1 matrices which contain exactly one 1-entry in each column is clearly in bijection with the set $\W_{n,n}$ of words of length $n$ over the alphabet $[0,n-1]$. For any permutation $\pi$, this restricts to a bijection between matrices which avoid the permutation matrix of $\pi$ and words avoiding the pattern $\pi$. Observing that $\W_{n,n}$ includes the set $\I_n$ of inversion sequences of size $n$, we obtain the following proposition.
\begin{prop} \label{propPermAsymptotics}
    For any permutation $\pi$, for all $n \geqslant 0$, we have $|\I_n(\pi)| \leqslant c_\pi^n$ for some constant $c_\pi$.
\end{prop}
\noindent In particular, the growth of $\I(101, 210)$ is at most exponential since 210 is a permutation pattern (i.e. a pattern without any repeated values).

Proposition \ref{propSuperExp} below implies that every class of inversion sequences avoiding only patterns with repeated values studied in this article have a super-exponential growth: $\I(010)$, $\I(000, 010)$, $\I(000,100)$, $\I(100,110)$, $\I(100,101)$, and $\I(010,110)$.

\begin{prop} \label{propSuperExp}
    The enumeration sequence of $\I(000, 010, 100, 101, 110)$ grows super-exponentially.
\end{prop}
\begin{proof}
Let $\mathfrak A_{n,k} = \{ \sigma \in \I_n \, : \, \sigma = \alpha \cdot (\beta+k) , \, \alpha = (0,0,1,1, \dots, k-1,k-1), \, \beta \in \I_{n-2k}^k(00)\}$ be the set of inversion sequences of length $n$ which begin by $(0,0,1,1, \dots, k-1,k-1)$, followed by $n-2k$ distinct values greater than or equal to $k$. Let $\mathfrak A_n = \coprod_{k \geqslant 0} \mathfrak A_{n,k}$. For all $k \geqslant 0$ and $n \geqslant 2k$, we have $|\mathfrak A_{n,k}| = (k+1)^{n-2k}$, hence for any $k \geqslant 0$, $|\mathfrak A_n|$ is asymptotically larger than $k^n$. Since this holds for an arbitrarily large value of $k$, $|\mathfrak A_n|$ has a super-exponential asymptotic behavior. By construction, every sequence in $\mathfrak A_n$ avoids the patterns 000, 010, 100, 101, and 110.
\end{proof}

Constructions similar to the above proof can be used to show that many classes of inversion sequences avoiding patterns with repeated letters have a super-exponential growth. However this is not always the case: for instance, \cite[Theorem 10]{Corteel_Martinez_Savage_Weselcouch_2016}
proves that $|\I_n(001)| = 2^{n-1}$ for all $n \geqslant 1$. Remarkably, 001 is the only pattern $\rho$ of size $3$ with repeated values such that the growth of $\I(\rho)$ is at most exponential.

Using his software \cite{DiffApprox}, Pantone provided us with precise conjectures about the asymptotic behavior of some of the enumeration sequences, based on the initial terms we computed. We present these conjectures in Table \ref{tableAsymptotics}. In particular, it seems that the classes $\I(010, 201)$ and $\I(101, 210)$ have the same growth rate, and the classes $\I(100, 120)$ and $\I(110, 201)$ as well.

\begin{table}[ht]
\begin{center}
    \begin{tabular}{|c|c|c|c|}
    \hline
    Class & $\mu$ & $\alpha$ & $\beta$ \\
    \hline
    $\I(000, 201)$ & $\sim 10.9282032$ & $\sim -7.3906$ & ? \\
    $\I(010, 201)$ & $\sim 10.1566572$ & $\sim -7.6168360$ & ? \\
    $\I(100, 102)$ & $\sim 5.066130494716195596699600$ & $-3/2$ & 0 \\
    $\I(100, 120)$ & $\sim 7.72334814688$ & $-3/2$& 0 \\
    $\I(101, 210)$ & $\sim 10.156657$ & $\sim -6.273831$ & 0 \\
    $\I(110, 201)$ & $\sim 7.7233481468847308370$ & $-3/2$ & 0 \\
    $\I(120, 201)$ & $\sim 7.4563913226671221339$ & $-3/2$ & 0 \\ \hline
    
    \end{tabular}
    \caption{Conjectured asymptotic behavior of the form $C \mu^n n^\alpha \log(n)^\beta$ for the enumeration sequences of some classes of pattern-avoiding inversion sequences.}
    \label{tableAsymptotics}
\end{center}
\end{table}

\section*{Appendix}

In this appendix, we present a proof of Conjecture \ref{conj000+102} provided by Jay Pantone in a personal communication.

Recall the succession rule $\Omega_{\{000,102\}}$ from Section \ref{000+102}:
$$\Omega_{\{000,102\}} = \begin{cases}
(1)\\
(s) \overset{1} \leadsto (j)^{s+1-j} &\text{for} \quad j \in [1, s]\\
(s) \overset{2} \leadsto (j+1)^{s+1-j} (j)^{\binom{s+1-j}{2}} &\text{for} \quad j \in [1, s].
\end{cases}$$
Let $\mathfrak a_{n,s} = |\{\sigma \in \I_n(000,102) \; : \; \sites(\sigma) = s\}|$ where $\sites$ is defined in Section \ref{000+102} and corresponds to the parameter of the rule $\Omega_{\{000,102\}}$. Let $A(x,y) = \sum_{n,s \geqslant 0} \mathfrak a_{n,s} x^n y^s$.

From the succession rule $\Omega_{\{000,102\}}$, it can be seen that the generating function of the children of a sequence of $\I(000,102)$ having size $n$ and $s$ active sites is
\begin{align*}
    & (x^{n+1} + x^{n+2}y) \left ( \sum_{j = 1}^s (s+1-j) y^j \right ) + x^{n+2} \left (\sum_{j = 1}^s \binom{s+1-j}{2} y^j \right) \\
    &= (x^{n+1}y + x^{n+2} y^2) \frac{y^{s+1} - y + s(1-y)}{(1-y)^{2}} \\
    & + x^{n+2}y \frac{2y - 2y^{s+1} - 2s(1-y) + s(s+1) (1-y)^2}{2(1-y)^3}.
\end{align*}
Seeing that
$$\frac{\partial A}{\partial y}(x,1) = \sum_{n,s \geqslant 0} s \, \mathfrak a_{n,s} x^n,$$
$$\frac{\partial^2 (y \cdot A)}{\partial y^2}(x,1) = \sum_{n,s \geqslant 0} s(s+1)\mathfrak a_{n,s} x^n,$$
we can apply this transformation to each monomial $\mathfrak a_{n,s} x^n y^s$, and obtain an equation which characterizes $A(x,y)$.
\begin{align*}
    A(x,y) &= y + \frac{xy + x^2 y^2}{(1-y)^2} \left( y (A(x,y)-A(x,1)) + (1-y) \frac{\partial A}{\partial y}(x,1) \right) \\
    &+ \frac{x^2y}{2(1-y)^3} \left ( \hspace{-1pt} 2y (A(x,1) - A(x,y)) - 2(1-y) \frac{\partial A}{\partial y}(x,1) + (1-y)^2 \frac{\partial^2 (y \cdot A)}{\partial y^2}(x,1) \hspace{-1pt} \right ) \\
    &= y - xy^2\frac{(1 + xy)(1-y) - x}{(1-y)^3} (A(x,1) - A(x,y)) \\ & + xy\frac{(1+xy)(1-y) - x}{(1-y)^2}\frac{\partial A}{\partial y}(x,1)
    + \frac{x^2y}{2(1-y)}\frac{\partial^2 (y \cdot A)}{\partial y^2}(x,1)
\end{align*}
    
To simplify this expression, we introduce two series $B(x,y)$ and $C(x,y)$.
\begin{align*} 
B(x,y) &= y\frac{A(x,1)-A(x,y)}{1-y} = \sum_{n,s \geqslant 0} \mathfrak a_{n,s} x^n \sum_{i=1}^s y^i \\
C(x,y) &= \frac{B(x,1)-B(x,y)}{1-y} = \sum_{n,s \geqslant 0} \mathfrak a_{n,s} x^n \sum_{i=1}^s \sum_{j=0}^{i-1} y^j
\end{align*}
From the definitions of $B(x,y)$ and $C(x,y)$, it can be seen that
\begin{align*}
B(x,1) &= \frac{\partial A}{\partial y}(x,1), \\
C(x,1) &= \sum_{n,s \geqslant 0} \mathfrak a_{n,s} x^n \sum_{i=1}^s \sum_{j=0}^{i-1} 1 = \sum_{n,s \geqslant 0} \frac{s(s+1)}{2} \mathfrak a_{n,s} x^n = \frac{1}{2} \frac{\partial^2 (y \cdot A)}{\partial y^2}(x,1).
\end{align*}
We can now rewrite our equation for $A(x,y)$ in terms of $B$ and $C$, without partial derivatives.
\begin{align*}
    A(x,y) &= y + xy\frac{(1 + xy)(1-y) - x}{(1-y)^2} (B(x,1) - B(x,y)) + \frac{x^2y}{1-y} C(x,1) \\
    &= y + xy\frac{(1 + xy)(1-y) - x}{1-y} C(x,y) + \frac{x^2y}{1-y} C(x,1)
\end{align*}
In summary, $A(x,y)$ is defined by the following system of equations involving $x$, $y$, the three series $A(x,y)$, $B(x,y)$, $C(x,y)$, and their evaluations for $y=1$.
$$\begin{cases}
    0 = (y-1) A(x,y) + (1-y)y + xy((1 + xy)(1-y) - x) C(x,y) + x^2y C(x,1)\\
    0 = (y-1) B(x,y) + y(A(x,1)-A(x,y)) \\
    0 = (y-1) C(x,y) + (B(x,1)-B(x,y))
\end{cases}$$
This can be seen as a system of three polynomial equations in eight variables ($x$, $y$, $A(x,1)$, $B(x,1)$, $C(x,1)$, $A(x,y)$, $B(x,y)$, and $C(x,y)$). Using Gröbner basis computations\footnote{We use the Maple package PolynomialIdeals.}, we can eliminate $B(x,y)$ and $C(x,y)$ from this system, and obtain an equation involving only six variables:
\begin{equation} \label{eqAxy}
\begin{split}
    &(-x^2 y^4 + x^2 y^3 - x^2 y^2 - xy^3 + xy^2 + y^3 - 3y^2 + 3y - 1)A(x,y) \\
    &+ xy^2 (xy^2 - xy + x + y - 1)A(x,1) + xy (y-1) (xy^2 - xy + x + y -1) B(x,1) \\
    &+ x^2 y (y-1)^2 C(x,1) - y (y-1)^3 = 0.
\end{split}
\end{equation}
This single equation uniquely defines the series $A(x,y)$, $A(x,1)$, $B(x,1)$, and $C(x,1)$, under the assumption that $A(x,y)$ is a formal power series in $x$ whose coefficients are polynomials in $y$.
Next, we use the approach from \cite{Bousquet-Mélou_Jehanne_2006} generalizing the kernel method to solve this equation. The kernel of \eqref{eqAxy} is
$$-x^2 y^4 + x^2 y^3 - x^2 y^2 - xy^3 + xy^2 + y^3 - 3y^2 + 3y - 1.$$
The coefficient of $x^0$ is a polynomial of degree 3 in $y$, therefore the kernel has three roots $Y_1(x)$, $Y_2(x)$, $Y_3(x)$ that are fractional power series in $x$, by \cite[Theorem 2]{Bousquet-Mélou_Jehanne_2006}.
Now, for $i \in \{1,2,3\}$, we have
$$-x^2 Y_i(x)^4 + x^2 Y_i(x)^3 - x^2 Y_i(x)^2 - xY_i(x)^3 + xY_i(x)^2 + Y_i(x)^3 - 3Y_i(x)^2 + 3Y_i(x) - 1 = 0$$
and
\begin{align*}
    &xY_i(x)^2 (xY_i(x)^2 - xY_i(x) + x + Y_i(x) - 1)A(x,1) \\
    &+ xY_i(x) (Y_i(x)-1) (xY_i(x)^2 - xY_i(x) + x + Y_i(x) -1) B(x,1) \\
    &+ x^2 Y_i(x) (Y_i(x)-1)^2 C(x,1) - Y_i(x) (Y_i(x)-1)^3 = 0.
\end{align*}
This forms a system of 6 equations. We add another indeterminate $Z$ and a seventh equation
$$Z(Y_1(x) - Y_2(x))(Y_1(x) - Y_3(x))(Y_2(x) - Y_3(x)) - 1 = 0,$$
to ensure that the system has a solution only if the series $Y_1(x)$, $Y_2(x)$, and $Y_3(x)$ are distinct.
Fortunately they are distinct, and using Gröbner basis computations again, we can find an equation involving only $x$ and $A(x,1)$.
$$x^4 A(x,1)^4 - 2x^3(x - 1)A(x,1)^3 + x(x^3 - 2x^2 + 4x - 1)A(x,1)^2 + (-2x^2 + 2x - 1)A(x,1) + 1 = 0$$
The left-hand side is the minimal polynomial of $A(x,1)$, and this proves Conjecture \ref{conj000+102}. We can compute the minimal polynomials of $B(x,1)$ and $C(x,1)$ in the same way.
$$x^6 B(x,1)^4 - 3x^4 B(x,1)^3 + x^2(x^2 - x + 3)B(x,1)^2 + (-x^2 + x - 1)B(x,1) + 1 = 0$$
$$x^8 C(x,1)^4 + x^5(3x + 1)C(x,1)^3 + 2x^3(2x + 1)C(x,1)^2 + (3x^2 + x - 1)C(x,1) + 1 = 0$$
Going back to equation \eqref{eqAxy}, we can now see that $A(x,y)$ is the sum of three algebraic functions of $(x,y)$ (of degree 4 each), so $A(x,y)$ is itself algebraic. Its minimal polynomial is difficult to obtain, since asking a computer algebra program to directly eliminate $A(x,1)$, $B(x,1)$ and $C(x,1)$ from a system of four equations consisting in \eqref{eqAxy} and the minimal polynomials of $A(x,1)$, $B(x,1)$, and $C(x,1)$ would yield a polynomial that is too big to compute (the minimal polynomial of $A(x,y)$ is a very small factor of this huge polynomial). Pantone still managed to eliminate the variables $A(x,1)$, $B(x,1)$, and $C(x,1)$ ``by hand" one at a time, using resultants. He found that the function $A(x,y)$ is also algebraic of degree 4, and its minimal polynomial is
\begin{align*}
& x^2 \big ((y-1)(x^2 y^3 + x y^2 + 3y) + x^2 y^2 - y^3 + 1 \big ) A(x,y)^4 \\
& - xy \big ((y-1)(2x^3 y^2 + 4x^2 y - x y^2 + 2x + 2y - 2) - 2 x^2 y^3 + 2 x^3 y \big ) A(x,y)^3 \\
& + y^2 \big ((y-1)(x^4 y - 2x^3 y + x^3 + 2 x^2 - 1) \\
& + (x-y)(- 2 x^2 - 2 xy + 3x) +  x^4 + x^2 y^2  \big ) A(x,y)^2 \\
&- y^3 \big (x^2 y + x^2 - 2x - y + 2 \big) A(x,y) + y^4.
\end{align*}

The statistic $\sites$ associated with the catalytic variable $y$ is essentially the same as the statistic $\mathbf{rank}$ studied in \cite{Huh_Kim_Seo_Shin_2025} (to be precise, $\sites(\sigma) = \mathbf{rank}(\sigma)+1$ for any $\sigma \in \I(000,102)$), so this answers the question of enumerating $\{000, 102\}$-avoiding inversion sequences according to their rank, that was left open in \cite{Huh_Kim_Seo_Shin_2025}.

\section*{Acknowledgements}
We thank Mathilde Bouvel for her support, and for giving a lot of feedback and suggestions on this work. We thank Jay Pantone for sharing interesting ideas about some constructions of inversion sequences, for using his software to provide us with conjectures, and for great help in proving one such conjecture in the appendix. We thank Mireille Bousquet-Mélou for some insight on power series equations. We thank the anonymous referee for some nice suggestions.

\emergencystretch=1em
\printbibliography
\end{document}